\documentclass[12pt,oneside]{amsart}
\usepackage{setspace}
\usepackage{amssymb} 
\usepackage{amsmath}
\usepackage{mathrsfs}
\usepackage{stmaryrd}

\usepackage{amsthm}
\usepackage{mathtools}
\usepackage{amsfonts}
\usepackage{graphicx}
\usepackage[colorinlistoftodos]{todonotes}
\usepackage{tikz-cd}   

\usepackage{enumitem}

\usepackage[all]{xy}
\usepackage{textcomp}
\usepackage{amsbsy}

\usepackage{datetime}
\renewcommand{\dateseparator}{-}
\renewcommand{\today}{\the\year \dateseparator \twodigit\month
\dateseparator \twodigit\day}

\usepackage
[pdfauthor={Max Menzies},
 pdftitle={The $p$-curvature conjecture for the non-abelian Gauss-Manin Connection},
 bookmarks=false]
{hyperref}

\title{The $p$-curvature conjecture for the non-abelian Gauss-Manin Connection} 
\author{Max Menzies}

\newtheorem{theorem}{Theorem}[section]

\newtheorem{thm-defn}[theorem]{Theorem/Definition}
\newtheorem{lemma}[theorem]{Lemma}
\newtheorem{prop}[theorem]{Proposition}
\newtheorem{corollary}[theorem]{Corollary}
\newtheorem{conjecture}[theorem]{Conjecture}

\theoremstyle{definition}
\newtheorem{definition}[theorem]{Definition}
\newtheorem{notation}[theorem]{Notation}

\theoremstyle{remark}
\newtheorem{remark}[theorem]{Remark}

\numberwithin{equation}{section}

\DeclareMathOperator{\Spec}{Spec}

\DeclareMathOperator{\End}{End}
\DeclareMathOperator{\Proj}{Proj}
\DeclareMathOperator{\Crys}{Crys}

\begin{document}
\pagenumbering{roman}

\doublespacing

% ----------------------------------------------------------------
\thispagestyle{plain}

\noindent Harvard University Ph.D Thesis \hfill Max Menzies

\vspace{0.3in}

\centerline{The $p$-curvature conjecture for the non-abelian Gauss-Manin connection}

\vspace{0.3in}

\centerline{Abstract}

\vspace{0.3in}

Originally conjectured unpublished by Grothendieck, then formulated precisely by Katz in \cite{Katz1}, the $p$-curvature conjecture is a local-global principle for algebraic differential equations. It is at present open, though various cases are known. In \cite{Katz2}, Katz proved this conjecture in a wide range of cases, for differential equations corresponding to the Gauss-Manin connection on algebraic de Rham cohomology.

This dissertation addresses the non-abelian analogue of Katz' theorem, in the sense of Simpson's non-abelian Hodge theory, surveyed in \cite{Simpson3}, and later developed in characteristic $p$ in \cite{Ogus-Vologodsky}. Specifically, there is a canonical non-abelian Gauss-Manin connection on $M_{dR}$, the stack of vector bundles with integrable connection, which is the appropriate definition of non-abelian de Rham cohomology. In this dissertation, I introduce this connection and its $p$-curvature; this requires the generalization of the \\ $p$-curvature conjecture due to Bost, Ekedahl and Shepherd-Barron. Then, I prove that the analogue of the main technical result of Katz' theorem holds for this connection.

\pagebreak

% ----------------------------------------------------------------

%%% Local Variables: 
%%% mode: latex
%%% TeX-master: "main"
%%% End: 

% ----------------------------------------------------------------

% ----------------------------------------------------------------
\thispagestyle{plain}

\tableofcontents
\thispagestyle{plain}

\pagebreak

% ----------------------------------------------------------------

%%% Local Variables: 
%%% mode: latex
%%% TeX-master: "main"
%%% End: 

% ----------------------------------------------------------------

% ----------------------------------------------------------------
\thispagestyle{plain}

\section*{Acknowledgements}

First, I must thank Professor Mark Kisin of Harvard University for his suggestion of this topic, help with all the background mathematics,  for his boundless wisdom and considerable patience. Secondly, I must thank Ananth Shankar, Yunqing Tang, Koji Shimizu, Chi-Yun Hsu, Arnav Tripathy, Thanos Papaioannou, Jacob Lurie and David Ben-Zvi for helpful discussions and explanations, and all number theory \\ students at Harvard, particularly Zijian Yao, for creating the richest possible learning environment. 

I'd like to thank Robin Gottlieb, Jameel Al-Aidroos, Janet Chen and Brendan Kelly for teaching me the privilege of teaching mathematics. 

On the administrative side, I must thank Susan Gilbert, who assisted me greatly in my choice to come to Harvard in the first place, for her kindness and \\ helpfulness navigating some difficult years of graduate school, and then her successor, Larissa Kennedy. 

I have immense gratitude to Dean Margot Gill and Rebecca Bauer Lock for the support of the Frank Knox Memorial Fellowship and to all Knox Fellows who enriched my experience here. 

Finally, my parents, sister and partner will always be the most important people to me.
\pagebreak

% ----------------------------------------------------------------

%%% Local Variables: 
%%% mode: latex
%%% TeX-master: "main"
%%% End: 

% ----------------------------------------------------------------

\pagenumbering{arabic}

\thispagestyle{plain}

\section{Introduction}
\subsection{Background to the problem}

Recall the classical notion of curvature. Let $(V,\nabla)$ be a vector bundle equipped with a connection on a smooth manifold. Given a derivation $D$ on the manifold, that is, a section of the tangent bundle, there arises a differential operator on V, notated $\nabla(D)$ or $\nabla_D$. The curvature of $\nabla$ measures the failure of $\nabla$ to respect the Lie bracket operation on derivations. It vanishes if and only if $\nabla$ preserves the Lie bracket, or equivalently, if $V$ is locally spanned by the solutions for the corresponding differential equation, $\nabla(v)=0$.

In characteristic $p$, it is a remarkable fact that the $p$-th power iterate $D^p$ is also a derivation, so one can form the $p$-curvature defined as $\psi_p (D) = \nabla(D)^p - \nabla(D^p)$. Given $(V, \nabla)$ in characteristic $p$, a theorem of Cartier (\cite{Katz3}, theorem 5.1) asserts that if $\nabla$ has vanishing curvature and $p$-curvature, then Zariski locally, $V$ is spanned by solutions of the differential equation. Hence the name, $p$-curvature. 

The $p$-curvature conjecture can then be understood as a local-global principle in which a vector bundle with integrable connection should have a full set of \textit{algebraic} solutions if and only if all the $p$-curvatures, that is, the local obstructions to a full set of solutions, vanish.

\begin{conjecture}[Original Grothendieck $p$-curvature conjecture]
Let $X$ be a \\ connected smooth scheme of finite type over $\mathbb{C}$, and $(V,\nabla)$ a vector bundle with \\ integrable connection on $X$. After a process of ``spreading out" and reducing modulo $p$, assume that almost all $p$-curvatures $\psi_p$ vanish. Then $(V,\nabla)$ becomes trivial after a finite \'etale pullback of $X$.
\end{conjecture}

Several equivalent forms of this conjecture have been introduced, and one notable generalization due to Bost-Ekedahl-Shepherd-Barron is as follows:

\begin{conjecture}[Bost-Ekedahl generalization]
Let $X/ \mathbb{Q}$ be an algebraic variety equipped with a subbundle $F \subset T_{X/ \mathbb{Q}}$. Assume $F$ is integrable, that is, closed under Lie bracket. By a classical theorem of Frobenius, this defines a foliation on $X(\mathbb{C})$. Further, assume that $F$ mod $p$ is closed under $p$-th powers for almost all $p$. Then the leaves of the foliation are algebraic subvarieties.  
\end{conjecture}

Bost's conjecture implies the $p$-curvature conjecture. Given a vector bundle $V$ with connection, set $X$ as the total space of $V$ and $F$ the horizontal subbundle corresponding to the connection. The vanishing of the curvature and $p$-curvature of the connection implies the required conditions on $F$. Again, the $p$-th power operation on derivations makes an appearance. This conjecture is stated in \cite{Bost} and proven there subject to additional assumptions.

 Katz proved the first conjecture for the so-called Picard-Fuchs equations. These are the differential equations corresponding to the vector bundle with integrable \\ connection that is algebraic de Rham cohomology equipped with the Gauss-Manin connection. 

To visualize the Gauss-Manin connection, consider a smooth proper map of smooth manifolds $X \to S$. This is a fibration, and so local nearby fibres $X_s$ are diffeomorphic. As such, it is possible to identify and parallel transport between their de Rham cohomology groups $H^i(X_s), s \in S$. This gives an integrable connection on algebraic de Rham cohomology over $S$.

\thispagestyle{plain}

\begin{theorem}[Katz, \cite{Katz2}, Theorem 5.1]
Let $X \xrightarrow{f}   S \xrightarrow{g} T$ be smooth morphisms of schemes. Suppose $f$ is proper and $g$ has geometrically connected fibres. Let \\ $\nabla_{GM}: R^n f_* \Omega^\bullet_{X/S} \to R^n f_* \Omega^\bullet_{X/S} \otimes \Omega^1_{S/T} $ be the Gauss-Manin connection on algebraic de Rham cohomology. Then the $p$-curvature conjecture holds for $( R^n f_* \Omega^\bullet_{X/S}, \nabla_{GM})$.
\end{theorem}

The conjecture that inspired my dissertation is the non-abelian analogue of Katz' theorem. The setting of this conjecture is Simpson's non-abelian de Rham \\ cohomology, equipped with an analogous Gauss-Manin connection, outlined in \cite{Simpson1}.

Because homotopy groups in higher dimensions are abelian, and cohomology \\ theories can be interpreted as spaces of maps into Eilenberg-Maclane spaces, non-abelian cohomology occurs essentially only in degree $1$. Most succinctly, one replaces regular de Rham cohomology with $H^1(X, GL_d)$. This is most naturally interpreted as a stack, although one could happily consider the associated coarse moduli space as well. Geometrically, the most intuitive way to understand this is that one replaces the cohomology groups $H^i(X_s), s \in S$ with the families of representations (of a fixed dimension $d$) of the fundamental groups $\pi_1(X_s)$. Roughly speaking, the dual of homology is replaced by the dual of $\pi_1$.

Having made this definition, $H^1(X, GL_d):=M_{dR}$, non-abelian de Rham \\ cohomology, has an integrable connection, in the sense of a horizontal integrable subbundle of its tangent bundle. The main conjecture of my dissertation is the statement that Bost's conjecture holds for $M_{dR}$ with this horizontal subbundle.
\thispagestyle{plain}

\begin{conjecture}[Bost-Ekedahl for $M_{dR}$]
Assume for almost all primes $p$, the \\ horizontal subbundle of the tangent bundle of $M_{dR}$ is stable under $p$-th powers modulo $p$. Then the corresponding foliation of $M_{dR}$ consists of algebraic leaves.
\end{conjecture}

This beautiful conjecture then states that for $H^1(X, GL_d)$, a natural local condition at almost all primes implies that the leaves of its natural foliation, a priori merely analytic subsets, are in fact algebraic. This applies both to the stack $M_{dR}$ and its coarse moduli space.

Katz' proof proceeds in two steps. In the first step (\cite{Katz2} theorem 3.2), working entirely in characteristic $p$, he relates the $p$-curvature of $\nabla_{GM}$ to the Kodaira-Spencer map, which is given by cup product with the Kodaira-Spencer class. Specifically, he equates the following maps, sources and targets:
\begin{center}

\begin{tikzcd}

   gr^\bullet_{con} R^n f_* \Omega^\bullet_{X/S}    \arrow["\psi"]{r}  &     gr^{\bullet +1}_{con} R^n f_* \Omega^\bullet_{X/S} \otimes F^* ( \Omega^1_{S/T})       \\
F^*(gr^\bullet_{Hod} R^n f_* \Omega^._{X/S}) \arrow["F^*(\rho)"]{r} & F^*(gr^{\bullet - 1}_{Hod} R^n f_* \Omega^\bullet_{X/S} \otimes \Omega^1_{S/T})

\end{tikzcd}  

\end{center}

 $F:X \to X$ is absolute Frobenius. The first map is induced by the $p$-curvature passing to the associated graded of the \textit{conjugate filtration}, while the Kodaira-Spencer map coincides with the Gauss-Manin connection induced on the associated graded of the \textit{Hodge filtration} (\cite{Katz2}, proposition 1.4.1.7). In particular, if the $p$-curvature \\ vanishes, then the Kodaira-Spencer map vanishes, and so the Gauss-Manin \\ connection extends over the Hodge filtration.
 
 Katz' second step proceeds entirely in characteristic $0$. If the Gauss-Manin \\ connection preserves the Hodge filtration on the nose, a stronger condition than Griffiths transversality, Katz shows the monodromy associated to the connection lies within a unitary group of automorphisms of polarized Hodge structures. The image of monodromy is always discrete, so if it is contained in a compact unitary group, it must be finite. From this it follows there exists a trivializing finite \'etale cover. 

The main theorem of my dissertation is the exact non-abelian analogue of Katz' first step in the setting of Simpson's non-abelian Hodge theory. In \cite{Simpson1}, Simpson introduces the non-abelian Gauss-Manin connection on $M_{dR}$, and in \cite{Simpson2}, a similar structure on the stack of Higgs bundles $M_{Dol}$. In the appropriate sense, this is the associated graded of $M_{dR}$ of its Hodge filtration, described in \cite{Simpson3}, and so is called the non-abelian Kodaira-Spencer map.
\thispagestyle{plain}

\begin{theorem}[Main Theorem]
Let $X/S$ be schemes in characteristic $p$. Assume both have lifts to $p^2$ and that relative Frobenius $F_X: X \to X'$ has a global lift to characteristic $p^2$ (for instance, $X/S$ is relatively affine). Then there is an \\ identification between the $p$-curvature of $M_{dR}(X/S)$ after passing to the associated graded of its conjugate filtation with the Frobenius twist of the Kodaira-Spencer map on $M_{Dol}$. Moreover, if the $p$-curvature vanishes, then the Kodaira-Spencer map \\ vanishes, so the Gauss-Manin connection extends over the Hodge filtration.
\end{theorem}

\begin{remark} 
In the abelian case, it is immediate from Katz' relation that if the \\ $p$-curvature vanishes, then the Kodaira-Spencer map vanishes, thus the Gauss-Manin connection extends over the Hodge filtration. This is not trivial at all in the non-abelian case, as filtrations behave quite differently. A flatness argument is required, which is much of the difficulty.
\end{remark}
\thispagestyle{plain}

\subsection{Contents of the dissertation}
I will spend sections 2-7 introducing all the necessary ingredients to state the main theorem. The main theorem concerns \\ algebraic stacks and groupoids of vector bundles with additional data, so I must begin with introducing all these concepts as preliminaries.

In section 2, I will give the necessary preliminaries on divided power structures, as they will appear throughout the dissertation. Most important will be the PD formal completion of the diagonal, $\hat{\mathcal{P}}$.

In section 3, I will introduce vector bundles with integrable connections and the ring of differential operators $\Lambda$. I will prove an equivalence between integrable \\ connections, modules over $\Lambda$, and isomorphisms between the two pullbacks to $\hat{\mathcal{P}}$.

In section 4, I will introduce the fundamental object of this dissertation, $M_{dR}$, or $H_{dR}^1(X,GL_d)$. Properly interpreted this is the stack of vector bundles with integrable connection (of a fixed dimension $d$). Informally, this is the family of representations of the $\pi_1$'s of the fibres $X_s$. I will reinterpret vector bundles with integrable connection as crystals, and use the crystalline property to explicitly construct the Gauss-Manin connection on $M_{dR}$.

In section 5, I will introduce algebraic stacks, groupoids, and the language of stratifications. In section 6, inspired by \cite{Simpson4}, I will go into detail on the particulars of the de Rham, Dolbeault and Hodge stacks, and define the non-abelian Kodaira-Spencer map and Hodge filtration on $M_{dR}$. In a precise sense, the Kodaira-Spencer map is then the associated graded of the Gauss-Manin connection, and if it vanishes, the Gauss-Manin connection extends over the Hodge filtration; this statement is the last theorem of this section.

In section 7, I present the bulk of important constructions that exist purely in characteristic $p$. These include the twisted Dolbeault groupoid, the important \\ $p$-curvature morphism, the canonical connection, and splittings of Cartier. Crucially, I introduce the conjugate groupoid and prove some of its necessary properties.

In section 8, I will prove the main theorem. There are three main tools. First, I extend the $p$-curvature groupoid morphism to the full conjugate groupoid. Secondly, I have two stratifications that are the precise non-abelian analogues of the two maps in Katz' theorem, namely the $p$-curvature after passing to the conjugate associated graded, and the Frobenius twist of Kodaira-Spencer. Thirdly, a deformation \\ argument allows me to prove some necessary flatness so that the vanishing of \\ $p$-curvature implies the vanishing of Kodaira-Spencer.

\thispagestyle{plain}

\subsection{Notation and conventions}
Throughout, $T$ will be a base scheme, with \\ $X\to S \to T$ smooth separated morphisms of schemes. Many results will not require both of these properties, but one or either of them will be frequently required. \\ $\Omega_{X/S}= \Omega^1_{X/S}$ will denote the module of K\"ahler differentials, and $\Omega^r_{X/S}$ will denote its $r$-th exterior power $\Lambda^r \Omega_{X/S}$.

All $\mathcal{O}_X$-modules will be assumed to be locally free of finite rank, and \\ interchangeably called vector bundles. All vector bundles under consideration will be implicitly of a fixed dimension $d$. 

I will frequently use local coordinates. By smoothness, there exists an open cover of $X$ consisting of $U$ \'etale over $\mathbb{A}^n$ for some $n$. If $x_1,\ldots,x_n$ are the coordinates of this copy of $\mathbb{A}^n$ then $\Omega_{X/S}$ has a local basis over $U$ given by $dx_1,\ldots,dx_n$.

If $X/S$ is a scheme in characteristic $p$, its absolute Frobenius morphism \\ $F:X \to X$ is given by the identity map on the underlying topological space of $X$, and the $p$-th power map on the structure sheaf. Next we define $X'$, the relative Frobenius twist with respect to $S$, by the fibre product square below. By the universal property of the fibre product we can also produce the relative Frobenius morphism $F_X: X \to X'$ as follows:

\begin{center}   
\begin{tikzcd}   
   X  \arrow[ddr] \arrow[drr,"F"]  \arrow[dr, "F_X"]   &    &                 \\
  &   X'      \arrow[d] \arrow[r,"\pi"]    &    X     \arrow[d]   \\ 
  &   S      \arrow[r,"F"]      &   S
\end{tikzcd}  
\end{center}
\thispagestyle{plain}

In a divided power algebra $(A,J,\gamma),$ I will usually suppress $\gamma$ from the notation and write $x^{[n]}$ for the divided powers $\gamma_n (x)$, $x \in J$. We denote $\Gamma$ as the left adjoint to the forgetful functor from divided power algebras to modules given by $(A,J)\mapsto J$. In simple terms, if a module $M$ is locally free of rank $r$ over $\mathcal{O}_X$ with generators $m_i$, then $\Gamma_X M$ is generated as a divided power algebra by $m_i$ over $\mathcal{O}_X$. I will use the notation $\mathcal{O}_X<m_1,\ldots,m_r>$ for the free divided power algebra.

\pagebreak

\section{Divided powers}
\thispagestyle{plain}

A \textit{divided power algebra} is informally a triple $(A,J,\gamma)$ where $A$ is a commutative ring, $J$ an ideal, and $\gamma$ a series of maps $\gamma_n: J \to J, n \geq 1$, that behave like \\ $x^{[n]}:=\gamma_n(x)=\frac{x^n}{n!}$. I will not write down all the axioms, as these are not illuminating and will not be used in my proofs. An informal introduction is given in \cite{Gros-LeStum} and a more comprehensive treatment is given in \cite{Berthelot-Ogus}. With the assumption of local freeness, these constructions are relatively easy to understand and the former treatment is more than sufficient. In this section, I will summarize the results I need. I will use the acronym PD from the French, and usually suppress $\gamma$ from my notation. In my notation, I will reserve $J$ for a PD ideal, while $I$ can be any ideal.

\begin{definition}
Given a PD algebra $(A,J)$, there is a decreasing filtration of ideals $J^{[n]}$ generated by $x_1^{[k_1]}\ldots x_m^{[k_m]}$ for $x_i \in J$ and $k_1+ \ldots k_m \geq n $.
\end{definition}

\begin{theorem}
The functor $A \mapsto I$ from PD $R$-algebras to $R$-modules has a left adjoint $M \mapsto \Gamma M$.
\end{theorem}
\begin{proof}
See \cite{Berthelot-Ogus} theorem 3.9.
\end{proof}

In fact, the proof given shows that $\Gamma M$ is naturally a graded algebra,

\[ \Gamma M = \bigoplus_{k=0}^\infty \Gamma_k M\]

Moreover, its divided power ideal is $J=\Gamma_{>0} M$, and $\Gamma_k M$ is generated by elements $\underline{s}^{[\underline{k}]}= \prod s_{\lambda}^{k_\lambda}$ for $s_{\lambda} \in M$, integers $k_\lambda \geq 0$ with $\sum k_\lambda =n$. If $M$ is (locally) free, $\Gamma_k M$ is freely generated by such products where $s_{\lambda} \in M$ are a (local) basis for $M$.

\begin{definition}
A closed immersion $X \xhookrightarrow{} Y$ of schemes is said to have a \textit{PD structure} if the structure sheaf $\mathcal{O}_Y$, together with the ideal of definition, has a PD structure.
\end{definition}

\begin{theorem}
The functor that forgets the PD structure on $X \xhookrightarrow{} Y$ has a left adjoint $X \xhookrightarrow{} P_X Y$.
\end{theorem}
\begin{proof}
See proposition 1.4.1 of \cite{Berthelot-French}, in the case $m=0$. This is called the \textit{PD envelope} of $X$ in $Y$.
\end{proof}

\thispagestyle{plain}

I will be mainly interested in the case of the diagonal embedding $X \xhookrightarrow{} X \times_S X$ for a smooth separated morphism $X/S$. In this case, adopt the following notation:

\begin{notation}
Denote the PD envelope of $X \xhookrightarrow{} X \times_S X$ by $P_X$ and its structure sheaf $\mathcal{P}_X$. Let $I$ be the ideal of the embedding of $X \xhookrightarrow{} X \times_S X$. Let $I^e$ be the extension of this ideal in $\mathcal{P}_X$. Let $J$ be the PD ideal of $\mathcal{P}$.
\end{notation}

To explicate the difference between these, assume $X=\mathbb{A}^n$ with coordinates $t_1,\ldots, t_n$. Then the structure sheaf of $X \times_S X$ is $\mathcal{O}_X [\tau_1, \ldots, \tau_n]$ where $\tau_i=1 \otimes t_i - t_i \otimes 1$ are the diagonal coordinates. Then $I \subset \mathcal{O}_X [\tau_1, \ldots, \tau_n] $ is simply given by $(\tau_1, \ldots, \tau_n)$.

$\mathcal{P}$ is the free PD algebra $\mathcal{O}_X <\tau_1, \ldots, \tau_n>$. Within this, $I^e$ is merely generated by $(\tau_1, \ldots, \tau_n)$ whereas $J$ is generated by all divided powers $(\tau_i^{[k]}), k \geq 1$.

Further, I will denote

\[ \mathcal{P}^n = \frac{\mathcal{P}}{J^{[n+1]}} \]
\[ \hat{\mathcal{P}} = \varprojlim\mathcal{P}^{n} \]

Finally, let $P_X^n Y$ or$P_X^n (Y)$  denote the closed subscheme defined by $\mathcal{P}^n$.

\subsection{Base change properties with divided powers}
In this section, I outline several important base change properties. I have adopted some of the arguments in \cite{Thanos}.

\begin{lemma}
Let $(A_1,I_1) \xrightarrow{f} (A_2,I_2)$ be a morphism of rings and ideals. This induces a morphism  $(\mathcal{P}_1, J_1, \gamma) \xrightarrow{g} (\mathcal{P}_2, J_2, \delta)$ between their PD envelopes. Then the following three ideals within $J_2$ are equal:

1. The extension ideal of $J_1$ in $\mathcal{P}_2$, that is, $J_1 \mathcal{P}_2$.

2. The ideal of $\mathcal{P}_2$ generated by elements $\delta_n (y), y \in I_1 A_2, n \geq 1$, that is, $\delta(I_1 A_2) \mathcal{P}_2$.

3. The ideal $\delta(I_1 \mathcal{P}_2) \mathcal{P}_2$.
\end{lemma}
\thispagestyle{plain}

\begin{proof}
To see that 1. is a subset of 2., note $J_1$ is generated by elements $\gamma_n(x), x \in I_1$. So the extension ideal is generated by elements $g(\gamma_m(x))=\delta_m (f(x))$, which all lie in 2. 

It is immediate that 2. is a subset of 3. 

To finish the proof, repeatedly use the properties of PD algebras. \\ For $a \in \mathcal{P}_2, x,y \in I_2$,

\[ \delta_m(x+y) = \sum_{0 \leq i \leq m} \delta_i(x) \delta_{m-i} (y)\]
\[ \delta_m (ax) =a^m \delta_m (x) \]

Specialize to the case where $x=f(z), y=f(w), z,w \in I_1$ and rewrite \\ $\delta_m (x)= g(\gamma_m(z)) \in J_1$. From the expansion rules above it follows $\delta(I_1 \mathcal{P}_2) \subset J_1 \mathcal{P}_2$. It follows 3. is a subset of 1.

\end{proof}

\begin{corollary}
Let $I_1^e = I_1 A_2$. Then there is a canonical isomorphism between PD envelopes as follows:

\[ \mathcal{P}_{A_2 /I_1^e} (I_2/I_1^e) \simeq \mathcal{P}_2/ \delta(I_1 \mathcal{P}_2) \mathcal{P}_2 \]

By the previous lemma, the left hand side contains a copy of $J_1 \mathcal{P}_2 + J_2^{[n]}$. The quotient by this ideal can be described in two different ways as an isomorphism of quotients of $\mathcal{P}_2$,

\[ \mathcal{P}_2 /J_2^{[n]} \otimes_{\mathcal{P}_1} \mathcal{P}_1/J_1 \xrightarrow{\sim} \mathcal{P}_{A_2 /I_1^e} (I_2/I_1^e)/(\delta(I_1 A_2) \mathcal{P}_2 + J_2^{[n]}) \]

\end{corollary}

\begin{proof}
The first isomorphism is tag 07HB of \cite{stacks-project}. Now focus on the second map. The domain is canonically isomorphic to the quotient of $\mathcal{P}_2$ by $J_1 \mathcal{P}_2 + J_2^{[n]}$. The codomain is canonically isomorphic to the quotient of $\mathcal{P}_2$ by $\delta(I_1 A_2) \mathcal{P}_2 + J_2^{[n]}$. By the equality of ideals in the lemma, we are done.
\end{proof}

\begin{prop}
Consider a commutative diagram of schemes where the horizontal arrows are closed immersions:

\begin{center}

\begin{tikzcd}

   X_0   \arrow[hookrightarrow]{r}  \arrow[d]&        X   \arrow[d]      \\
 S_0 \arrow[hookrightarrow]{r} & S

\end{tikzcd}  

\end{center}

\end{prop}
Then there is a natural isomorphism:

\[ P^n_{X_0} ( S_0 \times_S X) \xrightarrow{\sim}  P_{X_0}^n X \times_{P^n_{S_0} S} S_0 \]

of closed subschemes of $P_{X_0}^n X$.
\thispagestyle{plain} 
\begin{proof}
Work locally on $P_{X_0}^n X$; in doing so we may assume $X,S$ and hence $X_0,S_0$ are affine. Let $S=\Spec(A_1), X=\Spec(A_2)$ and let $I_1,I_2$ be the ideals of definition of $S_0$ and $X_0$ respectively. By corollary 2.7, it follows

\[ P^n_{X_0} ( S_0 \times_S X) \xrightarrow{\sim}  P_{X_0}^n X \times_{P_{S_0} S} S_0 \]

is an isomorphism. It remains to show the two right hand sides above are \\ canonically isomorphic, that is, $P_{X_0}^n X \times_{P^n_{S_0} S} S_0 \simeq P_{X_0}^n X \times_{P_{S_0} S} S_0$. We can do so by a consideration of the universal property of the fibre product. Given a pair of morphisms $U \to P_{X_0}^n $ and $U \to S_0$ such that the following diagram commutes:

\begin{center}

\begin{tikzcd}

   U  \arrow{r}  \arrow[d]&       P_{X_0}^n X    \arrow[d]      \\
 S_0 \arrow{r} & P^n_{S_0} S

\end{tikzcd}  

\end{center}
 we can simply compose with the natural immersion $P^n_{S_0} S \xhookrightarrow{} P_{S_0} S $ to yield a diagram of the following form:

\begin{center}

\begin{tikzcd}

   U  \arrow{r}  \arrow[d]&       P_{X_0}^n X    \arrow[d]      \\
 S_0 \arrow{r} & P_{S_0} S

\end{tikzcd}  

\end{center}
Conversely, given a pair of morphisms in the second diagram, note both $S_0 \to P_{S_0} S$ and $P_{X_0}^n X $ factor through the closed immersion $P^n_{S_0} S \xhookrightarrow{} P_{S_0} S $ so we can recover a diagram of the first type. Since this is a closed immersion, there is a natural bijection between the two diagrams, hence the fibre products are naturally isomorphic.
\end{proof}

\begin{corollary}
Let $X \to S \to T$ be separated morphisms. Then apply proposition 2.8 to the following diagram:

\begin{center}

\begin{tikzcd}

   X  \arrow[hookrightarrow]{r}  \arrow[d]&        X \times_T X  \arrow[d]      \\
 S \arrow[hookrightarrow]{r} & S \times_T S

\end{tikzcd}  

\end{center}

Deduce there is a natural isomorphism,

\[P^n_X (X \times_S X) \to P_X^n (X \times_T X) \times_{P_S^n (S \times_T S)} S \]

of closed subschemes of $P^n_X (X \times_T X)$.
\end{corollary}
\begin{proof}
All that is needed is that $ S \times_{S \times_T S} (X \times_T X)$ is naturally isomorphic to $X \times_S X$, which one can show using a similar argument and the universal property of a Cartesian square.
\end{proof}
\thispagestyle{plain} 

\begin{prop}
Let $X \xhookrightarrow{} Y$ be a closed immersion of schemes, both smooth over $S$. Let $S' \to S$ be an arbitrary base change, and let $X'=X\times_S S', Y'=Y \times_S S'$. Then $P^n_X(Y)$ is flat over $S$ and satisfies arbitrary base change:
\[P^n_{X'} Y' \simeq P^n_{X} Y \times_S S'  \]
\end{prop}
\begin{proof}
We can work locally on $P^n_X(Y)$, hence on $Y,X$ and $S$. See the proof of \\ proposition 1.5.3 in \cite{Berthelot-French}. After taking $m=0$ there, the hypothesis in that proposition that $R$ is a $\mathbb{Z}_{(p)}$-algebra is not necessary for the proof, and the given proof works for this proposition.
\end{proof}

\thispagestyle{plain}
\subsection{The PD envelope of the diagonal}
Throughout this section, $X/S$ will be a smooth separated morphism of schemes. $\mathcal{P}$ will denote the structure sheaf \\ $\mathcal{P}_X(X \times_S X)$. Let $x_1,\ldots, x_n$ be local coordinates, namely sections of $\mathcal{O}_X$ that define an \'etale map $U \to \mathbb{A}^n_S$ for an open set $U \subset X$. Let $\tau_i=1 \otimes t_i - t_i \otimes 1$. I will always understand $\mathcal{P}$ as an $\mathcal{O}_X$-module by the left projection. The following proposition underlines my calculations with $\mathcal{P}$:

\begin{prop}

1. If $X$ is of characteristic $m$ for some $m>0$, then we have a simple local description for $\mathcal{P}$ as 

\[ \mathcal{P} = \mathcal{O}_X <\tau_1, \ldots, \tau_n> \]

2. Regardless of characteristic, we have a simple local description for $\hat{\mathcal{P}}$ as the free PD power series algebra

\[ \hat{\mathcal{P}} = \mathcal{O}_X <<\tau_1, \ldots, \tau_n>> \]

\end{prop}

\begin{remark}
These results are both striking. They do not naively hold for the plain diagonal structure sheaf $\mathcal{O}_X \otimes_{\mathcal{O}_S} \mathcal{O}_X$. If $S=\Spec k$ and $X=\mathbb{A}^1 - {0}$, then the diagonal structure sheaf is $k[x,y,1/x,1/y]$, which cannot be expressed as $k[x,1/x][x-y]$.
\end{remark}

\begin{proof}
1. For a calculation that is not so illuminating, consult \cite{Berthelot-Ogus}, 3.3.2. A more conceptual proof sketch proceeds as follows: let $U'=\mathbb{A}^n_S$. We have an \'etale map $U \to U'$ and so can form a sequence of maps

\[ U \xhookrightarrow{} U \times_{U'} U  \xhookrightarrow{} U \times_S U  \]

The first map is both closed and open while the second is closed. That is, \\ $U \xhookrightarrow{} U \times_{U'} U$ is a connected component. As such, Zariski locally, it is cut out by an ideal generated by idempotents, by \cite{stacks-project}, tag 04PP.

Now if $J$ is an ideal generated by nilpotents with a PD structure, then for any idempotent $e \in J$,

\[e=e^m = m! \gamma_m (e) = 0 \]

and so the ideal $J$ must be $0$. Working Zariski locally, it follows $P_U( U \times_{U'} U) = U$.

It follows then that $P_U (U \times_S U) = P_{U \times_{U'} U} (U \times_S U)$. Everything is smooth over $S$, so these envelopes behave well under the base change $U \times_{S} U \to U' \times_S U'$. 

That is, 

\[P_{U'}( U' \times_{S} U') \times_{U' \times_S U'} U \times_{S} U \simeq P_{U \times_{U'} U}( U \times_{S} U ) \simeq P_U ( U \times_{S} U ) \]

The right hand side's structure sheaf is a Zariski localization of $\mathcal{P}$, and the left hand side is easy to describe. Since $U'=\mathbb{A}^n_S$ it can easily be seen that it is locally

$\mathcal{O}_{U'} <\tau_1,\ldots,\tau_n> \otimes_{U' \times_S U'} U \times_{S} U = \mathcal{O}_U <\tau_1,\ldots,\tau_n>$ 

This concludes the first statement. \\

\thispagestyle{plain}

2. Work Zariski locally. $\mathcal{O}_X \otimes_{\mathcal{O}_S} \mathcal{O}_X$ has a filtration given by powers of the ideal of the diagonal $I$. It is standard that $I/I^2$ is canonically isomorphic to $\Omega_{X/S}$, that this is free with generators $dt_i$, and that there is a simple description of the associated graded as the symmetric algebra of $\Omega$.

\[ Gr_{I^\bullet} (\mathcal{O}_X \otimes_{\mathcal{O}_S} \mathcal{O}_X) = \bigoplus I^n/I^{n+1} = S( \Omega_{X/S}) \]

Now consider instead the filtration of $\mathcal{P}$ given by divided powers $J^{[n]}$ of the PD ideal. Form the associated graded as follows:

\[ Gr_{J^\bullet} \mathcal{P} = \bigoplus J^{[n]}/J^{[n+1]} :=\Gamma \]

It is routine to check that this has a natural PD structure on the ideal $\Gamma_+$, induced from the PD structure on $\mathcal{P}$, and must therefore be the PD envelope of the \\ previous associated graded algebra with its ideal $S^+ \Omega_{X/S}$. That is, the PD envelope of $(S \Omega_{X/S}, S^+ \Omega_{X/S})$ is precisely $(\Gamma, \Gamma^+)$. But the PD envelope of $(S\Omega_{X/S}, S^+ \Omega_{X/S})$ is $\Gamma \Omega_{X/S}$ by the local freeness of $\Omega_{X/S}$.

\thispagestyle{plain}

Thus, this proves an equality, which respects grading, between $\Gamma \Omega_{X/S}$ and the $\Gamma$ defined above. In particular,

\[ J^{[n]}/J^{[n+1]} = \Gamma^n\Omega_{X/S}, n \geq 0   \]

The right hand side is freely generated by elements $\underline{dx}^{[\underline{k}]}$ with $\sum k_\lambda = n$. Moreover, these elements coincide with the images of elements $\underline{\tau}^{[\underline{k}]}$ where $\tau_i$ are the diagonal elements. 

Therefore, $J^{[n]}/J^{[n+1]}$ is freely generated by PD monomials of degree $n$ in the $\tau_i$. Now induct using the exact sequence,

\[ 0 \to J^{[n]}/J^{[n+1]} \to \mathcal{P}^n \to \mathcal{P}^{n-1} \to 0 \]

to deduce that $\mathcal{P}^n$ is freely generated by PD monomials $\underline{\tau}^{[\underline{k}]}$ of degree at most $n$. Take the inverse limit to get the result.

\end{proof}

\begin{remark}
There is a choice here of left or right module structure.  $\mathcal{P}$ does not have a canonical module structure over $\mathcal{O}_X$. I am making a choice always to consider the left module structure. However, $\mathcal{P}$ does have a canonical coring structure. I will now introduce this.
\end{remark}

I will use the following notation in subsequent propositions: we denote local \\ coordinates $z_1,\ldots,z_n$. For notational simplicity, assume $n=1$. That is, assume $X=\mathbb{A}^1$ with coordinate $z$. The left map in the tensor product below, the second projection, I will denote by $2.$ and consider as $z\mapsto y$; the right map, the first projection, I will denote by $1.$ and consider as $z \mapsto x$. Let $\tau= 1 \otimes z - z \otimes 1=y-x$.

\begin{prop}[The comultiplication map]
There is a comultiplication map,
\[\Delta: \mathcal{P}\to \mathcal{P} \otimes_{2,\mathcal{O}_X,1} \mathcal{P}\]

For local sections $a,b$ of $\mathcal{O}_X$, and $\tau$ as above, 
\[a \otimes b \mapsto a\otimes 1 \otimes 1 \otimes b\]
\[\Delta(\tau^{[n]})=  \sum_{i+j=n} \tau^{[i]}\otimes \tau^{[j]} \]
\end{prop}
\thispagestyle{plain}

\begin{proof}
Let $Q=\mathcal{O}_X \otimes_{\mathcal{O}_S} \mathcal{O}_X$. Define 

\[\Delta: Q\to Q \otimes_{2,\mathcal{O}_X,1} Q\]
\[a \otimes b \mapsto a\otimes 1 \otimes 1 \otimes b\]
Work with local coordinates, and for notational simplicity assume $n=1$, that is, $X=\mathbb{A}^1_S$. $\Delta$ takes the form:

\[\Delta: \mathcal{O}_S[x,y]\to \mathcal{O}_S[x,y]\otimes_{2,\mathcal{O}_S[z],1} \mathcal{O}_S[x,y] \]
\[\Delta(x)=x\otimes 1, \Delta(y)=1 \otimes y\]
Here the left tensor map is given by $z\mapsto y$, the right is given by $z \mapsto x$. To exhibit a quick calculation this means,
\[\Delta(y-x)=1 \otimes y - x \otimes 1 = 1\otimes y -1 \otimes x + y \otimes 1 - x \otimes 1=(y-x)\otimes 1 + 1 \otimes (y-x)\]

That is, $\Delta(\tau)=\tau \otimes 1 + 1 \otimes \tau$. Passing to divided power envelopes, this induces a map that we also call $\Delta$:

\[\Delta: \mathcal{P}\to \mathcal{P} \otimes_{2,\mathcal{O}_X,1} \mathcal{P}\]
\[\Delta: a \otimes b \mapsto a\otimes 1 \otimes 1 \otimes b\]
\[\Delta(\tau^{[n]})=  \sum_{i+j=n} \tau^{[i]}\otimes \tau^{[j]} \]

This is clearly co-associative.
\end{proof}

\thispagestyle{plain}
\begin{remark}
$\Gamma M$ also has a comultiplication structure induced by \\ $\Delta(m)= 1 \otimes m + m \otimes 1$ for $m \in M$. 
\end{remark}
\pagebreak

\section{Vector bundles with integrable connection}
In this section, I introduce modules, interchangeably vector bundles, with \\ integrable connection. I will introduce the ring $\Lambda$ of crystalline differential operators, show $\Lambda$ and $\mathcal{P}$ are dual in a suitable sense, and show equivalences of categories between vector bundles with integrable connection, modules over $\Lambda$, and stratifications over $\hat{\mathcal{P}}$. 

Recall $X/S$ is a smooth separated morphism of schemes. Several of these results can be stated under the lesser assumption that $X/S$ is locally of finite type, but for clarity, I will keep my assumptions consistent from section to section. All vector bundles will be implicitly assumed to be of a fixed rank $d$, which will not matter.

\thispagestyle{plain}

\subsection{Preliminaries from differential geometry}
\begin{definition}
By a \textit{vector bundle} or \textit{module} on $X$, we shall always mean a locally free $\mathcal{O}_X$-module of rank $d$. A connection relative to $X/S$, $\nabla: E \to E \otimes \Omega_{X/S}$, is an $\mathcal{O}_S$-linear operator satisfying a Leibniz condition: $\nabla(fe)=f\nabla(e) + df \otimes e$, for any sections $f$ of $\mathcal{O}_X$ and $e$ of $E$.
\end{definition}

Let $\Omega^r=\Omega^r_{X/S}$ and $\mathcal{A}^r= \Omega^r\otimes E$, understood as sheaves, with all tensor products taken over $X$. I will abuse notation and write $\sigma \in E$ to denote that $\sigma$ is a local section of $E$.

\begin{lemma}

If $(E,\nabla)$ is a vector bundle with connection, there exists a well-defined covariant derivative,

$d^E: \mathcal{A}^r \to \mathcal{A}^{r+1}$ satisfying, if $\omega \in \Omega^r, \sigma \in E$,
\[d^E(\omega \otimes \sigma)=d \omega \otimes \sigma + (-1)^r \omega \wedge \nabla \sigma\]

For $r=0$ this coincides with $\nabla$. It has a Leibniz rule: if $\xi \in \mathcal{A}^r, \omega \in \Omega^q$ then,

\[d^E(\omega \wedge \xi)= d \omega \wedge \xi + (-1)^q w \wedge d^E \xi  \]

\end{lemma}
\begin{proof}
Routine calculations with local coordinates.
\end{proof}

\begin{definition}
The connection $\nabla$ is said to be integrable if the composition \\ $(d^E)^2: E \to \mathcal{A}^{2}$ vanishes.
\end{definition}

Henceforth, $(E,\nabla)$ denotes a vector bundle with integrable connection.

\begin{lemma}
If $\mu \in \mathcal{A}^1$ and $A,B$ are vector fields, that is, sections of $T_X$, then:
\[ d^E(\mu)(A,B)=\nabla_A(\mu(B)-\nabla_B(\mu(A))- \mu ([A,B])  \]
\end{lemma}
\begin{corollary}
\[ 0=(d^E)^2(\sigma)(A,B) = \nabla_A \nabla_B \sigma - \nabla_B \nabla_A \sigma - \nabla_{[A,B]}\sigma \]
\end{corollary}
\begin{proof}
Two routine calculations, and using the integrability condition.
\end{proof}
\thispagestyle{plain}

\subsection{The ring of differential operators $\Lambda$}
Let $\Lambda$ be the ring of crystalline \\ differential operators on $X$. Not to be confused with the larger ring of Grothendieck differential operators, this is the universal enveloping algebra of the Lie algebra $T_X$. That is, $\Lambda$ is a sheaf of rings on $X$ that assigns to each affine open $U$ in $X$, the algebra generated by sections of $T_X$ and $\mathcal{O}_X$ over $U$ with relations:

\[ AB - BA =[A,B] \]
\[Af - fA=A(f) \]
for $A,B$ sections of $T_X$, and $f$ a section of $\mathcal{O}_X$. 

Locally, given coordinates $z_1,\ldots,z_n$, $T_X$ is freely generated as a left $\mathcal{O}_X$-module over basis elements $\partial_i=\frac{\partial}{\partial z_i}$. $\Lambda$ is freely generated as a left $\mathcal{O}_X$-module over monomials $\partial_1^{k_1}\ldots\partial_n^{k_n}$. For notational simplicity I will often assume $n=1$. Its algebra structure is determined by the relations that:

generators $\partial_i$ commute, $\partial^k \partial^l=\partial^{k+l}$, and:

\[\partial^n f = \sum_{0\leq k \leq n} {n \choose k} \partial^{n-k} (f) \partial^k \]
\thispagestyle{plain}

We have two canonical filtrations,

\[\mathcal{P} \supset J \supset J^{[2]} \supset \ldots \]
\[\Lambda_0 \subset \Lambda_1 \subset \Lambda_2 \ldots\]

both compatible with the multiplication on $\mathcal{P}$ and $\Lambda$. $J^{[n]}$ are the divided power ideals and $\Lambda_n$ are the differential operators of order $\leq n$. 

That is, $J^{[n]}J^{[m]} \supset J^{[n+m]}$ and $\Lambda_n \Lambda_m \subset \Lambda_{n+m}$.

Recall notation from proposition 2.14: we denote local coordinates $z_1,\ldots,z_n$, \\ assuming $n=1$ for notational simplicity. I will consider the second projection as $z\mapsto y$ and the first projection as $z \mapsto x$. Let $\tau= 1 \otimes z - z \otimes 1=y-x$.

\begin{prop}
Upon the left choice of $\mathcal{O}_X$-action for both $\mathcal{P}$ and $\Lambda$, there is a canonical pairing at finite levels $\mathcal{P}^n \otimes \Lambda_n \to \mathcal{O}_X$. This induces a perfect pairing,

\[ \varprojlim \mathcal{P}^n \times \varinjlim \Lambda_n \to \mathcal{O}_X  \]

that is, a perfect pairing between $\hat{\mathcal{P}}$ and $\Lambda$. This duality can be used to directly define the multiplication structure on $\Lambda$ as dual to the comultiplication structure $\Delta$. Moreover, having defined this duality as left modules, the standard action of $\Lambda$ as operators on $\mathcal{O}_X$ can be obtained from the second projection.
\end{prop}

\begin{proof}
For now, forget the multiplication structure on $\Lambda$. Work locally, for notational simplicity assume $n=1$, and regard $\Lambda$ as a left $\mathcal{O}_X$-module over basis elements $\partial^k$. Working locally on the same open set, $\mathcal{P}$ becomes a left $\mathcal{O}_X$-module over basis elements $\tau^{[k]}$.

Define a pairing between $\mathcal{P}$ and $\Lambda$ such that $\tau^{[k]}$ and $\partial^k$ are dual basis elements. Since $\mathcal{P}^n$ is freely generated over $\tau^{[k]}$ for $k \leq n$, this induces a perfect pairing between $\mathcal{P}^n$ and $\Lambda_n$.

I will first show the last claim of the proposition. Let $\lambda \in \Lambda$ be arbitrary. $\lambda$ must belong to some $\Lambda_n$. Therefore, we can interpret $\lambda$ as a module map \\ $\mathcal{P}^n \to \mathcal{O}_X$ considering the left module structures to both. Now compose with the second projection:
\[\mathcal{O}_X \xrightarrow{2} \mathcal{P}^n \xrightarrow{\lambda} \mathcal{O}_X \]

We claim this composition is $f\mapsto \lambda(f)$.

To see this, if $f=z^m$, then $f$ maps to $(z+\tau)^m=\sum_{0\leq k \leq m}{m \choose k} k! z^{m-k} \tau^{[k]} $ and then to ${m \choose k} k! z^{m-k}=\partial^k(z^m)$; this is the same as the standard action. This shows the choice of pairing did not depend on coordinates.
\thispagestyle{plain}

Next, we move on to carefully considering $\Delta:\mathcal{P}\to \mathcal{P} \otimes_{2,\mathcal{O}_X,1} \mathcal{P}$. A priori, this is merely $\mathcal{O}_S$-linear. Give the tensor product a left $\mathcal{O}_X$-structure by left multiplication on the left factor. That is, for $f \in \mathcal{O}_X, g,h \in \mathcal{P}$, act $f$ on $g \otimes h$ by $fg \otimes h$. In addition, give $\mathcal{P}$ the left module structure. Then $\Delta$ becomes a $\mathcal{O}_X$-linear map. Now dualize this. We get a map $\Theta: \Lambda \otimes_{2,\mathcal{O}_X,1} \Lambda \to \Lambda$ that is $\mathcal{O}_X$-linear when we give $\Lambda$ the left $\mathcal{O}_X$-module structure and the tensor product a left $\mathcal{O}_X$-module structure on the first factor. This induces an associative operation because $\Delta$ is co-associative. We show $\Theta$ induces the usual multiplication on $\Lambda$. 
\thispagestyle{plain}

Dualizing the identity $\Delta(\tau^{[n]})=  \sum_{i+j=n} \tau^{[i]}\otimes \tau^{[j]}$, we see that $\Theta(\partial^k,\partial^l)=\partial^{k+l}$. Since $\Theta$ is left-linear, $\Theta(f,\partial^n)=f\partial^n$. Finally, we need to show $\Theta(\partial^n, f)$ is what it's meant to be, that is, $\Theta(\partial^n, f)=\sum_{0\leq k \leq n} {n \choose k} \partial^{n-k} (f) \partial^k $.

By $S$-linearity, let $f=z^m$. Suffices to show that under the composition,

\[ \mathcal{P}\xrightarrow{\Delta} \mathcal{P} \otimes_{2,\mathcal{O}_X,1} \mathcal{P} \xrightarrow{(\partial^n,f)} \mathcal{O}_X \]

that $\tau^{[k]}\mapsto {n \choose k}\partial^{n-k}(f)$. The key here is that the $f$ acts on the second factor of the tensor product in its standard way, but $\Theta$ is not linear with respect to this action.
\thispagestyle{plain}

Under the composition of $\Delta$ and $(id,f)$, $\tau^{[k]}$ maps to:

\begin{align}
\sum_{i+j=k} \tau^{[i]}\otimes f(x) \tau^{[j]} &= \sum_{i+j=k} f(y) \tau^{[i]}\otimes \tau^{[j]} \\
&= \sum_{i+j=k} (x+\tau)^m \tau^{[i]}\otimes \tau^{[j]} \\ &=\sum_{i+j=k, 0 \leq h \leq m} \tau^{[i]}x^{m-h} \tau^{[h]}\frac{m!}{(m-h)!} \otimes \tau^{[j]} \\
&= \sum_{i+j=k, 0 \leq h \leq m} \tau^{[i+h]}{h+i \choose h} \frac{m!}{(m-h)!}x^{m-h}\otimes \tau^{[j]}
\end{align}
Then applying $(\partial^n, 1)$ to this kills all terms except $j=0,i=k,i+h=n$, so $h=n-k$. So applying $(\partial^n, 1)$ gives

\[{n \choose k} \frac{m!}{(m-n+k)!}x^{n-k}= {n \choose k}\partial^{n-k}(x^m)\]

as desired. To summarize, we gave $\mathcal{P}$ a left $\mathcal{O}_X$ structure, then defined $\Lambda$ to be its dual, or more precisely, the direct limit of the duals at finite levels. We defined $\Delta$, the comultiplication structure, then induced a multiplication on $\Lambda$, and showed it was the correct one. 

Finally, this pairing can be shown to be canonical by an alternative interpretation. Given a differential operator $D$ and a function $g(x,y)$ defined up to $I^{[n+1]}$, define their pairing by $Dg(x,x)$, differentiating with the first coordinate only.
\end{proof}

\begin{remark}
It is possible to define $\Lambda$ directly by dualizing $\mathcal{P}$ as above. Then, one can define another ring $\Lambda_{abstract}$ as the free non-commutative $\mathcal{O}_X$-algebra with generators sections of $T_X$ and the relations:

\[AB-BA=[A,B]; A f = A(f) + f A\]
Since $\Lambda$, defined as the dual of $\mathcal{P}$, satisfies these relations, there exists a map $\Lambda_{abstract} \to \Lambda$. Working in local coordinates, one shows this is a bijection.
\end{remark}

\thispagestyle{plain}

\subsection{The equivalence of three categories}

\begin{prop}
An integrable connection on $E$ is equivalent to a left-module \\ structure over $\Lambda$.
\end{prop}
\begin{proof}
To give an integrable connection on $E$ is equivalent to the conditions:

\[\nabla_A \nabla_B \sigma - \nabla_B \nabla_A \sigma = \nabla_{[A,B]}\]
\[ \nabla_A(f \sigma)=A(f) e + f\nabla_A(\sigma) \]

for all $A,B \in T_{X}$.

Thus, an integrable connection on $E$ is equivalent to a left action on $E$ under the free non-commutative $\mathcal{O}_X$-algebra with generators sections of $T_X$ and the relations:

\[AB-BA=[A,B]; A f = A(f) + f A\]

This is precisely $\Lambda$.
\end{proof}

\begin{prop}
The data of an integrable connection on a vector bundle $E$ is equivalent to a stratification over $\hat{\mathcal{P}}$. That is, regard $\hat{\mathcal{P}}$ as a formal scheme mapping to $X$ via two projections $pr_i: \hat{\mathcal{P}} \to X$. A stratification is an isomorphism \\ $\phi: pr_1^* E \simeq pr_2^* E$ over $(X  \times_S X)^\wedge_{PD}$ satisfying natural cocycle conditions.

\thispagestyle{plain}

These cocycle conditions are as follows: we have three projections \\ $p_{ij}: X \times_S X \times_S X \to X  \times_S X$ defined in the obvious manner. These extend to the PD completions $p_{ij}: (X \times_S X \times_S X)^\wedge_{PD}\to (X  \times_S X)^\wedge_{PD} $.

The two cocycle conditions are that $\phi$ restricts to the identity on the diagonal, and that $(p_{23}^* 
\phi)  (p_{12}^* \phi) = p_{13}^* \phi$.
\end{prop}

\begin{proof}
The data of an integrable connection is equivalent to compatible actions \\ $\Lambda_n \times E \to E$ for all $n$. Such an action takes the form $(\partial, e)\mapsto \partial(e)$. Considering $\Lambda$ as a left $\mathcal{O}_X$ module, this action is $\mathcal{O}_X$-linear for $\Lambda_n$ and for the target $E$. Hence, this induces a map $E \to \text{Hom}_{\mathcal{O}_X}(\Lambda_n , E), e \mapsto (\partial \mapsto \partial(e))$. This is not a priori $\mathcal{O}_X$-linear. However, it becomes $\mathcal{O}_X$-linear if we consider the right module action of $\mathcal{O}_X$ on $\Lambda$. 

Using duality, we have a map $E \to \mathcal{P}^n \otimes E = (pr^{(n)}_2)_* (pr^{(n)}_1)^* E$. Note that from $E$ we had to first extend by the left module action, the first projection, and then restrict by the right action, the second projection, hence the term $(pr^{(n)}_2)_* (pr^{(n)}_1)^* E$. 

Thus, an integrable connection is equivalent to compatible maps \\ $(pr^{(n)}_1)^* E \to (pr^{(n)}_2)^* E$. Taking $n \to \infty$ and using compatibility, we get the map $(pr_1)^* E \to (pr_2)^*E$. The associativity of the action $\Lambda_n \times \Lambda_m \times E \to E$ gives the cocycle conditions. The cocycle conditions force the map $(pr_1)^* E \to (pr_2)^*E$ to be an isomorphism.
\end{proof}

\subsection{The horizontal subbundle}
In this section, I give one more equivalent \\ formulation of an integrable connection, which will be essential in the main theorem. In this section, let $V$ be a vector bundle on $X$, and $W$ the associated total space of $V$. $W$ can be constructed as $\Spec_X ( S(V^\vee))$  where $V^\vee$ is the dual vector bundle. There is a natural morphism $\pi: W \to X$. A section $\sigma \in \Gamma(U,V)$ of $V$ over $U \subset X$ in the sheaf theoretic sense can be reinterpreted as a section $\sigma: U \to V$ of $\pi$ in the sense that $\pi \sigma=id_U$. 
\thispagestyle{plain}

The sequence of smooth morphisms $W \to X \to S$ induces the fundamental exact sequence of sheaves over $W$,

\[0 \to T_{W/X} \to T_{W/S} \to \pi^* T_{X/S} \to 0\]

$T_{W/X}$ is called the vertical subbundle and is naturally isomorphic to $\pi^* V$.

\begin{prop}
A connection $\nabla$ on $V$ relative to $X/S$ is equivalent to a splitting of this exact sequence. We call the resulting subbundle, given by the image of $\pi^* T_{X/S}$, the horizontal subbundle.
\end{prop}

\begin{proof}
We show connections and splittings are both torsors over the same group, and prove they correspond one-to-one locally. From the Leibniz rule, it follows connections form a torsor over $\Omega_{X/S} \otimes_{\mathcal{O}_X} \End(V)$. That is, given $\nabla_1,\nabla_2$, the difference $\nabla_1 - \nabla_2$ is an $\mathcal{O}_X$-linear morphism $V \to V\otimes \Omega_{X/S}$.
\thispagestyle{plain}

Now consider two splittings $L_i: T_{W/S} \to T_{W/X} $. The difference is a morphism $L_1 - L_2:T_{W/S} \to T_{W/X} $ that kills the vertical subbundle $T_{W/X}\subset T_{W/S}$. Hence this difference is equivalent to a morphism $\pi^* T_{X/S} \to \pi^* V$. This is equivalent to a linear morphism $\pi^* V^\vee \to \pi^* \Omega_{X/S}$, or a linear section of $\pi^*V \otimes_{\mathcal{O}_W} \pi^* \Omega_{X/S}$. But

\[ \pi^*V \otimes \pi^* \Omega_{X/S} = (V \otimes_{\mathcal{O}_X} \mathcal{O}_W) \otimes_{\mathcal{O}_W} (\mathcal{O}_W  \otimes_{\mathcal{O}_X}\Omega_{X/S}) \]
\[= (V \otimes_{\mathcal{O}_X} \mathcal{O}_W)\otimes_{\mathcal{O}_X} \Omega_{X/S}  \]
\[ = \pi^* V \otimes_{\mathcal{O}_X} \Omega_{X/S} \]
\[ = V \otimes_{\mathcal{O}_X} S( V^\vee) \otimes_{\mathcal{O}_X}\Omega_{X/S} \]

Finally, the linearity of $L_1 - L_2$ shows this section must lie in the degree $1$ component of $S( V^\vee)$, that is, $V^\vee$. So $L_1 - L_2$ is equivalent to a section of \\ $V \otimes_{\mathcal{O}_X}  V^\vee \otimes_{\mathcal{O}_X}\Omega_{X/S}=\Omega_{X/S} \otimes_{\mathcal{O}_X} \End(V)$.  It follows connections and splittings are both torsors over $\Omega_{X/S} \otimes_{\mathcal{O}_X} \End(V)$.

\thispagestyle{plain}

To show the one-to-one correspondence, it suffices to prove there is a natural \\ correspondence locally. Given a local frame of $V$, $e_1,\ldots,e_r$ there is a trivial \\ connection locally given by $\nabla_0 e_i=0$ for all $i$. Given this local frame, we have local coordinates on $W$ given by $e_1,\ldots,e_r, x_1,\ldots,x_n$, and the tangent bundle splits naturally into the span of $\frac{\partial}{\partial x_i}$ and $\frac{\partial}{\partial e_i}$. That is, there is a corresponding trivial splitting locally defined by $L_0: \frac{\partial}{\partial x_i} \mapsto 0$. Hence locally, both splittings and connections are non-trivial torsors over the same group, and so form a one-to-one correspondence globally. An explicit way to recover a connection from a splitting is described as follows:

\thispagestyle{plain}

Take a connection $\nabla$. Locally take a frame $e_1,\ldots, e_r$ and the corresponding trivial connection $\nabla_0$. Then $\nabla-\nabla_0$ is the corresponding section of $\Omega_{X/S} \otimes_{\mathcal{O}_X} \End(V)$. In local coordinates, we can write,

\[ \nabla e_i = \sum \theta_{ij} \otimes e_j\]

where $\theta_{ij}$ are $1$-forms. It follows $\nabla - \nabla_0$ also has the same formula:

\[ (\nabla- \nabla_0) e_i = \sum \theta_{ij} \otimes e_j\]

hence,

\[ \nabla- \nabla_0 = \sum \theta_{ij} \otimes e_j \otimes e_i^\vee = \sum f_{ijk} dx_k\otimes e_j \otimes e_i^\vee \]

And so by working through the equivalence between $\nabla- \nabla_0$ and $L-L_0$ described above,

\[L-L_0: \pi^* T_{X/S} \to \pi^* V \]
\[ \frac{\partial}{\partial x_k} \mapsto \sum f_{ijk} e_j \otimes e_i^\vee = (\nabla- \nabla_0)_{\frac{\partial}{\partial x_k}}\]

Now we want to get rid of $L_0$. Using the fact that this is linear over $\mathcal{O}_W = S( V^\vee)$ we can re-express the right hand side as,

\[(\nabla- \nabla_0)_{\frac{\partial}{\partial x_k}}= \sum e_i^\vee \otimes (\sum f_{ijk} e_j) \]

now note $(\sum f_{ijk} e_j) = (\nabla- \nabla_0)_{\frac{\partial}{\partial x_k}} e_i=\nabla_{\frac{\partial}{\partial x_k}} e_i$, so

\[L-L_0: \pi^* T_{X/S} \to \pi^* V \]
\[\frac{\partial}{\partial x_k} \mapsto \sum f_{ijk} e_j \otimes e_i^\vee = \sum e_i^\vee \otimes \nabla_{\frac{\partial}{\partial x_k}} e_i  \] 

Now compose with the projection,

\[ L-L_0: T_{W/S} \to \frac{T_{W/S}}{T_{W/X}} \simeq \pi^* T_{X/S} \to \pi^* V \]
\[\frac{\partial}{\partial e_i} \mapsto 0 \]
\[\frac{\partial}{\partial x_k} \mapsto \sum e_i^\vee \otimes \nabla_{\frac{\partial}{\partial x_k}} e_i  \] 
\thispagestyle{plain}

and add back $L_0$,

\[ L: T_{W/S} \to  \pi^* V \]
\[\frac{\partial}{\partial e_i} \mapsto \frac{\partial}{\partial e_i}  \]
\[\frac{\partial}{\partial x_k} \mapsto \sum e_i^\vee \otimes \nabla_{\frac{\partial}{\partial x_k}} e_i  \] 

Let the connection matrix of $\nabla$ be $[\theta]$ as described above, a matrix of $1$-forms. Then the right hand side above is given by the matrix $[\theta_{\frac{\partial}{\partial x_k}}]$, considered as a matrix of functions in $\mathcal{O}_X$. Then $L$ is given by $\frac{\partial}{\partial x_k} \mapsto [\theta_{\frac{\partial}{\partial x_k}}]$, understood as an element of $\End(V)$.

%\pagebreak

%take a splitting $L: T_{W/S} \to \pi^* V$. Given a section $\sigma$ of $V$, it defines a map $\sigma: X \to V$ locally defined on some Zariski open in $X$, which I will suppress from the notation. So $d\sigma: T_{X/S} \to \sigma^* T_{W/S}$. Composing with $L$ induces a morphism 
%\[ L \circ d\sigma: T_{X/S} \to \sigma^* T_{W/S} \to \sigma^* \pi^* V = V  \] Define $\nabla \sigma$ to be this element of $\Hom(TX, V)$. So $\nabla: V \to \Hom(TX,V)$. The linearity of $L$ translates into the Leibniz condition on $\nabla$ to show $\nabla$ is a connection. Given a local frame and the trivial splitting associated to it, this returns the trivial connection associated to it, as desired.
\end{proof}
\thispagestyle{plain}

\begin{lemma}
The horizontal subbundle is integrable if and only if the connection is integrable.
\end{lemma}
\begin{proof}
Let $H$ be the horizontal subbundle, and $\nabla$ the connection. The curvature of the connection is a section of $\Omega_{X/S}^2 \otimes_{\mathcal{O}_X} \End(V)$. We can also define the curvature of $H$ given by the morphism

\[\Lambda^2 H \to \frac{T_{W/S}}{H}\]
\[ v_1 \wedge v_2 \mapsto [v_1,v_2] \]

\thispagestyle{plain}
This morphism vanishes if and only if $H$ is preserved by Lie bracket, that is, \\ integrable. Since $H \simeq \pi^* T_{X/S}$ and $\frac{T_{W/S}}{H} \simeq \pi^* V$ this defines a section of \\ $\pi^* \Omega_{X/S}^2 \otimes_{\mathcal{O}_V} \pi^*V$. By a similar calculation as above, and linearity, this defines a section of $\Omega_{X/S}^2 \otimes_{\mathcal{O}_X} \End(V)$.

Now take a local frame, so $W$ has local coordinates $e_1,\ldots,e_r, x_1,\ldots,x_n$. Let $\nabla$ be a connection corresponding to $H$.

Unwinding the correspondences in the proof of proposition 3.10, the curvature of $H$ is given by 

\[ \frac{\partial}{\partial x_k} \wedge \frac{\partial}{\partial x_l} \mapsto [\nabla_{\frac{\partial}{\partial x_k}},\nabla_{\frac{\partial}{\partial x_l}}] \]

This vanishes if and only if $[\nabla_{\frac{\partial}{\partial x_k}},\nabla_{\frac{\partial}{\partial x_l}}] =0$ for all $k,l$, which is precisely the integrability condition for $\nabla$. As such, $\nabla$ is integrable if and only if $H$ is integrable.
\end{proof}

\pagebreak

\section{Crystals and the non-abelian Gauss-Manin connection}
In this section, I will introduce crystals, yet another way to interpret vector bundles with integrable connection. I will then introduce the main object of this dissertation, non-abelian de Rham cohomology, and use its crystalline interpretation to explicitly present its Gauss-Manin connection. The presentation here is modified from \cite{Simpson5}.

\thispagestyle{plain}
Recall $X/S$ is a smooth separated morphism of schemes. Several of these results can be stated under the lesser assumption that $X/S$ is locally of finite type, but for clarity I will keep my assumptions consistent from section to section.

\subsection{Crystals and the crystalline site}
Define the \textit{crystalline site} $\Crys(X/S)$ as follows: its objects are triples $(U,V,\gamma)$ with $U \to X$ a morphism of schemes, $U \subset V$ a closed immersion of $S$-schemes with a nilpotent sheaf of ideals, and $\gamma$ a PD structure. Precisely, we require $U$ to be defined by an ideal sheaf $I$, such that $I^k=0$ for some integer $k$, and $\gamma$ is a PD structure on $I$. A morphism in this category,

\[ f:(U \subset V) \to (U' \subset V') \]

consists of a morphism $V \to V'$ over $S$ respecting the PD structure, with restriction $U \to U'$ a morphism of $X$-schemes. Define a site by equipping this category with all Zariski covers. 

Let $\Crys^r (X/S)$ denote the full subcategory consisting of objects $(U \subset V)$ such that there is a retraction morphism $V \to X$, which itself is not part of the data, such that the following diagram commutes:

\begin{center}

\begin{tikzcd}

   U   \arrow[hookrightarrow]{r}  \arrow[d]&        V   \arrow[ld]      \\
 X &

\end{tikzcd}  

\end{center}

Define a \textit{crystal} of schemes $F$ on $X/S$ to be the data, for each $(U \subset V)$, of a scheme $F(U \subset V) \to V$, and for each morphism  $f:(U \subset V) \to (U' \subset V') $, an isomorphism of schemes,

\[ \psi(f): F(U\subset V) \to f^* F(U' \subset V') \]

satisfying the cocycle condition, $\psi(gf)=f^*(\psi(g)) \psi(f)$ for any $g,f$ morphisms in the site that can be composed. Similarly, we can define a crystal of vector bundles.

A \textit{restricted crystal} is the same sort of data, but with $F(U \subset V)$ only defined for $(U \subset V)$ in $\Crys^r (X/S)$. 

\thispagestyle{plain}

\begin{remark}
For $X/S$ smooth, any object of $\Crys(X/S)$ is Zariski locally isomorphic to an object of $\Crys^r(X/S)$. This follows from the infinitesimal lifting property of smooth morphisms, which guarantees the local existence of $V \to X$. 

As such, crystals and restricted crystals are the same for smooth $X/S$. The \\ distinction serves us in the case when $X/S$ is not smooth. A crystal is the \\ appropriate generalization of a vector bundle with integrable connection to the non-smooth setting.
\end{remark} 

\begin{lemma}
Let $X/S$ be locally of finite type. A restricted crystal of schemes $F$ on $X/S$ is equivalent to a scheme $F(X) \to X$, together with an isomorphism,

\[ \phi: pr_1^* F(X) \xrightarrow{\sim} pr_2^* F(X)  \]

where $pr_i: (X \times_S X)_{PD}^\wedge  \to X$ are the two morphisms from the PD formal \\ completion of the diagonal, such that natural cocycle conditions are satisfied.

\end{lemma}
\begin{proof}
One modifies the proof of \cite{Simpson5}, lemma 8.2. Briefly, in one direction, let $F$ be a restricted crystal on $X/S$. It immediately gives a scheme $F(X)$ over $X$. Consider $P^n_X=(X \times_S X)_{PD}^{[n]}$ as discussed in section 2, the $n$-th PD neighbourhood of the diagonal in $X \times_S X$. These are objects of $\Crys^r (X/S)$. The maps $p_{1,[n]}$ and $p_{2,[n]}$ from $(X \times_S X)_{PD}^{[n]}$ to $X$ serve as the necessary retractions. By the data of the restricted crystal, these give isomorphisms

\[F((X \times_S X)_{PD}^{[n]}) \simeq p_{1,[n]}^* F(X)  \]
\[F((X \times_S X)_{PD}^{[n]}) \simeq p_{2,[n]}^* F(X)  \]

Compose these to get isomorphisms $p_{1,[n]}^* F(X)\simeq p_{2,[n]}^* F(X)$. By the functoriality property of $F$, these are compatible, and provide an isomorphism \\ $\phi: pr_1^* F(X) \simeq pr_2^* F(X)$ between the two pullbacks to the PD completion. The functoriality also gives the cocycle conditions.
\thispagestyle{plain}

In the other direction, suppose we are given such a scheme $F(X)$ and an \\ isomorphism $\phi$. For any $(U \subset V)$ in $\Crys^r (X/S)$, choose a retraction $\rho_V: V \to X$. Define $F(U \subset V)= \rho_V^* F^*(X)$. Given $f: V \to V'$, one considers the pair \\ $(\rho_V, \rho_{V'} f): V \to X \times_S X$. These two maps are equal on $U$, and $(U \subset V)$ is defined by a nilpotent ideal with a divided power, so this map factors through $(X \times_S X)_{PD}^\wedge$. Then $\phi$ pulls back to an isomorphism,

\[\phi(f):=(\rho_V, \rho_{V'} f)^* \phi : F(V) = \rho_V^* F^*(X) \simeq (\rho_{V'} f)^* F(X) = f^* \rho_{V'}^* F(X)= f^* F(V')   \]

The cocycle conditions for $\phi$ give the functoriality property for the comparison isomorphisms.

\thispagestyle{plain}

\end{proof}

\begin{corollary}
If $X/S$ is smooth, a vector bundle with integrable connection on $X/S$ is equivalent to a crystal of vector bundles on $X/S$.
\end{corollary}
\begin{proof}
This follows from lemma 4.2 and proposition 3.9.
\end{proof}
The following proposition is attributed to Grothendieck.
\begin{prop}
Let $(S_0 \subset S)$ be a closed subscheme defined by a nilpotent \\ thickening of ideals equipped with a PD structure. Let $X_0=X \times_S S_0$, and $j: X_0 \to X$ be the inclusion. Then the pullback functor $F \mapsto j^*F$ is an equivalence from the category of crystals on $X/S$ to the category of crystals on $X_0/S.$
\end{prop}

\thispagestyle{plain}

Note that $X_0/S$ is likely not smooth.
\begin{proof}
Have functors $a: \Crys(X_0 /S) \to  \Crys(X /S)$ defined by $a(U \subset V)= (U \subset V)$ and $b: \Crys(X /S) \to  \Crys(X_0 /S)$ by $b(U \subset V)= (U \times_{X} X_0 \subset V)$. Let $U_0=U \times_{X} X_0$.

By pullback, one produces a functor $a^*$ from crystals on $X/S$ to crystals on $X_0/S$ and $b^*$ from crystals on $X_0/S$ to crystals on $X/S$. Now $ba$ is the identity so $a^* b^* $ is the identity. 

On the other hand, $b^* a^* F (U \subset V) = F( U_0 \subset V)$. But using the natural map $( U_0 \subset V) \to ( U \subset V)$ the crystal gives us an isomorphism $F( U_0 \subset V) \simeq F( U \subset V) $. This is a natural isomorphism from $b^* a^*$ to the identity. So $a^*,b^*$ are equivalences.
\end{proof}

\subsection{The Gauss-Manin connection on non-abelian de Rham cohomology}
Define \textit{non-abelian de Rham cohomology}, $M_{dR}(X/S)$, to be the moduli stack of vector bundles with integrable connection over $X/S$. That is, to a test scheme $R\to S$, $M_{dR}(X/S)$ associates the category whose objects are vector bundles with integrable connection on $X \times_S R/R$ and whose morphisms are isomorphisms of vector bundles respecting the connection.

In the next result, I use the crystalline interpretation of vector bundles with \\ integrable connection to explicitly describe the Gauss-Manin connection on $M_{dR}$.

\begin{theorem}
$M_{dR} \to S$ has an integrable connection relative to $S/T$. That is, there exists a canonical morphism satisfying cocycle conditions,

\[ \phi: pr_1^* M_{dR}(X/S) \to pr_2^* M_{dR}(X/S)  \]

where $pr_i: (S \times_T S)_{PD}^\wedge  \to S$ are the two projections from the PD formal completion of the diagonal.

\end{theorem}

\begin{proof}
We want to show this isomorphism between the two pullbacks exists by an \\ indirect method, via lemma 4.2. This involves appealing to the crystalline site $\Crys(S/T)$. Now $M_{dR}(X/S)$ cannot be defined on an element of the crystalline site $(S'_0 \subset S')$ as $S'$ is not a priori an $S$-scheme. So we introduce a new stack $M_{crys}$, which turns out to be a crystal of functors. We then restrict $M_{crys}$ to $\Crys^r(S/T)$ and there show it agrees with $M_{dR}$. It makes sense to define $M_{dR}$ on the restricted crystalline stack, as any element $(S'_0 \subset S') \in \Crys^r(S/T)$ has the property that $S'$ can be thought of as an $S$-scheme by a retraction. This will show $M_{dR}$ has the structure of a restricted crystal of functors and give us the Gauss-Manin connection.  

\thispagestyle{plain}

Let $(S'_0 \subset S')$ be an element of $\Crys(S/T)$. Define a functor \\ $M_{crys}(S'_0 \subset S') : Sch/S' \to Sets$ as follows: for $R' \to S'$ an $S'$-scheme define,
\[ M_{crys}(S'_0 \subset S')(R' \to S')=(\text{category of crystals of vector bundles on } X'_0 \times_{S'} R'/R')\]

%Prop 4.3 tells us that this right hand side is equivalent to

%\[(\text{category of crystals of vector bundles on } X' \times_{S'} Y'/Y')  \]

%Idea: in lemma 4.1, there is an equivalence between a restricted crystal of schemes, vector bundles, or functors, and simply a stratification of the global object, in this case $F(S)$. Apply this to what we claim is a restricted crystal of functors $F=M_{crys}$ on $Crys^r(S/T).$ Now $M_{dR}(X/S)$ is simply the global object $M_{crys}(S)$. As such, in order to prove that $M_{dR}(X/S)$ has an integrable connection relative to $S/T$, lemma 4.1 tells us this is equivalent to show that $M_{crys}$ is in fact a restricted crystal of functors on $\Crys^r(S/T)$. We show the stronger statement, that this is a crystal of functors on $\Crys(S/T)$.

As our first step, we show that this is a crystal of functors on $\Crys(S/T)$. So consider a morphism in this category,

\begin{center}
    
\begin{tikzcd}

S' \arrow[r,"f"]  & S''  \\
 
S_0' \arrow[hookrightarrow]{u} \arrow[r] & S_0'' \arrow[hookrightarrow]{u}

\end{tikzcd}  
\end{center}

Next, we compute the pullback functor $f^* M_{crys}(S''_0 \subset S''): Sch/S' \to Sets$. To an $S'$-scheme $R' \xrightarrow{j} S'$ it assigns,

\[  f^* M_{crys}(S''_0 \subset S'')( R' \xrightarrow{j} S') = M_{crys}(S''_0 \subset S'')( R' \xrightarrow{j} S' \xrightarrow{f} S'') \]

the right hand side given by the category of crystals on $X_0'' \times_{S'', f\circ j} R'/R'$. 

On the other hand, $M_{crys}(S'_0 \subset S')( R' \xrightarrow{j} S')$ is given by the category of crystals on $X_0' \times_{S',  j} R'/R'$. 

\thispagestyle{plain}

To prove the crystalline property we must prove these two categories are equivalent. For simplicity, I will first prove this equivalence in the simplest case, $X=S, R'=S'$ and next in general. The claimed result in this simplest case is an equivalence:

\[(\text{crystals on } S_0'' \times_{S'',f} S'/S') \simeq  (\text{crystals on } S_0'/S') \]

This motivates the use of proposition 4.4 but it is not immediate. First, the commutative diagram above induces a map to the fibre product $u: S_0' \to S_0'' \times_{S'',f} S'$. Then form the following diagram:

\begin{center}
    
\begin{tikzcd}

S' \arrow[dr] &   \\
 
f^{-1}(S_0'')=S_0'' \times_{S'',f} S' \arrow[r] \arrow[hookrightarrow]{u}{v} & S' \\

S_0' \arrow[hookrightarrow]{u}{u} \arrow[ur]

\end{tikzcd}  
\end{center}

This is a tower of closed immersions defined by nilpotent ideals equipped with a PD structure. The PD structure on the ideal of $S'$ cutting out $S'_0$ restricts to the subideal that cuts out $S_0'' \times_{S'',f} S'$ and then the quotient ideal that cuts out $S_0'$ inside $S_0'' \times_{S'',f} S'$. Now proposition 4.4 gives us the claimed result.

\thispagestyle{plain}

In general, once again considering $X, R'$ we have 

\begin{center}
    
\begin{tikzcd}

X_0' \times_{S''} S_0'' =X_0'' \times_{S'', f\circ j} R'  \arrow[r] & R' \\

X_0' \times_{S'} R' \arrow[hookrightarrow]{u}{u} \arrow[ur]

\end{tikzcd}  
\end{center}

in which the vertical map is a base change of a closed immersion with a PD \\ thickening. Then proposition 4.4 tells us the desired equivalence,

\[ (\text{crystals on } X_0'' \times_{S'', f\circ j} R'/R') \simeq  (\text{crystals on } X_0' \times_{S'} R'/R') \]

This concludes the proof that there is a natural isomorphism $$ M_{crys}(S'_0 \subset S')\simeq f^* M_{crys}(S''_0 \subset S'')$$ 

So $F=M_{crys}$ is a crystal of functors on $\Crys(S/T)$. Simply restrict this structure to get a restricted crystal of functors on $\Crys^r(S/T)$. Let $(S'_0 \subset S')$ be an element of $\Crys^r(S/T)$ with some retraction $\rho: S' \to S$. Hence there exists a pullback $X'= X \times_{S,\rho} S'$. Now invoke proposition 4.4 and corollary 4.3. They show

\[M_{crys}(S_0 \subset S')(R' \to S') =(\text{crystals on } X_0' \times_{S'} R'/R') = (\text{crystals on } X' \times_{S'} R'/R') \] 
\[= (\text{vbic on } X' \times_{S'} R'/R') = M_{dR}(X/S)(R' \to S' \xrightarrow{\rho} S) = \rho^* M_{dR}(X/S)(R' \to S')\]

\thispagestyle{plain}

This shows that the functor $M_{dR}(X/S) \to S$ and its pullbacks by the retractions $\rho$ form a restricted crystal of functors over $\Crys^r(S/T)$. By the proof of lemma 4.2, one can explicitly construct an isomorphism between the two pullbacks of $(S \times_T S)_{PD}^\wedge$ satisfying the cocycle conditions.

\end{proof}

%More precisely, we have two Cartesian squares

%\begin{center}
    
%\begin{tikzcd}   

% X'_{Y'} \arrow[r] \arrow[d] & X' \arrow[r] \arrow[d] & X    \arrow[d] \\
 
 %Y' \arrow[r]  & S' \arrow[r] & S

%\end{tikzcd}  
%\end{center}

%\end{proof}

\pagebreak

\section{Groupoids, stacks, and stratifications}

In this section, I will introduce notions of groupoids, fibred categories, their \\ associated stacks, and stratifications. All these notions will be developed in the \\ general case and subsequently applied to the de Rham, Dolbeault and Hodge groupoids. Proposition 5.12 will provide a generalization of the existence of the Gauss-Manin connection proven in the previous section.

\thispagestyle{plain}

\subsection{Groupoids}
Define a \textit{groupoid category} to be a small category $\mathcal{C}$ in which all morphisms are isomorphisms. This induces an equivalence relation on $Ob(\mathcal{C})$ in which two objects are equivalent if they are isomorphic.

Now let $B$ be an arbitrary base scheme, and define a \textit{groupoid scheme} over $B$ to be a tuple $(X,N,s,t,e,c,i)$ satisfying various conditions. Essentially, $X,N$ are (formal) $B$-schemes such that for any $B$-scheme $R$, the set of $R$-valued points becomes a groupoid category with object set $X(R)$ and morphism set $N(R)$. \\ $s,t$ associate to a morphism its source and target respectively, so $s,t: N \to X$. \\ $e$ associates to an object its identity morphism, so $e: X \to N$. \\ $c$ serves as the composition of two morphisms $(f,g) \mapsto f \circ g$. This composition only makes sense if the target of $g$ is the source of $f$, so $c: N \times_{s,X,t} N \to N$. \\ $i$ associates to a morphism its inverse, so $i: N \to N$. 

\begin{definition}
A \textit{groupoid scheme} over $B$ is a tuple $(X,N,s,t,e,c,i)$, with $X,N$ (formal) schemes over $B$, and morphisms of $B$-schemes $s,t: N \to X, e: X \to N, \\ c: N \times_{s,X,t} N \to N, i: N \to N$, satisfying a list of properties that make $(X(R), N(R))$ a groupoid category for any $B$-scheme $R$.
\end{definition}

It is tedious to write down all these axioms. For example, $s \circ e = t \circ e = id: X \to X$ expresses the fact that for an object $A$, $id_A$ has source and target $A$, while the associativity law becomes $c \circ (id, c) = c \circ (c,id): N \times_{s,X,t} N \times_{s,X,t} N \to N$.

\begin{definition}
A morphism of groupoids $(X_1, N_1) \to (X_2,N_2)$ is a pair of \\ morphisms of (formal) schemes $f:X_1 \to X_2, g: N_1 \to N_2$ that induces functors between the categories $(X_1(R),N_1(R)) \to (X_2(R), N_2(R))$. I will not write down all the definitions between the axioms; for instance, $f$ and $g$ must intertwine between source and target morphisms.
\end{definition}

\thispagestyle{plain}

\subsection{Fibred categories and stacks}

Let $\mathcal{C}$ be a base category, in practice always $Sch/B$. Informally, a fibred category $\mathcal{F}$ over $\mathcal{C}$, together with a choice of cleavage, consists of the data, for each $U \in Ob(\mathcal{C})$, of a category $\mathcal{F}(U)$, and for each morphism in $\mathcal{C},f: U \to V$, of a pullback functor $f^*: \mathcal{F}(V) \to \mathcal{F}(U)$. 

Recall that a presheaf is a functor $F: \mathcal{C}^{op} \to (Sets)$. Then a fibred category is also called a prestack, and is informally equivalent to a functor, formally a pseudo-functor, $\mathcal{F}: \mathcal{C}^{op} \to (Cat)$ from $\mathcal{C}$ to the category of small categories. 

Continuing the analogy, let $\mathcal{C}$ be a site, in practice always $(Sch/B)_{fppf}$. Then a \textit{stack} over $\mathcal{C}$ is a fibred category satisfying a sheaf condition for all covers in the site. Much like in the case of sheaves, there is a stack $\mathcal{F^+}$ associated to any prestack, called the stackification. There is always a natural morphism $\pi: \mathcal{F} \to \mathcal{F^+} $ of fibred categories. A fibred category is said to be fibred in groupoids if all categories $\mathcal{F}(U)$ are groupoid categories.

A detailed but accessible treatment can be found in \cite{TBL}. One of the important results proven there is the existence of fibre products, more precisely $2$-fibre products, of fibred categories fibred in groupoids. 

\begin{prop}
Let $\mathcal{F}, \mathcal{G},\mathcal{H}$ be fibred categories fibred in groupoids over $\mathcal{C}$ with morphisms $\mathcal{F} \xrightarrow{f} \mathcal{H}$, $\mathcal{G} \xrightarrow{g} \mathcal{H}$. Then the fibre product $\mathcal{F} \times_{\mathcal{H}} \mathcal{G}$ exists. Its objects over some $U \in Ob(\mathcal{C})$ consist of triples $(x,y,\alpha)$ with $x \in \mathcal{F}(U), y \in \mathcal{G}(U)$ and $\alpha$ an isomorphism in $\mathcal{H}(U)$ between $f(x) \simeq g(y)$.
\end{prop}

 Given a groupoid scheme $(X,N)$, we can consider the following category fibred in groupoids over $B$:

\[(Sch/B)_{fppf}^{op} \to Groupoids \]
\[ R \mapsto (X(R), N(R), s,t,e,c,i) \]

\thispagestyle{plain}

\begin{definition}
Denote this fibred category, that is, prestack, by $[X/_p N]$ and the corresponding stackification by $[X/N]$. This is said to be the \textit{quotient stack} of the groupoid.
\end{definition}

This stackification is a very natural object: informally, one ``sheafifies" the \\ ``functor" that associates the ``set" of isomorphism classes of elements $X(R)$ where isomorphisms are given by $N(R)$.

There are two natural associated morphisms. One is a rather trivial morphism of groupoid schemes $(X,X,id,id,id,id,id) \xrightarrow{(id,e)} (X,N,s,t,e,c,i) = [X/_p N]$ and the other is stackification $[X/_p N] \to [X/N]$. Compose these for a canonical morphism $\pi: X \to [X/N]$.

\subsection{Stratifications}
Using the formalism of groupoids, we can codify the \\ isomorphism between two pullbacks satisfying cocycle conditions that we introduced in previous sections.

\begin{definition}
Let $\mathcal{X}=(X,N,s,t,e,c,i)$ be a groupoid, and $M$ an $\mathcal{O}_X$-module. A \textit{stratification} is a morphism $\phi: t^* M \to s^* N$ of $\mathcal{O}_N$-modules satisfying conditions equivalent to the cocycle conditions. That is,

$e^* \phi : M \to M$ is the identity, and there is a commutative diagram of \\ $N \times_{s,X,t} N$ - modules as follows:

\begin{center}
    
\begin{tikzcd}

 & pr_2 ^* t^* M \arrow[r, "pr_2^* \phi"] & pr_2 ^* s^* M \arrow[dr, equal] & \\
pr_1 ^* s^* M \arrow[ur, equal] & &  & c^* s^* M \\
& pr_1 ^* t^* M \arrow[ul, "pr_1^* \phi"] \arrow[r, equal]  & c^* t^* M \arrow[ur, "c^* \phi"]&

\end{tikzcd}  
\end{center}

Informally, these axioms just express the cocycle conditions. $\phi$ allows you to travel back along a morphism from target to source, expressed in this order due to \\ convention, and this diagram says that using $\phi$ to travel in this way commutes with composition of morphisms in $N$. 
\thispagestyle{plain}

From the groupoid relations, one deduces $\phi$ is necessarily invertible with inverse $i^* \phi$, as also follows from the cocycle conditions of previous sections.

\end{definition}

\begin{definition}
A module over $\mathcal{X}=(X,N,s,t,e,c,i)$ is a pair $(M, \phi)$ of an \\ $\mathcal{O}_X$-module and a stratification.
\end{definition}

\subsection{Properties of groupoids, stacks, and stratifications}
In this section, we prove the properties of these objects we need. I have adopted proofs from \cite{stacks-project} and \cite{Thanos}.

\begin{lemma}
Let $\mathcal{X}$ be a groupoid scheme, $[\mathcal{X}]$ the stackification, and $Y$ a $B$-scheme. Then there is a one-to-one correspondence between morphisms of stacks $[\mathcal{X}] \to Y$ and morphisms of schemes $f: X \to Y$ that equalize $s,t$, namely, $f \circ s=f \circ t: N \to Y$.
\end{lemma}

\begin{proof}
Given a morphism $f: X \to Y$ of schemes that equalizes $s,t$, let \\ $f_N=f \circ s=f \circ t$. Then the morphism

\[ (X,N) \xrightarrow{(f,f_N)} (Y,Y)  \]

is a morphism of groupoid schemes. Passing to stackification induces a morphism $[\mathcal{X}] \to Y$.

Conversely, given a morphism of stacks $[\mathcal{X}] \to Y$, simply compose with the natural map $\pi: X \to [\mathcal{X}] $.

\end{proof}
\thispagestyle{plain}

\begin{definition}
Let $\mathcal{X}$ be a groupoid scheme with $s=t$. I will call such a groupoid scheme a \textit{loop groupoid scheme}. In this case, there is a section $\sigma: [ \mathcal{X} ] \to X$ of the canonical map $\pi: X \to [\mathcal{X}]$. Call this the \textit{trivializing section} of $\mathcal{X}$. 
\end{definition}

\begin{proof}
Take $f=id_X$ in the previous lemma. Essentially, this trivializing section is a forgetful functor that maps a looped category to its set of objects.
\end{proof}

\begin{remark}
Given a loop groupoid, that is $s=t$, it makes sense to ask whether a stratification $\phi$ is the identity. We call this the trivial stratification.
\end{remark}

\begin{prop}
Let $\mathcal{X}=(X,N)$ be a groupoid as above, and $\pi: X \to [X/N]$ the canonical map. Then there is an equivalence between stacks $M' \to [X/N] $, and stacks $M \to X$ with a stratification with respect to $\mathcal{X}$. In one direction, given $M'$ simply pullback to $M=\pi^* M'$.
\end{prop}
\begin{proof}
See \cite{stacks-project}, tag 06WT. In one direction, given $M' \xrightarrow{f} [\mathcal{X}] $, take \\ $M=\pi^* M' = M' \times_{f, [X/N], \pi} X$. Define the stratification by

\[s^* M =  M' \times_{f, [X/N], \pi \circ s} X \xrightarrow{\phi^{-1}} M' \times_{f, [X/N], \pi \circ t} = t^* M\]
\[ (m',n, \alpha: f(m') \to \pi(s(n))) \mapsto (m',n, \pi(n) \circ \alpha: f(m') \to \pi(t(n)))    \]

Here $m'$ is an object of $M'$, $n$ is an object of $N$, $\alpha$ an isomorphism in $[X/N]$. Note $\pi: (X,N) \to [X/N]$, so $\pi(s(n)),\pi(t(n))$ are objects of $[X/N]$ while $\pi(n)$ is a morphism in $[X/N]$.

Informally, both the stratification and the descent are ways to codify the fact that one can travel between objects in $M$, along morphisms in $N$.
\end{proof}

\subsection{Local systems on groupoids}
Henceforth, assume $X$ is smooth, $N$ is formally smooth, and $X$ and $N$ have the same underlying topological space. We can recover $N$ as the fibre product $X \times_{[X/N]} X$, so often abuse notation when considering either $(X,N)$ or $[X/N]$. For notational simplicity write $X_N = [X/N]$.

\thispagestyle{plain}

We have notions such as the structure sheaf $\mathcal{O}$ on $Sch/X_N$, and the sheaf of differentials $\Omega_{X/X_N}$ obtained from $\Omega_{N/X}$ by descent from $N=X \times_{X_N} X$ to $X$. In this formalism, a local system $V$ on $X_N$ is a sheaf over $X_N$ locally isomorphic to a direct sum of finitely many copies of the structure sheaf. It is equivalent to a vector bundle $V$ over $X$ together with an $N$-connection, $\phi: V \to V \otimes \Omega_{X/X_N}$.

Finally, we have the notion of the ring of differential operators $\Lambda_N$. Let $\pi: X \to X_N$ be the canonical map. Define $\Lambda_N$ as the continuous dual of $\pi_* (\mathcal{O_N})$ where the ring structure is the dual of the coalgebra structure on $\pi_* (\mathcal{O_N})$. Then a local system $V$ on $X_N$ is equivalent to a $\Lambda_N$ - module, with action induced by $\phi$.

In the next section, I will give the examples of the above notions for the de Rham, Dolbeault and Hodge groupoids.

\subsection{Groupoid arrow functors}
Let $\mathcal{D}$ be the category of smooth separated \\ morphisms in $(Sch)^{op}$. A \textit{groupoid arrow functor} $(.)$ is a presheaf of groupoids of schemes on $\mathcal{D}$ satisfying certain base change and descent properties. That is, for each smooth separated morphism of schemes $X \to S$, we produce a groupoid scheme $(X/S)$ functorial in $X,S$. The properties we demand are as follows:

$(X \xrightarrow{id} X)$ is a trivial groupoid scheme with $e$ an isomorphism; 

for $X \to S$ in $\mathcal{D}$ and $S' \to S$ an arbitrary base change, the canonical morphism,

\[(X\times_S S'/S') \to (X/S) \times_S S' \]

is an isomorphism; 

for $X \to S \to T$ morphisms in $\mathcal{D}$, the canonical morphism of fibred categories fibred in groupoids,

\[(X/S) \to (X/T) \times_{(S/T)} (S/S) \]

is an isomorphism. Note that such fibre products make sense by proposition 5.3.

In the following examples, $(S/S)$ will be either just the scheme $S$ or $\mathbb{A}^1_S.$

\subsection{The moduli stack of a groupoid}
Let $\mathcal{X}=(X,N)$ be a groupoid scheme over $B$. Recall a module over $\mathcal{X}$ is a pair $(M,\phi)$ with $M$ a module over $X$ and $\phi$ a stratification relative to the groupoid. Define the fibred category $M\mathcal{X}$ over $Sch/B$ as follows: to a scheme $R\to B$, associate the category whose objects are (locally free finite rank) modules on the groupoid $\mathcal{X} \times_B R$ and whose morphisms are isomorphisms of modules over this groupoid.

\thispagestyle{plain}

\begin{prop}
Assume $(X,N)$ are smooth schemes over $B$. Then $M\mathcal{X}$ is an algebraic stack, which is isomorphic to the stack of (locally free finite rank) modules on $[X/N]$.
\end{prop}
\begin{proof}
See \cite{stacks-project}. By tag 04TJ, $[X/N]$ is an algebraic stack. By tag 06WT, we have the isomorphism stated. By tag 08KC,  $M\mathcal{X}$ is a stack. By tag 08KA, it is an algebraic stack.
\end{proof}

In the next proposition, I give an alternative proof of the Gauss-Manin connection explicitly described for the de Rham groupoid.

\begin{prop}
Let $(.)$ be a groupoid arrow functor and $X \to S \to T$ smooth separated morphisms. For notational simplicity, let $X_{st}=[(X/T)], S_{st}= [(S/T)], \\ B = (S/S)$. Then we have a canonical isomorphism

\[ M(X/S) \simeq M(X_{st}/S_{st}) \times_{S_{st}} B     \]

as stacks over $Sch/B$.
\end{prop}

\begin{proof}

First, $(X/S) \to (S/S)$ is a groupoid over $B=(S/S)$. This is our base. Hence $M(X/S)$ is an algebraic stack over $Sch/B$.

Under the assumption that we have a canonical isomorphism

\[(X/S) \xrightarrow{\sim} (X/T) \times_{(S/T)} (S/S) \]

it follows that we also have an isomorphism for the corresponding stackifications, as stackification commutes with fibre products of fibred categories. So

\[[(X/S)] \xrightarrow{\sim} X_{st} \times_{S_{st}} B \]

Now let $R \xrightarrow{f} B$ be a scheme over $B$. Using the proposition 5.11, to give a $R$-valued point of the left hand side $M(X/S)$ is equivalent to giving a module over $[(X/S)] \times_B R$. But we can rewrite this as $X_{st} \times_{S_{st}} B \times_B R =X_{st} \times_{S_{st}, \tau \circ f} R $, where $\tau: (S/S) \to (S/T) \to [(S/T)]$. 

To give a $R$-valued point of the right hand side  $M(X_{st}/S_{st}) \times_{S_{st}} B$ is equivalent to giving an $R$-valued point of $M(X_{st}/S_{st})$ over $S_{st}$, that is, a module over the same stack $X_{st} \times_{S_{st}, \tau \circ f} R$. So the left and right hand sides are canonically identified.
\end{proof}

\thispagestyle{plain}

\begin{definition}
In the next section, we shall specialize to the de Rham and \\ Dolbeault groupoids. For the de Rham groupoid, proposition 5.12 states that $M_{dR}$ has a descent structure. This is the \textit{Gauss-Manin connection}. For the Dolbeault groupoid, proposition 5.12 states that $M_{Dol}$ has a descent structure. This is called the \textit{non-abelian Kodaira-Spencer map}.
\end{definition}

\pagebreak

\section{The de Rham, Dolbeault, and Hodge stacks}
In this section, I will introduce several objects central to this dissertation. In section 6.1, I introduce the three moduli stacks of note, $M_{dR}, M_{Dol}, M_{Hod}$. Next, I have some work to prove these stacks arise as the module stacks over three particular groupoid arrow functors, $(.)_{dR}$, $(.)_{Dol}$ and $(.)_{Hod}$. In section 6.2, I will prove some properties of Higgs Bundles and $\lambda$-connections analogous to the properties of \\ integrable connections in section 3. In sections 6.3, 6.4 and 6.5, I will define and present the de Rham, Dolbeault and Hodge groupoid.
\subsection{The moduli stacks $M_{dR}, M_{Dol}, M_{Hod}$}

\thispagestyle{plain}

\begin{definition}
$M_{dR}$ is the moduli stack of vector bundles with integrable connection over $X/S$. That is, to a test scheme $R\to S$, $M_{dR}(X/S)$ associates the category of vector bundles with integrable connection on $X \times_S R/R$, together with isomorphisms between them.
\end{definition}

\begin{definition}
A \textit{Higgs bundle} on $X/S$, or vector bundle with Higgs field, is an $\mathcal{O}_X$-module $E$ locally free of finite rank equipped with an $\mathcal{O}_X$-linear morphism \\ $\phi: E \to E \otimes \Omega_{X/S}$ satisfying an integrability condition $\phi^2=0$.

$M_{Dol}$ is the moduli stack of Higgs bundles over $X/S$. That is, to a test scheme $R\to S$, $M_{Dol}(X/S)$ associates the category of vector bundles with a Higgs field on $X \times_S R/R$, together with isomorphisms between them.

\end{definition}
\begin{definition}
$M_{Hod}$ is the moduli stack of $\lambda$-connections over $\mathbb{A}^1_S$. That is, to a test scheme $\lambda: R \to \mathbb{A}^1_S$, $M_{Hod}(X/S)$ associates the category of vector bundles $E$ on $X \times_S R$ over $R$ equipped with a $\lambda$-connection, always assumed to be integrable,

\[\nabla: E \to E \otimes \Omega_{X \times R/R} \]

such that $\nabla(a e)=\lambda e \otimes da + a\nabla(e) $ and $\nabla^2=0$.
\end{definition}

\begin{remark}
The universal case is $R=S\times \mathbb{A}^1$ and then a $\lambda$-connection is a vector bundle $E$ over $X \times \mathbb{A}^1$ with a morphism 

\[\nabla: E \to E \otimes \Omega_{X \times \mathbb{A}^1/S \times \mathbb{A}^1} \]

satisfying the same two conditions.
\end{remark}

The key diagram to keep in mind is the following

\begin{center}
    
\begin{tikzcd}

M_{dR} \arrow[hookrightarrow]{r} \arrow[d] & M_{Hod} \arrow[d] & M_{Dol} \arrow[hookrightarrow]{l} \arrow[d]\\
S \arrow[hookrightarrow]{r}{1} & \mathbb{A}^1_S & S\arrow[hookrightarrow]{l}{0}

\end{tikzcd}  
\end{center}

\thispagestyle{plain}
\subsection{Properties of Higgs bundles and $\lambda$-connections}
In this section, we show analogous results of section 3 for Higgs bundles and $\lambda$-connections. In what follows, let $\Omega^r=\Omega^r_{X \times \mathbb{A}^1/S \times \mathbb{A}^1}$ and  $\mathcal{A}^r= \Omega^r\otimes E$, understood as sheaves, with all tensor products taken over $X \times \mathbb{A}^1.$ Let $\nabla$ be a $\lambda$-connection on $E$. Specializing to $\lambda=1$ will recover statements for ordinary integrable connections, and $\lambda=0$ will recover statements for Higgs bundles.

\begin{prop}
There exists a well-defined covariant derivative

$d^E: \mathcal{A}^r \to \mathcal{A}^{r+1}$ satisfying, if $\omega \in \Omega^r, \sigma \in E$,
\[d^E(\omega \otimes \sigma)=\lambda d \omega \otimes \sigma + (-1)^r \omega \wedge \nabla \sigma\]

It has a Leibniz rule: if $\xi \in \mathcal{A}^r, \omega \in \Omega^q$ then 

\[d^E(\omega \wedge \xi)=\lambda d \omega \wedge \xi + (-1)^q w \wedge d^E \xi  \]

If $\mu \in \mathcal{A}^1$ and $A,B$ are sections of $T_X$,
\[ d^E(\mu)(A,B)=\nabla_A(\mu(B)-\nabla_B(\mu(A))-\lambda \mu ([A,B])  \]
\[ 0=(d^E)^2(\sigma)(A,B) = \nabla_A \nabla_B \sigma - \nabla_B \nabla_A \sigma - \lambda \nabla_{[A,B]}\sigma \]

Due to these identities, to give a $\lambda$-connection on $E$ is equivalent to the conditions

\[\nabla_A \nabla_B \sigma - \nabla_B \nabla_A \sigma = \lambda \nabla_{[A,B]}\]
\[ \nabla_A(f \sigma)=\lambda A(f) e + f\nabla_A(\sigma) \]

for all $A,B \in T_{X}.$ 
\end{prop}
\begin{remark}
Note that $T_{X\times\mathbb{A}^1/S\times\mathbb{A}^1}=T_{X/S} \otimes_{\mathcal{O}_X} \mathcal{O}_X[t]$ so I will not distinguish between $T_X$ and its extension $T_{X/S} \otimes \mathcal{O}_X[t]$. $\nabla$ will always be linear over $t$.
\end{remark}
\begin{proof}
Routine, following from the same calculations as in the vector bundle case.

\thispagestyle{plain}

\end{proof}
\begin{corollary}
As for vector bundles with integrable connection, there is an \\ equivalence of three categories, between Higgs bundle on $X$, modules over $S(T_X)$, and stratifications over $\hat{\Gamma}_X \Omega_{X/S}$.
\end{corollary}
\begin{proof}
Take $\lambda=0$ above. A Higgs field on $E$ is equivalent to a left action on $E$ under the free non-commutative $\mathcal{O}_X$-algebra with generators sections of $T_X$ and the relations 
\[AB-BA=0; A f = f A\]
This algebra is $S(T_X)$, hence the first part of the claim.

For the second part of the claim, one uses the fact that $T_X$ and $\Omega_X$ are locally free and dual, so there is a perfect pairing between $S^{\leq n} T_X$ and $\Gamma^{\leq n}\Omega_X$, which is canonical, not left or right sided. This induces a pairing between $S(T_X)$ and $\hat{\Gamma}_X \Omega_{X/S}$. From there, it proceeds exactly like the proof of proposition 3.9 that a module structure over $S(T_X)$ is equivalent to a stratification over $\hat{\Gamma}_X \Omega_{X/S}$.
\end{proof}

\begin{definition}[$\Lambda_{Hod}$]
Let $\Lambda_{dR}=\Lambda$ be the standard ring of differential operators on $X/S$. This has an increasing filtration $\Lambda_0 \subset \Lambda_1 \subset \ldots$ given by the differential operators of order at most $i$. Let $j: \mathbb{G}_m \to \mathbb{A}^1$ be the inclusion. Then $\Lambda_{Hod}$ is defined to be the subsheaf of $j_{*} (\Lambda_{dR} \otimes \mathcal{O}_{\mathbb{G}_m})$ generated by sections of the form $t^{m} \lambda_m$, where $t$ is the coordinate on $\mathbb{A}^1$ and $\lambda_m \in \Lambda_m$. That is, \\ $\Lambda_{Hod}= \mathcal{O}_X [t] \Lambda_0+ \mathcal{O}_X [t] t\Lambda_1+ \mathcal{O}_X [t] t^2\Lambda_2  \ldots$
\end{definition}

Taking the fibre at $t=1$ returns $\Lambda_{dR}$. Taking the fibre at $t=0$ returns the associated graded $Gr(\Lambda_{dR})=\Lambda_{Dol}$. To see this,

\[ \Lambda_{Hod}= \mathcal{O}_X [t] \Lambda_0+ \mathcal{O}_X [t] t\Lambda_1+ \mathcal{O}_X [t] t^2\Lambda_2  \ldots \]
\[t \Lambda_{Hod}= t \mathcal{O}_X [t] \Lambda_0+ \mathcal{O}_X [t] t^2\Lambda_1+ \mathcal{O}_X [t] t^3\Lambda_2  \ldots\]

So 
\[\frac{\Lambda_{Hod}}{t\Lambda_{Hod}}= \mathcal{O}_X\Lambda_0 + \bar{t} \mathcal{O}_X \frac{\Lambda_1}{\Lambda_0} + \bar{t^2} \mathcal{O}_X \frac{\Lambda_2}{\Lambda_1} + \ldots \]

The $\bar{t^i}$ are now formal symbols; this is a direct sum, which makes the associated graded.

\begin{corollary}
A $\lambda$-connection on $E$ is equivalent to a left-module structure over $\Lambda_{Hod}$.
\end{corollary}
\begin{proof}
A $\lambda$-connection on $E$ is equivalent to a left action on $E$ under the free non-commutative $\mathcal{O}_{X\times\mathbb{A}^1}$-algebra with generators sections of $T_X$ and the relations 

\[AB-BA=t[A,B]; A f =t A(f) + f A\]
\thispagestyle{plain}

Call this $\Lambda_{abstract}$ as before. Now map $\Lambda_{abstract}$ to $\Lambda_{Hod}$ by sending $m$-th fold product terms of elements to $t^m\Lambda_m$. $\Lambda_{Hod}$ satisfies the relations above, so this induces a well-defined map $\Lambda_{abstract} \to \Lambda_{Hod}$. By construction, it sends $\mathcal{O}_{X\times \mathbb{A}^1}$ to $\mathcal{O}_X [t] \Lambda_0$ and $T_X$ to $\mathcal{O}_X [t] t\Lambda_1$.  Passing to local coordinates, it is a bijection. $\Lambda_{abstract}$ as described above is canonically isomorphic to $\Lambda_{Hod}$. Hence a $\lambda$-connection on $E$ is equivalent to a $\Lambda_{Hod}$-module.
\end{proof}

\subsection{The de Rham groupoid $(.)_{dR}$}
In section 4, I explicitly showed $M_{dR}$ has an integrable connection in the form of a stratification between two pullbacks. In section 5, I introduced a general formalism of groupoids as a means to formalize stratifications in a more general context. In this section, I will give an alternative proof of the Gauss-Manin connection. Recall all morphisms are smooth separated.

\begin{definition}[de Rham groupoid]
Define the groupoid $(X/S)_{dR}$ as follows: its object object is $X$, while its morphism object $N$ is the PD formal completion of the diagonal of $X\times_S X$. That is, in the notation of section 3,
\[ N= \hat{\mathcal{P}}=\lim_{n\to \infty} \text{Spec} \dfrac{\mathcal{P}}{J^{[n]}} \]

Then $s,t$ are induced by the two projections $X\times_S X \to X$, $e$ is induced by the diagonal $X \to X\times_S X$ and $c$ is induced by the comultiplication map on $\mathcal{P}$ described in proposition 2.14.

\thispagestyle{plain}

\end{definition}

\begin{prop}
$(.)_{dR}$ is a groupoid arrow functor.
\end{prop}
\begin{proof}
There are three properties to verify as described in section 5.6. The identity property is trivial, as $(X/X)_{dR}=X$. The base change property is proposition 2.10 and the descent property is corollary 2.9, and taking the inverse limit over $n$.
\end{proof}

\begin{corollary}
$M_{dR}(X/S)$ is an algebraic stack over $S$, and has a descent to $(S/T)_{dR}$. This descent is the Gauss-Manin connection.
\end{corollary}
\begin{proof}
By proposition 3.9, $M_{dR}(X/S)$ is nothing other than the moduli stack $M(X/S)_{dR}$. By proposition 5.11, $M_{dR}(X/S)$ is an algebraic stack. By proposition 5.12, we see $M_{dR}$ has a descent to $[(S/T)]_{dR}$. This descent is the Gauss-Manin connection \\ described explicitly in section 4.
\end{proof}

\begin{remark}
See section 5.5 for the general formalism of differential operators on a groupoid. In this instance, the sheaf of differential operators associated to \\ $X_N=[(X/S)_{dR}]$ is just the standard ring $\Lambda_{dR}=\Lambda$ of differential operators on $X/S$. A local system over $X_{dR}$ is just a vector bundle on $X/S$ with integrable connection. This statement summarizes the equivalences in section 3.3.

\end{remark}

\begin{definition}
Let $x: S \to X$ be a section of $X/S$. Then there is a scheme $R_{dR}(X,S,x)$ representing the functor that assigns to $R/S$ the set of isomorphism classes of $(V,\nabla, \alpha)$ where $(V,\nabla)$ is a vector bundle with integrable connection on $X_R/R$ and $\alpha: V_{x(R)} \simeq \mathcal{O}^d_R$ is a trivializing frame along this section. 
\end{definition}

This representability is theorem 6.13 of \cite{Simpson5}. As a consequence, we may explicitly express $M_{dR}$ as a quotient stack, $[R_{dR}/GL_d]$. In particular we have a surjective morphism of stacks $R_{dR} \to M_{dR}$, which shows $M_{dR}$ is an algebraic stack.

\thispagestyle{plain}

\subsection{The Dolbeault groupoid $(.)_{Dol}$}

\begin{definition}[Dolbeault groupoid]
Define the groupoid $(X/S)_{Dol}$ as follows: its object object is $X$, while its morphism object $N$ is the PD formal completion of the zero section of the tangent bundle, a formal scheme with structure sheaf $\hat{\Gamma}_X \Omega_{X/S}$.
\end{definition}

$e$ is induced by the projection $\hat{\Gamma}_X \Omega_{X/S} \to \mathcal{O}_X$ that kills all of $\Omega_{X/S}$. Geometrically, this is induced by the zero section $X \to T_{X/S}$ 

$s=t$ is induced by the natural inclusion in the other direction $\mathcal{O}_X \to \hat{\Gamma}_X \Omega_{X/S}$

$c$ is induced by the comultiplication structure on $\Gamma M$; see remark 2.15. \\ Geometrically this corresponds to the addition law on the tangent bundle. Note this is a loop groupoid, that is, $s=t$.

\begin{prop}
$(.)_{Dol}$ is a groupoid arrow functor.
\end{prop}
\begin{proof}
Again, there are three properties to verify, with the identity property trivial, as $(X/X)_{Dol}=X$. The base change property follows from the analogous base change property for $\Omega_{X/S}$. All that remains is the descent property, that
\[ (X/S)_{Dol} \to (X/T)_{Dol} \times_{(S/T)_{Dol}} S  \]
is an isomorphism.
\thispagestyle{plain}

To proceed, consider the diagram

\begin{center}

\begin{tikzcd}

   X   \arrow[hookrightarrow]{r}  \arrow[d]&        T_{X/T}   \arrow[d]      \\
 S \arrow[hookrightarrow]{r} & T_{S/T}

\end{tikzcd}  

\end{center}

where the horizontal maps are the closed embeddings given by the zero section of the tangent bundle.

By proposition 2.8, there are natural isomorphisms

\[ P^n_X (S \times_{T_{S/T}} T_{X/T}) \xrightarrow{\sim} P^n_X (T_{X/T}) \times_{P^n_S(T_{S/T})} S  \]

I claim $S \times_{T_{S/T}} T_{X/S} \simeq T_{X/S}$. Granting this claim, we have 

\[ P^n_X (T_{X/S}) \xrightarrow{\sim} P^n_X (T_{X/T}) \times_{P^n_S(T_{S/T})} S  \]

Now take the limit over $n$ to deduce the result. 

To prove the claim, we must prove the following diagram is Cartesian:

\begin{center}

\begin{tikzcd}

   T_{X/S}   \arrow[hookrightarrow]{r}  \arrow[d]&        T_{X/T}   \arrow[d]      \\
 S \arrow[hookrightarrow]{r} & T_{S/T}

\end{tikzcd}  

\end{center}

This holds because both the top and bottom immersions are cut out by the \\ vanishing of $\Omega_{S/T}$. That is, if $X/S$ has local coordinates $x_1,\ldots, x_n$ and $S/T$ has local coordinates $s_1,\ldots,s_m$, then on the level of rings of functions the bottom map is $\mathcal{O}_S [ds_i] \to \mathcal{O}_S $ sending all $ds_i \mapsto 0$ while the top map is $\mathcal{O}_X [dx_j, ds_i] \to \mathcal{O}_X [dx_j] $ also sending all $ds_i \mapsto 0$.

\end{proof}

\begin{corollary}
$M_{Dol}(X/S)$ is an algebraic stack over $S$, and has a descent to $[(S/T)]_{Dol}$. This descent is the non-abelian Kodaira-Spencer map.
\end{corollary}
\thispagestyle{plain}

\begin{proof}
By corollary 6.7, $M_{Dol}(X/S)$ is nothing other than the moduli stack $M(X/S)_{Dol}$. By proposition 5.11, $M_{Dol}(X/S)$ is an algebraic stack. By proposition 5.12, we see $M_{Dol}$ has a descent to $[(S/T)]_{Dol}$. Note that because $(S/T)_{Dol}$ is a loop groupoid, it makes sense for the associated stratification to be trivial.
\end{proof}

\begin{remark}
Recall the definitions of section 5.5. In this case, the sheaf of differential operators associated to $X_N=[(X/S)_{Dol}]$ is the associated graded \\ $\Lambda_{Dol}=Gr(\Lambda_{dR})=S(T_{X/S})$ of $\Lambda$. A local system over $X_{Dol}$ is just a Higgs bundle on $X/S$.

\end{remark}

\begin{definition}
Let $x: S \to X$ be a section of $X/S$. Then there is a scheme $R_{Dol}(X,S,x)$ representing the functor that assigns to $R/S$ the set of isomorphism classes of $(V,\phi, \alpha)$ where $(V,\phi)$ is a Higgs bundle on $X_R/R$ and $\alpha: V_{x(R)} \simeq \mathcal{O}^d_R$ is a trivializing frame along this section. 
\end{definition}

This representability is proven in section 6 of \cite{Simpson5}. As a consequence, we may explicitly express $M_{Dol}$ as a quotient stack, $[R_{Dol}/GL_d]$. In particular, we have a surjective morphism of stacks $R_{Dol} \to M_{Dol}$, which shows $M_{Dol}$ is an algebraic stack.

\subsection{The Hodge groupoid $(.)_{Hod}$}
In this section, we begin with a definition and detailed local description of the Hodge groupoid. Then, we show that modules over $X_{Hod}$ are equivalent to $\lambda$-connections.
\begin{definition}
Let $Z \xhookrightarrow{} W$ be a closed immersion of schemes. Consider the composition $Z \xhookrightarrow{} W \xhookrightarrow{} \mathbb{A}^1_W.$ The \textit{deformation to the normal cone} is the blowup of $\mathbb{A}^1_W$ along $Z$, followed by the removal of the proper transform of $W$. That is,

\[ DNC_Z W = Bl_Z( \mathbb{A}^1_W) - Bl_Z W \]

where the blowup $Bl_Z W$ embeds into the $0$-fibre. We will adopt the presentation of \cite{Fulton}.
\end{definition}
\thispagestyle{plain}

\begin{definition}[Hodge groupoid]
$(X/S)_{Hod}$ is a groupoid whose object object is $X \times \mathbb{A}^1_S$ and whose morphism object $N_{Hod}$ is given by the following procedure. First, blow up $ X \times_S X  \times_S  \mathbb{A}_S^1$ at  $\Delta \times 0$. Next, take the complement of the strict transform of $X \times X \times_S 0$. Finally, take the PD completion along $\Delta \times_S \mathbb{A}^1$.
\end{definition}

Equivalently, form the deformation to the normal cone

\[ Bl_{\Delta \times 0} (X \times X \times \mathbb{A}^1) - Bl_X (X \times X) \]

and then PD complete along $\Delta \times \mathbb{A}^1$.

\begin{prop}
$N_{Hod}$ is flat over $\mathbb{A}^1.$ Its fibres over $t=1$ and $t=0$ are the \\ morphism objects of the de Rham and Dolbeault groupoids respectively. Working in local coordinates, we can give an explicit affine description for its morphism object, as the completion of $\mathcal{O}_X <\sigma_1,\ldots,\sigma_n>[t]$ where $t\sigma_i=\tau_i$ are the diagonal coordinates.
\end{prop}

\begin{proof}
Work locally on $X$ and $S$. Assume $X=\Spec{A}$, let $B=A\otimes A$, and let $I$ be the ideal of the diagonal in $X\times_S X$. The blowup of $ X \times_S X  \times_S  \mathbb{A}_S^1$ along $\Delta \times 0$ is given by 

\[ \Proj(\bigoplus (I,t)^n ) \]

where $t$ is the coordinate on $\mathbb{A}^1$. Under the embedding of $\Delta \times \mathbb{A}^1$ into $ X \times X  \times  \mathbb{A}_S^1$, the preimage of $\Delta \times 0$ is a Cartier divisor, so $\Delta \times \mathbb{A}^1$ naturally lifts to the blowup. Similarly, $Bl_X(X \times X)$ also naturally maps into the $0$-fibre of the larger blowup $Bl_X( X \times X \times \mathbb{A}^1)$. 

Following Fulton, let $S^n =(I,t)^n $. $\Proj(S^\bullet)$ has basic open covers $\Spec(S^\bullet _{(b)})$ where $b$ runs over generators of $(I,t)$, and 
\thispagestyle{plain}

\[ S^\bullet _{(b)} = \{\frac{s}{b^n}, s \in S^n \}.\]

Consider $U = \Spec(S^\bullet_{(t)})$. The complement of $U$ in the blowup is covered by the intersection of opens $\Spec(S^\bullet _{(b)})$ for $b \in I$ (that is, all generators of $(I,t)$ excluding $t$ itself) with the vanishing locus $t=0$. This is precisely the blowup $Bl_X (X \times X)$. Hence $U$ is the deformation to the normal cone. Now

\[ S^\bullet_{(t)} = \ldots \oplus I^n t^{-n} \oplus \ldots It^{-1} \oplus B \oplus Bt \oplus \ldots  \]

This can be rewritten as 

\[A [\frac{i}{t}, t] \]

where $i \in I$. The ideal $K$ of $\Delta \times \mathbb{A}^1$ is generated by these elements $\frac{i}{t}$. 

Now we have local coordinates $x_i$ and their corresponding diagonal elements $\tau_i$. In the ring $S^\bullet_{(t)}$ , we have globally defined elements $\sigma_i$ with $t \sigma_i = \tau_i$.

Now proceed as in proposition 2.11. $K/K^2 \simeq I t^{-1}/I^2 t^{-2} \simeq \frac{1}{t} \Omega_{X/S} \otimes A[t]$ is freely generated by $\frac{1}{t} dx_i$. As before there is a simple description of the associated graded

\[Gr_{K^\bullet} (S^\bullet_{(t)}) = \bigoplus K^n/K^{n+1} = S^\bullet (\frac{1}{t} \Omega_{X/S}) \otimes A[t] \]

Passing to the PD envelope $\mathcal{Q}$ with PD ideal $J$,
\[Gr_{J^\bullet} \mathcal{Q} = \bigoplus J^{[n]}/J^{[n+1]} =: \Gamma \]

So $J^{[n]}/J^{[n+1]}$ is freely generated over $A[t]$ by elements $\frac{1}{t^n} \underline{dx}^{[\underline{k}]}$ with $\sum k_\lambda = n$. 

These elements coincide with the images of elements $\frac{1}{t^n} \underline{\tau}^{[\underline{k}]}$ where $\tau_i$ are the diagonal elements. 

By using the same exact sequence in proposition 2.11, and inducting on $n$, this shows the PD completion can be described as a completed PD algebra \\ $\mathcal{O}_X <<\sigma_1,\ldots, \sigma_n>> [t]$.

\thispagestyle{plain}

Now let's recover the two key fibres, namely the PD envelope of the diagonal, and $\Gamma \Omega$. Above $t=1$, we have $\tau_i= \sigma_i$ and we clearly have a copy of $\hat{\mathcal{P}}$. Above $t=0$, we clearly see a copy of $\mathcal{O}_X<<\sigma_1,\ldots,\sigma_n>>$ but we must explain why this is canonically $\Gamma \Omega_{X/S}=\mathcal{O}_X<<dx_1,\ldots,dx_n>>$. We need to see why the homogeneous coordinates $\Sigma_i$ that are introduced in computing the Proj reduce to just $dx_i.$

Given a general deformation to the normal cone, where we take \\ $\text{Bl}_{W \times 0}(Z \times \mathbb{A}^1)-\text{Bl}_W Z$, the central fibre of $\text{Bl}_{W \times 0}(Z \times \mathbb{A}^1)$ has two pieces. If $I$ is the ideal of $W$ inside $Z$, then the central fibre of $\text{Bl}_{W \times 0}(Z \times \mathbb{A}^1)$ is 

\[\text{Proj} \frac{\oplus(I,t)^n}{t(\oplus(I,t)^n) }        \]

This is the union of 

\[E_{W\times 0}(Z\times \mathbb{A}^1)=\text{Proj}\oplus \frac{(I,t)^n}{(I,t)^{n+1}}  \]

and

\[ \text{Bl}_W Z = \text{Proj} \oplus I^n  \]

along

\[E_W Z=\text{Proj} \oplus \frac{I^n}{I^{n+1}} \]

By removing $\text{Bl}_W Z$, we reduce the central fibre to just $\text{Spec}(\oplus \frac{I^n}{I^{n+1}})$. This whole process corresponds to introducing generating elements of $I$ as homogeneous \\ coordinates and then taking their image in this associated graded. That is, the $\sigma_i$ are simply the images of $\tau_i$ in $\frac{I}{I^2}$. However, the elements $dx_i$ by definition are also the images of $\tau_i$ in  $\Omega_{X/S}=\frac{I}{I^2}$. So there is a canonical identification between $\sigma_i$ in the $0$ fibre and $dx_i$. 

\thispagestyle{plain}

Finally, given this local affine description, it is immediate that $N_{Hod}$ is flat over $\mathbb{A}^1$.

\end{proof}

\begin{prop}
A vector bundle with $\lambda$-connection is equivalent to a module over $X_{Hod}$. That is, the data of a $\lambda$-connection on $E$ over $X \times \mathbb{A}^1$ is equivalent to an isomorphism between the two pullbacks $(pr_1)^* E \to (pr_2)^*E$ to the Hodge groupoid morphism object, satisfying cocycle conditions.
\end{prop}

\thispagestyle{plain}

\begin{corollary}
As for Higgs bundles, there is an equivalence of three categories, between $\lambda$-connections, modules over $\Lambda_{Hod}$, and stratifications with respect to $X_{Hod}$. That is, local systems on the Hodge groupoid are precisely $\lambda$-connections.
\end{corollary}
\begin{proof}
We have already shown that there is an equivalence between $\lambda$-connections and modules over $\Lambda_{Hod}$. Let $\hat{\mathcal{P}}_{Hod}$ be the morphism object of the Hodge groupoid. Locally this has a description 
\[ \mathcal{O}_X <<\sigma_1,\ldots,\sigma_n>>[t] \]
where $\tau_i=t\sigma_i$.

Recall that there exists a duality between $\Lambda$ and $\hat{\mathcal{P}}$ as left $\mathcal{O}_X$-modules between basis elements $\partial^k$ and $\tau^{[k]}$. More precisely, there exists a duality between $\Lambda_n$ and $\mathcal{P}^n$. $\Lambda_{Hod}$ has an increasing order filtration, while $\mathcal{P}_{Hod}$ has a decreasing sequence of divided power ideals, hence a surjective system of quotients $\mathcal{P}^n_{Hod}$.

Consider $\Lambda_{Hod}^{\leq n}$ and $\mathcal{P}^n_{Hod}$ as left $\mathcal{O}_X [t]$-modules. $\Lambda_{Hod}^{\leq n}$ has basis elements given by $t^k\partial^k$, which in the $t=0$ fibre specialize to elements $\bar{\partial}^k$ as elements of the associated graded $S(T_X)$. $\mathcal{P}^n_{Hod}$ has basis elements $\sigma^{[k]}$, which specialize to elements $(dx)^{[k]}$ at the central fibre $\Gamma \Omega_X$. 

The $\mathcal{O}_X$-linear pairing between $\partial^k$ and $\tau^{[k]}$ extends uniquely to a $\mathcal{O}_X [t]$-linear \\ pairing between $t^k\partial^k$ and $\sigma^{[k]}$, which specializes at the $t=0$ fibre to the pairing between $S(T_X)$ and $\Gamma \Omega_X$. That is, the canonical pairing between $\Lambda_n$ and $\mathcal{P}^n$ extends over $\mathbb{A}^1$ to a pairing between $\Lambda_{Hod}$ and $\mathcal{P}_{Hod}$, which is perfect at finite levels.

Once this is established, it follows just as before that a compatible set of actions $\Lambda_{Hod}^{\leq n}$ on $E$ dualizes to a compatible set of morphisms $(pr^{(n)}_1)^* E \to (pr^{(n)}_2)^* E$ between the two pullbacks from $\mathcal{P}^n_{Hod}$, and hence a stratification over the Hodge groupoid. Again, the associativity of the action of $\Lambda_{Hod}$ corresponds to the cocycle conditions.

\end{proof}
\thispagestyle{plain}

\begin{prop}
$(.)_{Hod}$ is a groupoid arrow functor.
\end{prop}
\begin{proof}
Note this time, $(X/X)_{Hod}=\mathbb{A}^1_X$. We must prove two properties, first that given $S' \to S$ an arbitrary base change, the canonical morphism

\[(X\times_S S'/S')_{Hod} \to (X/S)_{Hod} \times_S S' \]

is an isomorphism. This is trivial on object objects and follows for morphism objects by the explicit description above. Alternatively, we can appeal to flatness. This is an isomorphism fibrewise over $\mathbb{A}^1$ by the fact that the de Rham and Dolbeault groupoids are groupoid arrow functors ($t\neq 0$ and $t=0$ respectively). By flatness, it is an isomorphism.

Secondly, we must prove for $X \to S \to T$ morphisms in $\mathcal{D}$, that the canonical morphism is an isomorphism

\[(X/S)_{Hod} \to (X/T)_{Hod} \times_{(S/T)_{Hod}} \mathbb{A}^1_S \]

Again, this is trivial on object objects. For morphism objects, work locally, \\ using local coordinates. If $X/S$ has local coordinates $x_1,\ldots, x_n$, with \\ corresponding diagonal coordinates $\tau_i$, while $S/T$ has coordinates $s_1,\ldots, s_m$, with \\ corresponding diagonal coordinates $\sigma_j$, then this claim reduces to noticing that \\ beginning with $\mathcal{O}_X <<\tau_i,\sigma_j>>[t]$ and killing all $\sigma_j$ returns $\mathcal{O}_X <<\tau_i>>[t]$.

Alternatively, we could appeal to flatness again as above.
\end{proof}

\begin{corollary}
$M_{Hod}(X/S)$ is an algebraic stack over $\mathbb{A}^1_S$, and has a descent to $[(S/T)]_{Hod}$.
\end{corollary}
\begin{proof}
Corollary 6.24 shows $M_{Hod}(X/S)=M(X/S)_{Hod}$, proposition 5.11 gives the algebraicity, and proposition 5.12 gives the descent.
\end{proof}
\thispagestyle{plain}

\begin{definition}
Let $x: S \to X$ be a section of $X/S$. Then there is a scheme $R_{Hod}(X,S,x)$ representing the functor that assigns to $R/\mathbb{A}^1_S$ the set of isomorphism classes of $(V,\nabla, \alpha)$ where $V$ is a vector bundle on $X_R/R$, $\nabla$ a $\lambda$-connection and $\alpha: V_{x(R)} \simeq \mathcal{O}^d_R$ is a trivializing frame along this section. 
\end{definition}

This representability is proven in section 6 of \cite{Simpson5}. As a consequence, we may explicitly express $M_{Hod}$ as a quotient stack, $[R_{Hod}/GL_d]$. In particular we have a surjective morphism of stacks $R_{Hod} \to M_{Hod}$, which shows that $M_{Hod}$ is an algebraic stack.

\begin{corollary}
There is a blowdown map $(X/S)_{Hod} \to (X/S)_{dR} \times \mathbb{A}^1_S$. On object objects it is the identity map. On morphism objects it can be described by 

\[ \mathcal{O}_X <<\tau_1,\ldots,\tau_n>>[t] \to \mathcal{O}_X <<\sigma_1,\ldots,\sigma_n>>[t] \]
\[ \tau_i \mapsto t\sigma_i \]

In particular over the $0$-fibre it is given by a trivial map $X_{Dol} \to X_{dR}$, which factors through $X_{Dol} \to X \to X_{dR}$.
\end{corollary}
\begin{proof}
Immediate from the calculations of proposition 6.22.
\end{proof}

\subsection{An alternative construction of $X_{Hod}$}
In this section, we aim to construct a groupoid from a different perspective, with object object $X\times \mathbb{A}^1_S$ and with a morphism object that deforms between the de Rham and Dolbeault groupoids. First, we deform between the uncompleted analogues, namely the diagonal and the zero section of the tangent bundle. The diagonal is

\[X \times X = \text{Hom} (\text{Spec}\dfrac{R[\epsilon]}{\epsilon(\epsilon-1)}, X) \]

while the zero section of the tangent bundle is 

\[ \text{Hom} (\text{Spec}\dfrac{R[\epsilon]}{\epsilon^2}, X) \]

\begin{prop}
There is a groupoid G that deforms between the diagonal and the zero section of the tangent bundle. It maps to $X \times X \times \mathbb{A}^1$, then lifts to the appropriate blowup of $\Delta \times 0$. This lift is an isomorphism onto the complement of the strict transform of $X \times X \times 0$.  
\end{prop}
\thispagestyle{plain}

\begin{proof} 
%We want to describe $X_{Hod}$ as the quotient of $X \times \mathbb{A}^1$ by the PD completion of an appropriate group. 

Let $\text{Spec}R \to \mathbb{A}^1_S$ be an arbitrary affine scheme over $\mathbb{A}^1_S$ with a distinguished element $\lambda \in R$, corresponding to the image of the coordinate of $\mathbb{A}^1$.

We want to describe the functor $X_{Hod}(R\to \mathbb{A}^1_S)=X_{Hod,\lambda}$ as a quotient of $X_R$. Define

\[ T_\lambda= \text{Spec}\dfrac{R[\epsilon]}{\epsilon(\epsilon-\lambda)} \]
\[G(R)= \text{Hom} (T_\lambda, X_R) \]

When $\lambda=0$, $G$ parametrizes pairs of points of $X$ and tangent vectors at that point. Its formal PD completion parametrizes a pair consisting of a point and an infinitesimal tangent vector, the Dolbeault stack. For $\lambda=1$, $G$ parametrizes pairs of points of $X$. Its formal PD completion parametrizes pairs of infinitesimally close points, the de Rham stack. 

Define $i_0,i_1:\text{Spec}R \to T_\lambda$ by sending $\epsilon$ to $0$ and $\lambda$ respectively. That is, \\ $i_o(a+b\epsilon)=a, i_1(a+b\epsilon)=a+b\lambda$. Note $i_0-i_1$ lands in the ideal generated by $\lambda$. $T_\lambda$ comes with a distinguished map $j$ to $\text{Spec}R$. Let

\[ B = \dfrac{R[\epsilon_1]}{\epsilon_1(\epsilon_1-\lambda)} \times_{i_0,R,i_1} \dfrac{R[\epsilon_2]} {\epsilon_2(\epsilon_2-\lambda)}  \]

$B$ consists of pairs $(a+\epsilon_1 b, c+\epsilon_2 d)$ where $a=c+\lambda d$. There exists a map of schemes $\mu: T_\lambda \to \text{Spec}B$ given by $(a+\epsilon_1 b, c+\epsilon_2 d) \mapsto c+ \epsilon(b+d)$. This is a simple matter to verify. Note this can be expressed independently of $\lambda$.

There is a commutative diagram

\begin{center}

\begin{tikzcd}

          &   \Spec B  &                 \\
 \Spec R  \arrow[ur, "i_{1}"]   \arrow[dr, "i_1", swap] &             &    \Spec  R    \arrow[ul,"i_0", swap]     \arrow[dl,"i_0"]   \\ 
          &   T_{\lambda}      \arrow[uu,"\mu"]       &

\end{tikzcd}  

\end{center}

Note $\text{Spec}B$ is the gluing of two copies of $T_\lambda$ along $i_0$ and $i_1$. 
\thispagestyle{plain}

$\mu$ allows us to multiply two maps in $\text{Hom} (T_\lambda, X_R)$ that agree along $\text{Hom} (\text{Spec}R, X_R)$ and thus provides us with a groupoid structure with identity morphism from $X(R)$ induced by $j$. 

Now suppose $\text{Spec}A$ maps to $G$ for some algebra $A$. This corresponds to a map $\text{Spec}A \times T_\lambda \to X$. Assume $X$ is an affine scheme $\text{Spec}D$; then this data is equivalent to a map $D \to \dfrac{A[\epsilon]}{\epsilon(\epsilon-\lambda)}$. Note the two images under $i_0$ and $i_1$ differ by a multiple of $\lambda$.

Now we give the blowup description of $G$. Using the maps $i_0,i_1$ and $\lambda$ we can map $G$ to $X \times X \times \mathbb{A}^1$. Consider the closed subscheme $\Delta \times 0$ where $\Delta$ is the diagonal of $X \times X$. Its ideal is generated by $(x \otimes 1 - 1 \otimes x, \lambda)$. By the above note, this closed subscheme pulls back to $G$ to a closed subscheme generated just by $\lambda$, because $i_0(a+b\epsilon)-i_1(a+b\epsilon)$ lies in the ideal generated by $\lambda$. This means that the pullback of this closed subscheme is a divisor of $G$. By the universal property of the blowup, $G\to X \times X \times \mathbb{A}^1$ lifts to the blowup. When $\lambda=0$, $i_0$ and $i_1$ coincide so the image must not intersect the strict transform of $X \times X \times 0$. 

\thispagestyle{plain}

This gives a morphism of $G$ to the scheme remaining, namely the result of blowing up $X \times X \times \mathbb{A}^1$ and removing the strict transform of $X \times X \times 0$. This can be seen to be an isomorphism by comparing its fibres and invoking flatness. When $\lambda \neq 0,$ the $\lambda$-fibres are both copies of $X \times X$. When $\lambda=0$, a local calculation shows both $0$-fibres are the zero section of the tangent bundle of $X.$ 

As discussed above, I must then complete $G$. That is, we must take the PD completion of this groupoid along its identity, or $X \times \mathbb{A}^1$ along the diagonal map. That is, we PD complete along the closed subscheme $\Delta \times \mathbb{A}^1$. 

This deforms between the de Rham and Dolbeault groupoid at $\lambda=1$ and $\lambda=0$ \\ respectively. Moreover, the group structure induced by $\mu$ is compatible with the two group structures.
\end{proof}
\thispagestyle{plain}

\subsection{Non-abelian filtrations and the Kodaira-Spencer map}
\begin{definition}[Rees module]
Let $V$ be a vector bundle on $X$ with decreasing exhaustive filtration $F^\bullet$. Define $\zeta(V,F^\bullet)$, a $\mathbb{G}_m$-equivariant sheaf of vector bundles over $\mathbb{A}^1_X$ as follows:

let $j: \mathbb{G}_m \to \mathbb{A}^1$ be the inclusion and $t$ the coordinate on $\mathbb{A}^1$. Define $\zeta(V,F^\bullet)$ to be the subsheaf of $j_*(V \otimes \mathcal{O}_{\mathbb{G}_m})$ generated by sections of the form $t^{-n}v_n$ for $v_n \in F^n$.
\end{definition}

That is, $\zeta(V,F^\bullet)=\{ v_0 + t^{-1} v_1 + \ldots \}$. The fibre $\zeta_1$ returns $V$ while the fibre $\zeta_0$ gives the associated graded of the filtration. Note $\zeta$ is flat over $\mathbb{A}^1$. 

Conversely, given a vector bundle $W$ on $\mathbb{A}^1$ with $\mathbb{G}_m$-action, set $V=W_1$ and then obtain a decreasing filtration on $V$ by defining $F^n$ to be the set of $\mathbb{G}_m$-invariant sections with pole of order at least $n$ at $0$. These constructions are inverses, and provide an equivalence of categories

\[(\text{vector bundle with filtration}) \simeq (\mathbb{G}_m -\text{equivariant sheaves on } \mathbb{A}^1 )  \]

This motivates the definition of a filtration in the non-abelian setting. Let $M$ be a prestack or stack over $(Sch/S).$ By a \textit{filtration} of $M$ we mean a $\mathbb{G}_m$-equivariant prestack or stack $\mathcal{F}$ flat over $\mathbb{A}^1_S$, together with a given isomorphism $\mathcal{F}_1 \simeq M$. We impose the flatness condition, for otherwise one could add anything over the $t=0$ fibre and preserve the definition. The \textit{associated graded} of the filtration is the central fibre $\mathcal{F}_0$. 

Since $\lambda$-connections can be rescaled for $\lambda \neq 0$, it immediately follows $M_{Hod}$ is a filtration of $M_{dR}$ with associated graded $M_{Dol}$. We call $M_{Hod}$ the Hodge filtration of $M_{dR}$. Flatness of $M_{Hod}$ over $\mathbb{A}^1$ is proven in \cite{Simpson3}.

\thispagestyle{plain}
\begin{prop}
Let $W$ be a vector bundle over $X \times \mathbb{A}^1_S$ with a $\mathbb{G}_m$-action and a $\mathbb{G}_m$-equivariant descent to $[(X/S)]_{Hod}$. Then from $W$, we can associate a Higgs field as well as a vector bundle with integrable connection and filtration $(V,\nabla, F^\bullet)$, satisfying Griffiths transversality, that is, $\nabla(F^n) \subset F^{n-1} \otimes \Omega_{X/S}$.

\end{prop}
\begin{proof}
First, invert the Rees module construction to produce a vector bundle $V$ over the $t=1$ fibre, with decreasing exhaustive filtration $F^\bullet$. We are given $W$ has a descent to $[(X/S)]_{Hod}$. Specialize to the $t=1$ fibre to deduce $V$ has a descent to $[(X/S)]_{dR}$, hence an integrable connection $\nabla$. By the $\mathbb{G}_m$-equivariance, this connection extends over all of $ V \times \mathbb{G}_m$.

Because $W$ has a descent to the Hodge groupoid, proposition 6.23 tells us that there is a $\lambda$-connection $\tilde{\nabla}$ on all of $W$. 

Now restrict to $ V \times \mathbb{G}_m$. Over $X \times \mathbb{G}_m$, $[(X/S)]_{Hod}$ is simply a product \\ $[(X/S)]_{dR} \times \mathbb{G}_m$ so the restriction of $\tilde{\nabla}$ over $\mathbb{G}_m$ must coincide with the rescaling of $\nabla$. That is, $\tilde{\nabla}=t \nabla$ away from the $t=0$ fibre. So $t \nabla$, a priori defined only away from $0$, has an extension to a $\lambda$-connection on all of $W$. In particular, $\tilde{\nabla}_0$ is the associated Higgs field. Now $W=\zeta(V,F^\bullet)$, so we must have

\[ \tilde{\nabla}(\sum t^{-n}F^n ) \subset \sum t^{-n} F^n \otimes \Omega_{X \times \mathbb{A}^1/\mathbb{A}_S^1 }= \sum t^{-n} F^n \otimes \Omega_{X/S}[t] \]

Since $t$ is an indeterminate here, this implies $t\nabla(\sum t^{-n}F^n) \subset \sum t^{-n} F^n \otimes \Omega_{X/S}[t]$, which rearranges to $\nabla(F^n) \subset F^{n-1} \otimes \Omega_{X/S}$.

\end{proof}

\begin{corollary}
Under the above assumptions, the following are equivalent.
\begin{enumerate}
\item The connection $\nabla$ preserves the filtration, $\nabla(F^n) \subset F^{n} \otimes \Omega_{X/S}$.
\item $\nabla$ extends over all of $\mathbb{A}^1$.
\item The associated Higgs field vanishes.
\item $W$ descends further along the blowdown morphism $(X/S)_{Hod} \to (X/S)_{dR} \times \mathbb{A}^1_S.$

\end{enumerate}
\end{corollary}
\begin{proof}
$(1 \implies 2)$. If $\nabla$ preserves the filtration then $\nabla$ passes to the associated graded, and hence the $0$-fibre of the Rees module, so extends over all of $\mathbb{A}^1$.

$(2 \implies 3)$. Under this assumption, $\tilde{\nabla}$ and $t\nabla$ are both $\lambda$-connections defined over all of $\mathbb{A}^1$ that agree on $\mathbb{G}_m$. By flatness they must agree everywhere, and hence over $0$. So $\tilde{\nabla}_0=(t\nabla)_0=0$. 

\thispagestyle{plain}

$(3 \implies 1)$. $\tilde{\nabla}$ is a $\mathcal{O}_S[t]$-linear map $\tilde{\nabla}:M_1 \to M_2$ where $M_1$ is the Rees module $\zeta$ and $M_2=\zeta \otimes \Omega_{X/S}$. Now the associated Higgs field is the induced map $\frac{M_1}{tM_1} \to \frac{M_2}{tM_2}$, and we are given it vanishes. So $\tilde{\nabla}(M_1) \subset t M_2$. Since $t$ is not torsion, we have a well-defined morphism $\frac{\tilde{\nabla}}{t}$. This is an extension of $\nabla$. Moreover, the condition

\[\tilde{\nabla}(\sum t^{-n}F^n) \subset t(\sum t^{-n} F^n \otimes \Omega_{X/S}[t])\]

rearranges to $\nabla(F^n) \subset F^{n} \otimes \Omega_{X/S}$, so $\nabla$ preserves the filtration. 

$(4 \implies 3)$. Assume $W$ descends further along the blowdown morphism. Now the $\lambda$-connection structure $\tilde{\nabla}$ is induced from the pullback from $(X/S)_{Hod}$. But $W$ has a descent to $(X/S)_{dR} \times \mathbb{A}^1_S$, hence a connection on all of $W$. Hence $\tilde{\nabla}$ is induced by a rescaled connection, which implies 3. 

$(2 \implies 4)$ If $\nabla$ extends over all of $\mathbb{A}^1$, this provides a descent from $W$ all the way to $(X/S)_{dR} \times \mathbb{A}^1_S$. Since $\nabla$ scales to $\tilde{\nabla}$, it follows that this descent to $(X/S)_{dR} \times \mathbb{A}^1_S$ pulls back to the descent to $(X/S)_{Hod}$.

\end{proof}

\begin{theorem}
Assume the non-abelian Kodaira-Spencer map vanishes, that is, the stratification of $M_{Dol}$ over $[(S/T)]_{Dol}$ is trivial. Then the Gauss-Manin connection on $M_{dR}$ extends to the Hodge filtration.
\end{theorem}
\thispagestyle{plain}

\begin{proof}
By the remark following definition 6.19, $M_{Dol}=[R_{Dol}/GL_d]$. The condition that the Higgs field is trivial for $M_{Dol}$ is equivalent to the condition that the Higgs field is trivial for $R_{Dol}$. 

Now work locally on the scheme $R=R_{Hod}$. Recall $M_{Hod}=[R_{Hod}/GL_d]$, so the descents for $M_{Hod}$ and $R_{Hod}$ to $[(S/T)]_{Hod}$ are compatible. In particular, $R$ has a descent to $[(S/T)]_{Hod}$ in which the central fibre over $t=0$, $R_{Dol}$, has a trivial stratification. 

We may work locally on $\mathbb{A}^1_S$, equivalently on $S$. Choose an affine open $U \subset R$. The descent data to $[(S/T)]_{Hod}$ provides a filtration and integrable connection on $\mathcal{O}_U$ over $\mathbb{A}_S^1$ in which the central fibre has a trivial stratification. As such, $U$ descends along the blowdown morphism $[(S/T)]_{Hod} \to [(S/T)_{dR} \times \mathbb{A}^1_S$. These descents are compatible on intersections and so provide a descent of $R_{Hod}$ along the blowdown morphism $[(S/T)]_{Hod} \to [(S/T)_{dR} \times \mathbb{A}^1_S$. 

Finally, $M_{Hod}=[R_{Hod}/GL_d]$, and $GL_d$ automatically descends, so $M_{Hod}$ descends to $[(S/T)]_{dR} \times \mathbb{A}^1_S$. This descent induces the Gauss-Manin connection on the $t=1$ fibre, $M_{dR}$. Hence the Gauss-Manin connection extends over all of $\mathbb{A}^1_S$ to $M_{Hod}$.

\end{proof}

\thispagestyle{plain}

\pagebreak

\section{Characteristic $p$}
In this section, we specialize to the case where all schemes in question are \\ characteristic $p$. We shall use the same notation for local coordinates as in section 3.2, sometimes specializing to $n=1$ for notational simplicity. Recall $X/S$ is a smooth separated morphism and $\Lambda$ is the set of crystalline differential operators. I shall study the structure of $\Lambda$, $\mathcal{P}$ and its important quotient $\dfrac{\mathcal{P}}{I}$. Next, I will introduce the conjugate filtration on $\mathcal{P}$ and define the conjugate moduli stack, which will be required to state and prove the main theorem.

\thispagestyle{plain}

Later in this section, I will study two necessary constructions, namely splittings of Cartier, and the canonical connection.

\subsection{More on the structure of $\Lambda$ and $p$-curvature}

Recall multiplication in $\Lambda$ satisfies
\[\partial^n f = \sum_{0\leq k \leq n} {n \choose k} \partial^{n-k} (f) \partial^k \]
%In characteristic $p$ this sum simplifies to
%\[\partial^n f = \sum_{0\leq k \leq n, n-k<p} {n \choose k} \partial^{n-k} (f) \partial^k \]

\begin{definition}
Let $D$ be a derivation on $X$, namely a section of the tangent bundle, understood as an element of $\Lambda$. By \cite{Katz3} section 5.0, the $p$-th iterate $D^{\circ p}$ is itself a derivation. Define the $p$-curvature of $D$ to be $\psi(D)=D^p - D^{\circ p}$ where $D^p$ is the product in $\Lambda$ and $D^{\circ p}$ is the $p$-th iterate derivation.
\end{definition}

\begin{lemma}
$\psi$ is $p$-linear on derivations. $\psi(\partial)=\partial^p$. The images $\psi$ generate a commutative $\mathcal{O}_X$-subalgebra of $\Lambda$.
\end{lemma}

\begin{proof}
The fact that $\psi$ is $p$-linear is proven in \cite{Katz3}, proposition 5.2. It is easy to compute for $D=\partial$ that $D^{\circ p}=0$, so $\psi(\partial)=\partial^p$. By $p$-linearity,

\[\psi(f_1\frac{\partial}{\partial x_1} + \ldots f_n\frac{\partial}{\partial x_n}) = f_1^p (\frac{\partial}{\partial x_1})^p + \ldots f_n^p(\frac{\partial}{\partial x_n})^p   \]

Next, we claim $\partial^{pn} f = f\partial^{pn}$. From the multiplication relations in $\Lambda$, described in section 3.2, we have

\[\partial^{pn} f=\sum_{k \leq pn} {pn \choose k} \partial^k(f) \partial^{pn-k}\]
For $0<k<p, {pn \choose k} =0$ in characteristic $p$. For $k\geq p$, $\partial^k(f)=0$. So only the $k=0$ term remains, which is just $f\partial^{pn}$. That is, $\partial^{pn} f = f\partial^{pn}$.

So the $\mathcal{O}_X$-submodule of $\Lambda$ generated by $\partial^{pn}, n \geq 0$ is a commutative $\mathcal{O}_X$-subalgebra of $\Lambda$. By the identity $\psi(\partial)=\partial^p$, the $p$-curvatures commute, and generate this \\ subalgebra of $\Lambda$. Call this $\Lambda_c$. It is the largest commutative $\mathcal{O}_X$-subalgebra of $\Lambda$.
\end{proof}

Let $F$ be the absolute Frobenius and $F_X$ the relative Frobenius $X \to X'$. By the above lemma, $\psi$ induces a linear map

\[ F^*_X T_{X'/S}=F^* T_{X/S} \to \Lambda \]

whose image lies in the commutative subalgebra $\Lambda_c$. This induces a map on the symmetric algebra, and we consider two adjoint forms of it. First,
\[F^*_X S(T_{X'/S}) \to \Lambda. \]
This is an injection, and the image is $\Lambda_c$. In coordinates, it is spanned by $f (\frac{\partial}{\partial x_i})^p$.

\thispagestyle{plain}

Dually, 

\[S( T_{X'/S}) \to (F_X)_* \Lambda. \]

This is again an injection, and the image is $Z_X$, the centre of $\Lambda$. In coordinates, it is spanned by $f^p (\frac{\partial}{\partial x_i})^p$. 

Now consider the first map $F^*_X S(T_{X'/S}) \to \Lambda$. The source has a natural decreasing filtration of ideals $F^*_X S^{\geq i} T_{X'/S}$. 

This induces a decreasing filtration of ideals on $\Lambda$, which we label \\ $\Lambda \supset \Lambda^{\geq p} \supset \Lambda^{\geq 2p} \supset \ldots$ 

From the explicit description of coordinates, one can see $\Lambda^{\geq pk}$ is generated as an ideal of $\Lambda$ by elements $(\frac{\partial}{\partial x_i})^{pk}$.

\subsection{The $p$-curvature morphism $\mathcal{P}\to\dfrac{\mathcal{P}}{I}$}
I now examine the structure of $\dfrac{\mathcal{P}}{I}$, particularly in characteristic $p$. 

\thispagestyle{plain}

\begin{lemma}
Let $p$ be prime, $k\geq 1$ any integer, $0\leq r \leq p-1$. Then in any divided power ideal of any divided power algebra,

\[ (x^{[p]})^{[k]} = \dfrac{(kp)!}{p!^k k!} x^{[pk]} \]
\[ x^{[pk]}x^{[r]} = {kp+r \choose r} x^{[kp+r]} \]

Both of these coefficients are integers that are $1$ modulo $p$.
\end{lemma}
\begin{proof}
See \cite{Gros-LeStum} or \cite{Berthelot-Ogus}.
\end{proof}
\begin{lemma}
Let $S$ be a scheme over $\mathbb{F}_p$. Then in local coordinates, we can give explicit polynomial algebra descriptions of $\mathcal{P}$ and $\dfrac{\mathcal{P}}{I}$. Specifically, we can describe $\dfrac{\mathcal{P}}{I}$ itself as a divided power algebra on generators $\tau_i^{[p]}$.
\end{lemma}

\begin{remark}
If $S$ is a scheme over $\mathbb{Q}$, then $I=J$ and $\dfrac{\mathcal{P}}{I}=\mathcal{O}_X$.
\end{remark}

\begin{proof}
First, since we are in characteristic $p$, the first part of proposition 2.11 tells us that even the uncompleted object $\mathcal{P}$ has a simple form $\mathcal{O}_X <\tau_1,\ldots, \tau_n>$.

Using the divided power identities above, one can show by iteration that, if \\ $a_0,\ldots,a_k<p$  and $x \in J$, the divided power ideal, then 

\[x^{[a_0+\ldots + a_k p^k]}=x^{[a_0]}\ldots x^{[a_kp^k]} = x^{[a_0]}\ldots (x^{[p^k]})^{[a_k]}    \]

Moreover, for any $y \in J$, and $r<p$, $y^r$ and $y^{[r]}$ differ by a unit, while $y^p=p!y^{[p]}=0$. Thus $\mathcal{P}$ has an explicit presentation 

\[ \dfrac{\mathcal{O}_X [\tau_i, \tau_i^{[p]},\tau_i^ {[p^2]},\ldots]      }{ \tau_i^p=0, (\tau_i^{[p]})^p=0,(\tau_i^ {[p^2]})^p=0,\ldots }  \]

And $\dfrac{\mathcal{P}}{I}$ has an explicit presentation 

\[ \dfrac{\mathcal{O}_X [\tau_i^{[p]},\tau_i^ {[p^2]},\ldots]      }{ (\tau_i^{[p]})^p=0,(\tau_i^ {[p^2]})=0,\ldots }  \]

which is a divided power algebra with generators $\tau_i^{[p]}$. That is, in characteristic $p$,

\[ \dfrac{\mathcal{P}}{I} = \mathcal{O}_X <\tau_1^{[p]},\ldots,\tau_n^{[p]}>  \]

Without using coordinates, we can say that $\dfrac{\mathcal{P}}{I}$ is a divided power algebra with divided power ideal given by $\dfrac{J^{[p]} \cap I}{I}$.
\end{proof}

\thispagestyle{plain}

\begin{prop}
There is a canonical isomorphism $\dfrac{\mathcal{P}}{I} \xrightarrow{\sim} F_X^* \Gamma_{X'}\Omega_{X'/S}$. \\Composing with the canonical projection $\mathcal{P} \to \dfrac{\mathcal{P}}{I}$ and then taking the completion yields a canonical morphism $\hat{\mathcal{P}} \to F_X^* \hat{\Gamma}_{X'}\Omega_{X'/S}$. This is a coalgebra morphism.
\end{prop}
\begin{proof}
First, consider the map $i \mapsto i^{[p]}$, $I \to \mathcal{P}$.

Compose with the canonical projection to give a map $I \to \dfrac{\mathcal{P}}{I}$. \\

For $x,y \in I, a \in \mathcal{O}_X$ we have $(xy)^{[p]}=x^p y^{[p]}= p! x^{[p]}y^{[p]}=0$, and $(ax)^{[p]}=a^p x^{[p]}.$ Thus, the above map kills $I^2$ and induces a Frobenius-linear map \[\Omega_{X/S}=\dfrac{I}{I^2} \to \dfrac{\mathcal{P}}{I}\]

Recall the diagram

\begin{center}

\begin{tikzcd}

   X  \arrow[ddr] \arrow[drr,"F"]  \arrow[dr, "F_X"]   &    &                 \\
  &   X'      \arrow[d] \arrow[r,"\pi"]    &    X     \arrow[d]   \\ 
          &   S      \arrow[r,"F"]      &   S

\end{tikzcd}  

\end{center}

This induces a linear map $F^*\Omega_{X/S} \to \dfrac{\mathcal{P}}{I}$.

Now $F^*\Omega_{X/S}=F_X^* \pi^* \Omega_{X/S}=F_X^*\Omega_{X'/S}$. So this induces a map of $\mathcal{O}_X$-modules

\[ F_X^*\Omega_{X'/S} = F^*\Omega_{X/S}\to \dfrac{\mathcal{P}}{I} \]

 Passing to the PD envelope, this induces a map of $\mathcal{O}_X$-divided power algebras
 
 \thispagestyle{plain}

\[  \Gamma_X  F_X^*\Omega_{X'/S} \to \dfrac{\mathcal{P}}{I} \]

This map is canonical; we now use local coordinates to see it is an isomorphism. If $x_1,\ldots,x_n$ are coordinates on $X/S$ then these have pullbacks $x_1',\ldots,x_n'$ to $X'/S$. Therefore, $\Omega_{X'/S}$ is locally free on $dx_1',\ldots,dx_n'$ and $\Gamma_{X'}\Omega_{X'/S}$ is the divided \\ polynomial algebra $\mathcal{O}_{X'} <dx_1',\ldots,dx_n'>$. So the Frobenius pullback is simply $\mathcal{O}_{X} <dx_1',\ldots,dx_n'>$ and the $p$-curvature map is

\[\mathcal{O}_{X} <dx_1',\ldots,dx_n'>= F_X^*\Gamma_{X'}\Omega_{X'/S} \to \dfrac{\mathcal{P}}{I} = \mathcal{O}_X <\tau_1^{[p]},\ldots,\tau_n^{[p]}> \]

\[ dx_i' \mapsto \tau_i^{[p]}   \]

Viewing locally, we can now see that this is an isomorphism. In particular it gives a canonical grading on $\dfrac{\mathcal{P}}{I}$. 

Let $\theta_u: \mathcal{P} \to F_X^*\Gamma_{X'}\Omega_{X'/S}$ be the composition of the natural projection \\ $\mathcal{P}\to\dfrac{\mathcal{P}}{I}$ with this isomorphism above. $u$ stands for uncompleted. We have shown this is a surjective morphism of divided power algebras. Now we verify the coalgebra conditions, which shows this induces a map of groupoids. There are three coalgebra conditions to verify. We verify they hold before completion of the morphism objects.

First, we have a diagram for the identity morphisms $X\to N$  to commute for each groupoid. Let $N_1=\text{Spec}_X F^*_X \Gamma_{X'} \Omega_{X'/S}=\text{Spec}_X\dfrac{\mathcal{P}}{I}$, $N_2=\text{Spec}_X\mathcal{P}$. The necessary commutative diagram is

\begin{center}
\begin{tikzcd}

X \arrow[r,"id_X"] \arrow[d,"e_1"] &   X \arrow[d, "e_2"] \\

N_1 \arrow[r,"\theta_u"]           &   N_2

\end{tikzcd}
\end{center}

$e_1$ is induced by sending all of $\Omega_{X'/S}$ to zero. 

\thispagestyle{plain}

$e_2$ is induced from the diagonal $X\to X\times_S X$, that is, $(\mathcal{O}_X \otimes_{\mathcal{O}_S}  \mathcal{O}_X,I) \to (\mathcal{O}_X,0)$. By the universal property of the PD envelope, this induces a rather trivial map $(\mathcal{P},J)\to(\mathcal{O}_X,0)$. The diagram above then becomes

\begin{center}
\begin{tikzcd}

\mathcal{O}_X  &   \mathcal{O}_X \arrow[l,"id_X"]\\

\dfrac{\mathcal{P}}{I} \arrow[u]           &  \mathcal{P} \arrow[u] \arrow[l]

\end{tikzcd}
\end{center}

This diagram commutes as all generators of the diagonal ideal $\tau_i^{[p^k]}$ are killed by the right vertical arrow; their images in $\dfrac{\mathcal{P}}{I}$ are elements $dx_i'^{[p^k]}$ that are also killed. 

Next, there is a condition for the source and target maps of the groupoids.

\begin{center}
\begin{tikzcd}

N_1 \arrow[r,"\theta_u"] \arrow[d,"s=t"] &   N_2 \arrow[shift right=2]{d}{s} \arrow[shift left=2]{d}{t} \\

X   \arrow[r,"id"]     &   X 

\end{tikzcd}
\end{center}

In terms of algebras, this is 

\begin{center}
\begin{tikzcd}

\dfrac{\mathcal{P}}{I} &   \mathcal{P} \arrow[l,"\theta_u"] \\

\mathcal{O}_X  \arrow[u,"s=t"]       &   \mathcal{O}_X \arrow[shift right=2]{u}{s} \arrow[shift left=2]{u}{t} \arrow[l,"id"]  

\end{tikzcd}
\end{center}

The left vertical arrow is the canonical algebra map $\mathcal{O}_X \to F_X^*\Gamma_{X'} \Omega_{X'/S'}$. As for the right vertical arrows, we have been writing $\mathcal{P}=\mathcal{O}_X<\tau_1,\ldots,\tau_n>$, but really this $\mathcal{O}_X$-algebra structure implicitly uses the first projection. 

So the right vertical source map is the $\mathcal{O}_X$-algebra structure map, whereas the right vertical target map is $x_i \mapsto 1\otimes x_i = x_i \otimes 1 + \tau_i=x_i + \tau_i$. 

The difference in the source and target maps' action on $x_i$ is thus $\tau_i$. Critically, $\theta_u$ kills $\tau_i$ and therefore equalizes the source and target maps in the right vertical arrow to make the diagram commute. 

Finally, I address the comultiplication. For the Dolbeault groupoid, the \\ comultiplication is induced by $\Gamma \Omega \to \Gamma(\Omega \oplus \Omega) \simeq \Gamma \Omega \otimes \Gamma \Omega$. This can be written as $m\mapsto (m,m)$ or $dx_i \mapsto dx_i \otimes 1 + 1 \otimes dx_i$. 

\thispagestyle{plain}

For the de Rham groupoid, the comultiplication is given by 

\[\Delta(\tau^{[p^k]})= \sum_{i+j=p^k} \tau^{[i]}\otimes \tau^{[j]}\]

To see that $\theta_u$ respects the comultiplication, we have a diagram

\begin{center}
\begin{tikzcd}

\dfrac{\mathcal{P}}{I} \otimes \dfrac{\mathcal{P}}{I} &   \mathcal{P} \otimes_{2,\mathcal{O}_X,1} \mathcal{P}\arrow[l,"\theta_u\otimes \theta_u",swap] \\

\dfrac{\mathcal{P}}{I}  \arrow[u]        &   \mathcal{P}\arrow[l,"\theta_u"]\arrow[u, "\Delta"]

\end{tikzcd}
\end{center}

This commutes: $\tau_i$ goes around either way to be $0$. For $ k \geq 1, \tau^{[p^k]}$ maps as follows: while going up then left, 
\[\tau^{[p^k]}\mapsto \sum_{i+j=p^k} \tau^{[i]}\otimes \tau^{[j]} \mapsto \sum_{i+j=p^k, p|i,j} \tau^{[i]}\otimes \tau^{[j]} \]
which is the same as going left then up. 

After completing, all these coalgebra conditions still hold. Hence we have a \\ canonical coalgebra morphism $\theta: \hat{\mathcal{P}} \to F_X^* \hat{\Gamma}_{X'}\Omega_{X'/S}$.

\end{proof}

\begin{definition}
Define a new groupoid $(X/S)_{FDol}$ whose object object is $X$ and whose morphism object is the twisted PD formal completion of the tangent space, $F^*\hat{\Gamma}_{X}\Omega_{X/S}= F_X^* \hat{\Gamma}_{X'}\Omega_{X'/S}$. Then the above proposition proves there is a canonical embedding $(X/S)_{FDol} \hookrightarrow (X/S)_{dR}$, which is the identity on object objects and the above coalgebra morphism $\hat{\mathcal{P}} \to F_X^* \hat{\Gamma}_{X'}\Omega_{X'/S}$ on morphism objects. The coalgebra condition implies that this is a morphism of groupoids.
\end{definition}

\subsection{A new criterion for $p$-curvature vanishing}
At this point we have \\ considered $p$-curvature for both $\Lambda$ and $\mathcal{P}$. Now we use their duality.
\begin{lemma}
The morphism $\mathcal{P} \to \frac{\mathcal{P}}{I}$ is dual to the inclusion of $\Lambda_c$ in $\Lambda$.
\end{lemma}
\begin{proof}
$\frac{\mathcal{P}}{I} = \mathcal{O}_X <\tau^{[p]}>$. As a morphism of left $\mathcal{O}_X$-modules, the morphism $\mathcal{P} \to \frac{\mathcal{P}}{I}$ acts as identity on all basis elements $\tau^{[pn]}$ and annihilates all others.

Passing to the dual basis, this corresponds to the inclusion of left $\mathcal{O}_X$-modules of the span of $\partial^{pn}$ into the span of all $\partial^{n}$. This is precisely the inclusion of $\Lambda_c$ in $\Lambda$.

As an alternative proof, note that the map $\mathcal{P} \to \frac{\mathcal{P}}{I}$ is the largest quotient which equalizes the two projections. This corresponds to taking the largest commutative subalgebra of $\Lambda$.
\end{proof}
\thispagestyle{plain}

At this point, we can rephrase the condition of vanishing $p$-curvature in a beautiful way. Let $(V,\nabla)$ be a vector bundle with integrable connection in characteristic $p$. By proposition 3.9, this is equivalent to an isomorphism between two pullbacks 

\begin{center}
\begin{tikzcd}

V \arrow[d] & pr_1^* V \xrightarrow[\phi]{\sim} pr_2^* V \\

X & (X\times_S X)^\wedge_{PD} \arrow[l,"pr_1"', yshift=0.7ex] \arrow[l, "pr_2" yshift=-0.7ex]

\end{tikzcd}
\end{center}

It's a natural question to ask, what is the equalizer of $pr_1,pr_2$? And under what circumstances does the isomorphism $\phi$, between the two pullbacks, pull back further to give the identity on the equalizer? After all, this is the closest possible statement to the n\"aive and impossible statement that $\phi$ ``is" the identity.

\begin{prop}
The equalizer of $pr_1,pr_2$ is precisely the completion of $\dfrac{P}{I}$. \\ Moreover, $V$ has trivial $p$-curvature if and only if $\phi$ pulls back to the identity.
\end{prop}
\begin{proof}
Let $Z$ be the closed formal scheme of $\hat{\mathcal{P}}$ cut out by the ideal $I$. Since $I$ is the ideal of the diagonal, it is immediate that $Z$ is the equalizer of $pr_1,pr_2$. Pull back the isomorphism $\phi$ as below,
\begin{center}
\begin{tikzcd}

V \arrow[d] & pr_1^* V \xrightarrow[\phi]{\sim} pr_2^* V \arrow[d, xshift=0.7ex]\arrow[d, xshift=-0.7ex] & p^*V \simeq p^*V \arrow[d] \\

X & (X\times_S X)^\wedge_{PD} \arrow[l,"pr_1"', yshift=0.7ex] \arrow[l, "pr_2" yshift=-0.7ex] & Z \arrow[l]

\end{tikzcd}
\end{center}

The $p$-curvature of $V$ is trivial if and only if $\Lambda_c$ acts trivially on $V$. Dualizing, this is equivalent to the induced stratification over $\dfrac{P}{I}$ being trivial.

\end{proof}

\subsection{Twisted Higgs fields}
\thispagestyle{plain}

\begin{definition}
Let $u:X \to Y$ be a morphism between schemes smooth over a base. A module $M$ over $X$ is said to be a $u$-Higgs bundle if there is a morphism $\psi: M \to M \otimes u^* \Omega_Y$ with an appropriate integrality condition $\psi^2=0$. When $X=Y, u=id$ this returns the regular definition of Higgs bundle.
\end{definition}
Analogous to corollary 6.7, this data is equivalent to a structure of a \\ $u^* S(T_Y)$-module. Now $F_X^*S(T_{X'})=F^* S(T_X)$, so there is no distinction between an $F_X$-Higgs bundle or an $F$-Higgs bundle on $X$.

Higgs bundles and their twists carry two different notions of pullback.

First, let $M$ be a Higgs bundle $M$ on $Y.$ $S(T_Y)$ acts on $M$, so $u^* S(T_Y)$ acts on $u^*M$, thus $u^*M$ is naturally a $u$-Higgs bundle on $X$. Call this the pullback bundle.

Next, given a $u$-Higgs bundle $N$ on $X$, the morphism $S(T_X) \to u^* S(T_Y)$ allows us to restrict $N$ to get a Higgs bundle on $X$. Call this the restricted bundle. So there is a sequence of categories

(HB on $Y$) $\to$ ($u$-HB on $X$) $\to$ (HB on $X$).

Specialize to the case of $F_X: X \to X'$ and invoke Cartier's theorem (\cite{Katz3}, 5.1) to get a sequence 

(HB on $X'$) $\to$ ($F$-HB on $X$ with specified descent data to $X'$) $\to$ ($F$-HB on $X$) $\to$ (HB on $X$). In fact, the composition of these maps is zero, although we will not use this.

\subsection{The dual conjugate filtrations}

Recall the decreasing filtration of ideals on $\Lambda$, which we label $\Lambda \supset \Lambda^{\geq p} \supset \Lambda^{\geq 2p} \supset \ldots$

\thispagestyle{plain}

This induces a surjective sequence of quotient algebras

\[\ldots \to \frac{\Lambda}{\Lambda^{\geq 2p}} \to \frac{\Lambda}{\Lambda^{\geq p}}\]

These are all finite rank locally free left $\mathcal{O}_X$-modules. Under the duality described above, this corresponds to an injective sequence of subcoalgebras of $\mathcal{P}$, all finite rank locally free left $\mathcal{O}_X$-modules, 

\[\mathcal{P}^{< p} \subset \mathcal{P}^{< 2p} \subset \ldots  \]

In coordinates, these can be described as the span of elements $\tau^{[i]}$ for $i<p$, $i<2p$, and so on.

In \cite{Berthelot-French}, $\Lambda$ is denoted as $D_X$, while the ideal $\Lambda^{\geq p}$ is denoted $K_X$, and the completion of $\Lambda=D_X$ along $K_X$ is denoted $\hat{D}_X$. I will denote this by $\hat{\Lambda}$. By taking the inverse and direct limit of the above two sequences respectively, we see there is a duality of left $\mathcal{O}_X$-modules between $\hat{\Lambda}$ and $\mathcal{P}$.

\subsection{The conjugate moduli stack $M_{conj}$}
These conjugate filtrations on $\Lambda$ and $\mathcal{P}$ produce a Rees module

\[\xi(\Lambda)=\Lambda_{conj}=\Lambda + t^{-p}\Lambda^{\geq p} + t^{-2p} \Lambda^{\geq 2p} + \ldots\]
\[\Xi=\mathcal{P}_{conj}=\mathcal{P}^{<p} + t^p\mathcal{P}^{<2p} +t^{2p} \mathcal{P}^{<3p}+ \ldots \]

$\Lambda_{conj}$ locally has generators over $\mathcal{O}_X[t]$ given by $t^{-pk}\partial^{pk + r}$ for $0\leq r < p$, while $\Xi$ has generators $t^{pk} \tau ^{[pk+r]}=t^{pk} \tau ^{[pk]} \tau^{[r]}$. As before, there exists a duality between $\Lambda_{conj}$ and $\Xi$. $\Lambda_{conj}$ has an order filtration indexed by $n=pk+r$. This corresponds to a decreasing filtration of $\Xi$, with respect to which we can complete.

\begin{prop}
The generators and relations of $\xi(\Lambda)$ as an algebra over $\mathcal{O}_X[t]$ can be determined explicitly. Modules over $\xi(\Lambda)$ consist of $\mathcal{O}_X$-modules with a $F_X$-Higgs field $\psi$ and an integrable connection, such that $t^p\psi$ coincides with the $p$-curvature of that integrable connection.
\end{prop}

\thispagestyle{plain}

\begin{proof}
Work in local coordinates $x_1,\ldots,x_m$. We explicitly find local generators and relations for the Rees module. First, specify to $t=1$. $\Lambda$ is the free left $\mathcal{O}_X$-algebra with generators $\partial_i$ and relations:

\[\partial_i\partial_j=\partial_j\partial_i\]
\[\partial^n f = \sum_{0\leq j \leq n} {n \choose j} \partial^j (f) \partial^{n-j} \]

We can modify these relations a little. First, we can introduce formal generators $\tilde{\partial_i}$ and the relation $\tilde{\partial_i}=\partial_i^p$. These are central elements of $\Lambda$. Next, we can refine the second relation above, restricting to $n<p$. For if $n=pk+r$, the second relation above reduces to just 

\[\partial^{pk+r} f= \partial^r f \partial^{pk} = \sum_{0\leq j \leq r} {r \choose j} \partial^j (f) \partial^{pk+r-j} \]

which follows from the relation for $r$ and the centrality of $\tilde{\partial_i}$. Therefore, we can reexpress $\Lambda$ as the free left $\mathcal{O}_X$-algebra with generators $\partial_i$ and $\tilde{\partial_i}$ with relations,

\[  \partial^r f = \sum_{0\leq j \leq r} {r \choose j} \partial^j (f) \partial^{r-j}, 0<r<p \]
\[\partial_i\partial_j=\partial_j\partial_i\]
\[ \tilde{\partial_i}=\partial_i^p  \]
and $\tilde{\partial_i}$ are central.

Now there exists a well-defined map $F_X^* S(T_{X'}) \to \Lambda + t^{-p}\Lambda^{\geq p} + t^{-2p} \Lambda^{\geq 2p} + \ldots$ sending the generators $\frac{\partial}{\partial x'_j} \mapsto t^{-p} \tilde{\partial_i}$. This is also valid at $t=0$, where we take the image of $t^{-p} \tilde{\partial_i}$ in the associated graded. It follows that $\xi(\Lambda)$ is the left $\mathcal{O}_X[t]$-algebra with generators $\partial_i$ and $\tilde{\partial_i}$ and relations

\[  \partial^r f = \sum_{0\leq j \leq r} {r \choose j} \partial^j (f) \partial^{r-j}, 0<r<p \]
\[\partial_i\partial_j=\partial_j\partial_i\]
\[ t^p\tilde{\partial_i}=\partial_i^p  \]
and $t^p$ and $\tilde{\partial_i}$ are central elements.

\thispagestyle{plain}

It follows that a module over $\xi(\Lambda)$ is a module simultaneously over $F_X^* S(T_{X'})$ and $D_X$, with the relations that under the action, $t^p\tilde{\partial_i}=\partial_i^p$. By sections 7.1 and 7.4, this is precisely the statement that such a module carries a $F_X$-Higgs field and an integrable connection, such that the $p$-curvature of the connection is $t^p$ times the Higgs field. Recall that there is no distinction between $F_X$-Higgs and $F$-Higgs fields.

\end{proof}

\thispagestyle{plain}

\begin{definition}[Conjugate stack]
Define a moduli stack $M_{conj}(X/S)$ over $\mathbb{A}^1_S$. To a scheme $R \to \mathbb{A}^1_S$, it assigns the category of all tuples $(V, \nabla, \psi)$ such that $V$ is a vector bundle over $X \times_S R$, $\nabla$ is an integrable connection on $V$, $\psi$ is a $F$-Higgs field on $V$, and the $p$-curvature of $\nabla$ is precisely $\lambda^p \psi$, where $\lambda$ is the coordinate on $\mathbb{A}^1_S$.
\end{definition}

\begin{corollary}
All non-zero fibres of $M_{conj}$ over $\mathbb{A}^1_S$ are equivalent to $M_{dR}$ and there is a $\mathbb{G}_m$-equivariance between them.
\end{corollary}
\begin{proof}
Setting $\lambda=1$ returns $M_{dR}$. Away from $\lambda=0$, if $\nabla$ is an integrable connection with $p$-curvature $\lambda^p \psi$, then $\frac{\nabla}{\lambda}$ has $p$-curvature $\psi$.
\end{proof}

\begin{corollary}
The central fibre $M_{conj,0}$ consists of $F$-Higgs bundles on $X$ with a specified descent data to $X'$.
\end{corollary}
\begin{proof}
At $\lambda=0$, the condition on elements of $M_{conj}$ implies $\nabla$ has trivial $p$-curvature and provides no condition on the $F$-Higgs field. There is a one-to-one correspondence between vector bundles with integrable connection of trivial $p$-curvature, and vector bundles with descent to $X'$. An integrable connection with trivial $p$-curvature is the canonical connection relative to some descent to $X'$, as proven by Katz and Cartier in \cite{Katz3}, theorem 5.1.
\end{proof}

\begin{remark}
The associated graded of the conjugate filtration of $\Lambda$ is 
\[F_X^* S(T_{X'/S}) \otimes \frac{D_X}{K_X} = F_X^* S(T_{X'/S}) \otimes \text{End}_{\mathcal{O}_{X'}} \mathcal{O}_X \]
with this equality due to \cite{Berthelot-French}. $M_{conj,0}$ consists of modules over this.
\end{remark}

\thispagestyle{plain}

\begin{definition}[Conjugate groupoid]
There is a groupoid $(X/S)_{conj}$ with object object $X \times \mathbb{A}^1$ and morphism object $\hat{\Xi}$. Its $t=1$ fibre is $(X/S)_{dR}$. Modules over $(X/S)_{conj}$ are those over $\Lambda + t^{-p}\Lambda^{\geq p} + t^{-2p} \Lambda^{\geq 2p} + \ldots$
\end{definition}

\begin{lemma}
$(.)_{conj}$ is a groupoid arrow functor.
\end{lemma}
\begin{proof}
There are three properties to verify. The first, that $(X/X)_{conj}$ is a trivial groupoid, is immediate. The second two properties require one to show that two natural maps are isomorphisms: if $X \to S \to T$ are smooth separated and $S'/S$ arbitrary then

\[(X\times_S S'/S')_{conj} \xrightarrow{\sim} (X/S)_{conj} \times_S S' \]

\[(X/S)_{conj} \xrightarrow{\sim}  (X/T)_{conj} \times_{(S/T)_{conj}} (S/S)_{conj} \]

For object objects these are immediate. For morphism objects, the Rees module construction is flat over $\mathbb{A}^1$ so it suffices to show this fibrewise. For $t\neq 0$ fibres, these isomorphisms follow from the $\mathbb{G}_m$-equivariance and the fact that the de Rham groupoid is a groupoid arrow functor. So all that remains is the $t=0$ fibre.

At $t=0$ the morphism object is simply 
\thispagestyle{plain}

\[ \dfrac{\mathcal{P}}{I} \otimes \frac{D_X}{K_X} \]

by the self-duality of $\frac{D_X}{K_X}=\text{End}_{\mathcal{O}_{X'}} \mathcal{O}_X $. Using local coordinates, this morphism object becomes

\[  \dfrac{\mathcal{O}_X [\tau_i]}{ \tau_i^p=0} \otimes \mathcal{O}_X <\tau_i^{[p]}>  \]

Then the base change property is immediate. For the descent property, assume $X/S$ has diagonal coordinates $\tau_i$ while $S/T$ has diagonal coordinates $\sigma_j$. Then one finishes by noticing

\[  \big( \dfrac{\mathcal{O}_X [\tau_i,\sigma_j]}{ \tau_i^p=0=\sigma_j^p} \otimes \mathcal{O}_X <\tau_i^{[p]},\sigma_j^{[p]}> \big) \otimes_{  \dfrac{\mathcal{O}_S [\sigma_j]}{\sigma_j^p=0} \otimes \mathcal{O}_S <\sigma_j^{[p]}>  } \mathcal{O}_S \simeq \dfrac{\mathcal{O}_X [\tau_i]}{ \tau_i^p=0} \otimes \mathcal{O}_X <\tau_i^{[p]}>  \]

Simply put, the tensor product on the left hand side kills all $\sigma_j$ and its higher powers, and returns the right hand side. That shows the descent isomorphism for the $t=0$ fibre and finishes the proof. Note that the morphism exists canonically, while the coordinates show that it is an isomorphism.
\thispagestyle{plain}

\end{proof}

\begin{corollary}
$M_{conj}(X/S)$ is an algebraic stack over $\mathbb{A}^1_S$, and has a descent to $[(S/T)]_{conj}$.
\end{corollary}
\begin{proof}
Once again, proposition 5.11 gives the algebraicity, and then proposition 5.12 gives the descent.
\end{proof}

\begin{prop}
The stack of modules over this groupoid $M(X/S)_{conj}$ coincides with $M_{conj}(X/S)$ defined above.
\end{prop}
\begin{proof}
See 7.12 for the definition of the elements of $M_{conj}$ as data $(V,\nabla,\psi)$. By proposition 7.11, this coincides with modules over an appropriate base change of $\Lambda_{conj}$. By the duality between $\Lambda_{conj}$ and $\hat{\Xi}=\hat{\mathcal{P}}_{conj}$, and the same proof as propositions 3.9 or 6.23, modules over $\Lambda_{conj}$ correspond to stratifications over $\hat{\Xi}$.
\end{proof}

\thispagestyle{plain}

\begin{remark}
Informally speaking, the modules over this groupoid are those \\ deforming between modules over $D_X$, and modules over its associated graded of the conjugate filtration.
\end{remark}

\subsection{Splittings of Cartier}
In this section I will introduce a technical construction used in the proof of the main theorem.

First, we follow \cite{Ogus-Vologodsky} and \cite{Katz3} in their presentation. Consider the Cartier \\ isomorphism, $C^{-1}: \Omega_{X'/S} \simeq \mathcal{H}^1({F_X}_*\Omega_{X/S})$, discussed in \cite{Katz3}. Reverse this and \\ compose with the natural map ${F_X}_* Z^1 \to \mathcal{H}^1({F_X}_*\Omega_{X/S})$ to get a map called the Cartier operator, $D: {F_X}_* Z^1 \to \Omega_{X'/S}$. Using the local freeness of the latter, we can find a section, $\Omega_{X'/S} \to {F_X}_* Z^1$. Now compose with the natural map \\ ${F_X}_* Z^1 \to {F_X}_*\Omega_{X/S}$. We get the following diagram:

\begin{center}
\begin{tikzcd}

& {F_X}_*\Omega_{X/S} & \\
\Omega_{X'/S} \arrow[ur,"\zeta"] \arrow[r,"\text{section}"] \arrow[rd,"C^{-1}"] & {F_X}_* Z^1 \arrow[d,"\text{natural}"] \arrow[u,"\text{natural}"] \arrow[r,"D"] & \Omega_{X'/S} \arrow[ld, "C^{-1}"] \\
& \mathcal{H}^1({F_X}_*\Omega_{X/S}) &

\end{tikzcd}
\end{center}

In particular, we can find such $\zeta$ for all affines. Call $\zeta$ a \textit{splitting of Cartier}.

\thispagestyle{plain}
The following is an alternative procedure. Recall relative Frobenius $F_X: X \to X'$. Suppose there exists a lift to characteristic $p^2$, denoted $\tilde{F}_X: \tilde{X} \to \tilde{X'}$. Let the ideals of the diagonal be $\tilde{I},\tilde{I'}$, and PD envelope $\tilde{\mathcal{P}}$. By an explicit calculation in \cite{Gros-LeStum}, there is a well-defined map

\[ \frac{1}{p!}\tilde{F}_X^*:  \tilde{I'} \to p \tilde{\mathcal{P}}  \simeq \frac{\tilde{\mathcal{P}}}{p\tilde{\mathcal{P}}} \simeq \mathcal{P}  \]

with values landing in $J$, the divided power ideal in $\mathcal{P}$. This induces a map

\[\Omega_{X'/S} \to \mathcal{P}\]

which lifts the map $\Omega_{X'/S} \to \dfrac{\mathcal{P}}{I}$ from the proof of proposition 7.6.

Since values land in $J$, one can also compose with the canonical projection to get a $p$-linear map \[\Omega_{X'/S} \to J \to \frac{J}{J^{[2]}}\simeq \Omega_{X/S} \]

This is a splitting of Cartier, $\zeta:\Omega_{X'/S} \to {F_X}_*\Omega_{X/S}$.

\subsection{The canonical connection}
Let $F_X: X \to X'$ be relative Frobenius, where $X'$ is the Frobenius twist relative to $S$. Let $E$ be a vector bundle over $X'$. Then the pullback $F_X^* E$ has a canonical connection

\[ \nabla_{can}: F_X^*E \to F_X^*E \otimes_{\mathcal{O}_X} \Omega_{X/S} \simeq E \otimes_{\mathcal{O}_{X'}} \Omega_{X/S}  \]

given by $\sigma \otimes s \mapsto \sigma \otimes ds$. Consider this with local coordinates. Let $e'_1,\ldots, e'_n$ be a local basis of $E$. Then $e_i=e'_i \otimes 1$ is a local basis of $F_X^*E$, and $\nabla(e_i)=0$. In particular, considering the corresponding action of $\Lambda$ on $F_X^*E$, we must have $\partial_j e_i = 0$ for all $i,j$.
\thispagestyle{plain}

\subsection{The horizontal subbundle in characteristic $p$}
Recall from section 3.4 that an integrable connection $\nabla$ on $V$ relative to $X/S$ corresponds to an integrable \\ subbundle $H$ of the tangent space $T_{W/S}$, where $W$ is the total space of $V$.
\begin{lemma}
Assume the integrability condition. Then $H$ is closed under $p$-th powers if and only if $\nabla$ has trivial $p$-curvature.
\end{lemma}

\thispagestyle{plain}

\begin{proof}
One direction is easy to see. If the $p$-curvature vanishes, then there exists a local frame of $V$ for which $\nabla e_i=0$. As discussed in proposition 3.10, this implies the splitting $L$ is trivial, and $H$ is given by the span of elements $\frac{\partial}{\partial x_j}$; these are closed under $p$-th powers.

Alternatively, consider the following morphism, called the $p$-curvature of $H$:

\[ H \to \frac{T_W}{H}\]
\[ D \mapsto D^p \]
This vanishes if and only if $H$ is preserved by $p$-th powers. This map is $p$-linear and defines a section of $F^* H^\vee \otimes \frac{T_W}{H}$. Since $H \simeq \pi^* T_{X/S}$ and $\frac{T_{W/S}}{H} \simeq \pi^* V$, this defines a section of $F^* \pi^* \Omega_{X/S} \otimes_{\mathcal{O}_V} \pi^*V = F^*  \Omega_{X/S} \otimes_{\mathcal{O}_X} \pi^*V$. By a similar calculation as section 3.4, and linearity, this defines a section of $F^*  \Omega_{X/S} \otimes_{\mathcal{O}_X} V^\vee \otimes V = F^* \Omega_{X/S} \otimes_{\mathcal{O}_X} \End(V)$. 

On the other hand, recall the $p$-curvature of $\nabla$ is a $p$-linear map \\ $\psi: T_{X/S} \to \End(V)$, that is, a section of $F^* \Omega_{X/S} \otimes_{\mathcal{O}_X} \End(V)$. Now invoke the \\ correspondence explicitly written down in section 3.4. Then the $p$-curvature \\ morphism of $H$ is given by the map 

\[ \pi^* T_{X/S} \to \pi^* V\]
\[ \frac{\partial}{\partial x_j} \mapsto (\nabla_{\frac{\partial}{\partial x_j}})^p \]

But $\psi(D) = (\nabla_D)^p - \nabla_{D^p} $. Specialize to $D= \frac{\partial}{\partial x_j}$. Then $D^p=0$, so \\ $\psi(\frac{\partial}{\partial x_j})= (\nabla_{\frac{\partial}{\partial x_j}})^p$. Now the $p$-curvature of $H$ vanishes if and only if $(\nabla_{\frac{\partial}{\partial x_j}})^p=0$ for all $j$. By $p$-linearity of $\psi$, these all vanish if and only if $\psi$ is identically zero.
\thispagestyle{plain}

\end{proof}
\thispagestyle{plain}

\pagebreak

\section{The main theorem}

Each of the subsequent sections contains a technical result necessary to state and prove the main theorem. 

\subsection{Bost's conjecture for the horizontal subbundle of $M_{dR}$}
\begin{prop}
$R_{dR}$ and $M_{dR}$ have canonical horizontal subbundles of their tangent bundles $T_{M_{dR}/T}$ and $T_{R_{dR}/T}$.
\end{prop}
\thispagestyle{plain}

\begin{proof}
$M_{dR}$ and $R_{dR}$ both have an integrable connection over $[(S/T)_{dR}]$. By taking affine open subsets of $S$ and $R_{dR}$, the data of the integrable connection on $R_{dR}$ yields integrable connections on all structure sheaves $\mathcal{O}_U$, for affine opens $U \subset R_{dR}$. By \\ section 3.4, this yields horizontal subbundles of $T_{U/T}$ for all $U$. These are \\ compatible, so yield a horizontal subbundle of $T_{R_{dR}/T}$. Finally, $M_{dR}=[R_{dR}/GL_d]$, so this horizontal subbundle descends to $M_{dR}$.
\end{proof}

Thus, we can state Bost's criterion for $M_{dR}$, namely, say the $p$-curvature of $M_{dR}$ is trivial if this horizontal subbundle is closed under $p$-th powers.

\subsection{An alternative definition of non-abelian $p$-curvature}
Let $S/T$ be a smooth separated scheme in characteristic $p$, $S'$ the Frobenius twist with respect to $T$, and $F_{S/T}: S \to S'$ the relative Frobenius with respect to $T$.

For $S/T$, consider the $p$-curvature morphism $\theta:(S/T)_{FDol}\to (S/T)_{dR}$ introduced in definition 7.7. Abuse notation and also use $\theta$ for the induced morphism on \\ quotient stacks. $(S/T)_{FDol}$ is the twisted Dolbeault groupoid, with object object $S$ and morphism object $F_S^* \hat{\Gamma}_{S'}\Omega_{S'/T}$.  Recall the Gauss-Manin connection on $M_{dR}(X/S)$ is the descent $M_{dR}(X/S) \to [(S/T)]_{dR}$. We can form the diagram

\pagebreak
\begin{center}
\begin{tikzcd}

\theta  ^* M_{dR}(X/S) \arrow[d] \arrow[r] &   M_{dR}(X/S) \arrow[d]  \\

[(S/T)]_{FDol} \arrow[r,"\theta"]           &  {[(S/T)]_{dR}}

\end{tikzcd}
\end{center}

As a second definition, declare the non-abelian $p$-curvature of the Gauss-Manin connection on $M_{dR}(X/S)$ to be trivial if the induced stratification of $\theta  ^* M_{dR}(X/S)$ over the twisted Dolbeault groupoid $(S/T)_{FDol}$ is trivial. Note this groupoid is one with source and target maps being equal so this condition makes sense. 

\begin{prop}
The two definitions of $p$-curvature being trivial are equivalent. That is, the horizontal subbundle is closed under $p$-th powers if and only if the pullback stratification induced by $\theta$ is trivial.
\end{prop}

\begin{proof}
Assume the horizontal tangent bundle of $R_{dR}$, equivalently $M_{dR}$, is preserved by $p$-th powers. Work locally on $R_{dR}$ as above. By taking affine open subsets of $S$ and $R_{dR}$, the data of the integrable connection on $R_{dR}$ yields integrable connections on all structure sheaves $\mathcal{O}_U$ for affine opens $U \subset R_{dR}$. By lemma 7.21, the induced integrable connection on $\mathcal{O}_U$ has trivial $p$-curvature; by proposition 7.9, the \\ corresponding stratification is equalized by $\theta$. So $\theta  ^* R_{dR}$ has trivial stratification over $[(S/T)]_{FDol}$. But $GL_d$ automatically descends, so this implies $\theta  ^* M_{dR}$ has trivial stratification over $[(S/T)]_{FDol}$. 

The converse is identical, applying proposition 7.9 then 7.21.
\end{proof}
\thispagestyle{plain}

\subsection{Extending the $p$-curvature morphism}
Recall the morphisms constructed in proposition 7.6 and definition 7.7. We have an embedding of groupoids and their quotient stacks $\theta:(S/T)_{FDol}\hookrightarrow (S/T)_{dR}$ given by the identity on object objects and $\hat{\mathcal{P}} \to F_S^* \hat{\Gamma}_{S'}\Omega_{S'/T}$ on morphism objects.

\begin{prop}
$\theta$ can be extended over all of $\mathbb{A}^1$ to a morphism \\ $\Theta:(S/T)_{FDol}\times \mathbb{A}^1_T \to (S/T)_{conj}$ such that $\Theta_1=\theta$.
\end{prop}
\begin{proof}

Consider the $p$-curvature morphism $\mathcal{P} \to \dfrac{\mathcal{P}}{I}$. 

On $\mathcal{P}$ we have the conjugate filtration $\mathcal{P}^{< p} \subset \mathcal{P}^{< 2p} \subset \ldots $

On $\dfrac{\mathcal{P}}{I}$ we have a filtration induced by the grading in the natural way. That is, write $\dfrac{\mathcal{P}}{I} = \bigoplus A_n$ where $A_n$ is locally generated by elements $\tau^{[pn]}$. Set $F_n = \bigoplus_{k< n} A_n$ and $F_0=0$. Then $F_0 \subset F_1 \subset \ldots$ is a filtration on $\dfrac{\mathcal{P}}{I}$.

In local coordinates, $\mathcal{P}^{< pr}$ is spanned by elements $\tau^{[i]}$ for $i<pr$, so $\mathcal{P}^{< pr} $ maps to $F_r$ under $\mathcal{P} \to \dfrac{\mathcal{P}}{I}$. That is, the $p$-curvature morphism respects the filtrations, and hence extends to a morphism between the two Rees modules $\Xi \to \zeta(\dfrac{\mathcal{P}}{I}, F^\bullet)$. Since $\dfrac{\mathcal{P}}{I}$ is graded, its Rees module is trivial over $\mathbb{A}^1$ and is simply a copy of $\dfrac{\mathcal{P}}{I} \times \mathbb{A}^1$. Now pass to the completion and take the identity morphism on objects to yield a groupoid morphism

\[ \Theta:(S/T)_{FDol}\times \mathbb{A}^1_T \to (S/T)_{conj}. \]

\end{proof}

\subsection{Key deformation lemma}
The following lemma will be used in the proof of the main theorem.

\begin{lemma}
Let $(E,\psi)$ be a Higgs bundle on $X/S$. Form the pullback by Frobenius $F^* E$ with its canonical connection $\nabla_{can}$. Assume there exists a splitting of Cartier $\zeta$, as introduced in section 7.7.

let $V=F^* E \times_S \mathbb{A}^1_S$ be the trivial extension, and equip this with an integrable connection $\nabla= \nabla_{can} + t \zeta (F^* \psi)$.

Then $(V,\nabla,F^*\psi)$ is an element of $M_{conj}$.
\end{lemma}
\begin{proof}

Recall from section 1.3 that absolute Frobenius factors through relative \\ Frobenius. By section 7.8, $F^*E$ has a canonical connection. Our construction is a variant of \cite{Lan-Sheng}, spread over $\mathbb{A}^1$. There, Lan-Sheng-Zuo prove that for $t=1$, this connection $\nabla_1= \nabla_{can} + \zeta (F^* \psi)$ is integrable with $p$-curvature equal to $F^*\psi$.
\thispagestyle{plain}

Work in local coordinates. If $e'_i$ is a local basis of $E$ and its pullback is $e_i$, the canonical connection annihilates these basis elements. Now consider the action of $\Lambda$ corresponding to this canonical connection: if $\partial_j$ are the generators of $\Lambda$, then they act by $\partial_j e_i = 0$.

Thus, corresponding to the connection $\nabla_1$ on the $t=1$ fibre, $\partial_j e$ acts as \\ $\zeta (F^* \psi) \wedge \partial_j (e_i)$. So $t\partial_j e$ acts as simply the multiple by $t$. To determine the \\ $p$-curvature, it suffices to simply take the $p$-th power $\partial_j^p$. By \cite{Lan-Sheng}, page 5, if $\partial_j e$ acts as $\zeta (F^* \psi) \wedge \partial_j (e_i)$, corresponding to $\nabla_1$, then $\partial^p_j e$ acts by $F^* \psi$. So corresponding to the connection $\nabla$ over $\mathbb{A}^1$, if $\partial_j e$ acts as $t\zeta (F^* \psi) \wedge \partial_j (e_i)$, then $\partial^p_j e$ acts by $t^pF^* \psi$. 

That is, $p$-linearity and the triviality of $\nabla_{can}$ shows that the $p$-curvature of $\nabla$ is $t^pF^* \psi$. So $(V,\nabla,F^*\psi)$ is an element of $M_{conj}.$

\thispagestyle{plain}

In particular, taking $t=0$, $F^* \psi$ lies in the central fibre of the conjugate groupoid and can be deformed away from $t=0$. 

\end{proof}

\subsection{Proof of main theorem}
\begin{theorem}[Main theorem]
Let $X/S$ be schemes in characteristic $p$. Assume both have lifts to characteristic $p^2$ and that relative Frobenius $F_X: X \to X'$ has a global lift to characteristic $p^2$. Then there is an identification between the $p$-curvature of $M_{dR}(X/S)$ after passing to the associated graded of its conjugate filtation, with the Frobenius twist of the Kodaira-Spencer map on $M_{Dol}$. Moreover, if the $p$-curvature vanishes in the sense of Bost's condition, then the Kodaira-Spencer map vanishes, and so the Gauss-Manin connection extends over the Hodge filtration.
\end{theorem}
\begin{remark}
The assumption that both have lifts to $p^2$ is of no consequence as \\ ultimately this result should be applied to $X/S$ which come from characteristic $0$, exactly like the context of \cite{Katz2}. The assumption that $F_X$ should have a lift is more restrictive, but is satisfied if $X/S$ is relatively affine. That detail aside, this is the exact non-abelian analogue of \cite{Katz2}, Theorem 3.2.

\end{remark}

\begin{proof}

\thispagestyle{plain}

Recall from section 8.2,

\begin{center}
\begin{tikzcd}

\theta  ^* M_{dR}(X/S) \arrow[d] \arrow[r] &   M_{dR}(X/S) \arrow[d]  \\

[(S/T)]_{FDol} \arrow[r,"\theta"]           &   {[(S/T)]_{dR}}

\end{tikzcd}
\end{center}

By proposition 8.3 and corollary 7.18, this extends to a diagram

\begin{center}
\begin{tikzcd}

\Theta  ^* M_{conj}(X/S) \arrow[d] \arrow[r] &   M_{conj}(X/S) \arrow[d]  \\

[(S/T)]_{FDol} \times \mathbb{A}^1_T \arrow[r,"\Theta"]           &   {[(S/T)]_{conj}}

\end{tikzcd}
\end{center}

%Now the associated graded of $M_{dR}$ with respect to the conjugate filtration is $M_{conj}(X/S)_0$ and therefore the induced $p$-curvature of $M_{dR}$ after passing to this associated graded is the central fibre $\Theta_0  ^* M_{conj}(X/S)_0$. Let us compare this with the Kodaira-Spencer map $M_{Dol}$. Inspired by Katz' theorem, we expect a Frobenius twist.

We now wish to identify the following two stratifications

\begin{center}
\begin{tikzcd}

\Theta_0  ^* M_{conj,0}(X/S) \arrow[d]  &   M_{Dol}(X/S) \arrow[d]  \\

[(S/T)]_{FDol}        &   {[(S/T)]_{Dol}}

\end{tikzcd}
\end{center}

The stratification of $\Theta_0  ^* M_{conj}(X/S)_0$ over $[(S/T)]_{FDol}$ is the precise non-abelian \\ analogue of the $p$-curvature of the Gauss-Manin connection after passing to the \\ associated graded of the conjugate filtration. The stratification of $ M_{Dol}(X/S) $ over $[(S/T)]_{Dol}$ is the precise analogue of the Kodaira-Spencer map. To compare these two stratifications, up to Frobenius, is the analogue of Katz' theorem (\cite{Katz2}, 3.2).

Given a Higgs bundle on $X$, or more generally on $X_R /R$ for an $S$-scheme $R$, pulling it back by absolute Frobenius lands in the central fibre $\Theta_0  ^* M_{conj}(X/S)_0$ by corollary 7.14, section 7.4, and the final line of lemma 8.4. Indeed, pulling back by absolute Frobenius factors through relative Frobenius, so the resulting pullback has a descent to $X'$, and also carries a $F_X$-Higgs field. That is, there is a diagram

\begin{center}
\begin{tikzcd}

 \Theta_0  ^* M_{conj}(X/S)_0  \arrow[d]  &   M_{Dol}(X/S) \arrow[d] \arrow[l,"F^*"]   \\

S          &   S\arrow[l,"F^*"] 

\end{tikzcd}
\end{center}
\thispagestyle{plain}

Note absolute Frobenius is not a map of $S$-schemes. Instead, it is natural, so the bottom map must be taken to be absolute Frobenius for this diagram to commute. 

Now we invoke the fact that both groupoids here are groupoid arrow functors, and use proposition 5.12. This proposition explicitly identifies the descent data of $M(X/S)$ as $M(X_{st}/S_{st})$. Again, since Frobenius is natural we have a commutative diagram, replacing $X/S$ by $X_{st}/S_{st}$.

\begin{center}
\begin{tikzcd}

 \Theta_0  ^* M_{conj}(X_{st}/S_{st})_0  \arrow[d]  &   M_{Dol}(X_{st}/S_{st}) \arrow[d] \arrow[l,"F^*"]   \\

S_{st}          &   S_{st}\arrow[l,"F^*"] 

\end{tikzcd}
\end{center}

It follows the former diagram is commutative when we include the data of the stratifications as well

\begin{center}
\begin{tikzcd}

 \Theta_0  ^* M_{conj}(X/S)_0  \arrow[d]  &   M_{Dol}(X/S) \arrow[d] \arrow[l,"F^*"]   \\

[(S/T)]_{FDol}          &   {[(S/T)]_{Dol}} \arrow[l,"F^*"] 

\end{tikzcd}
\end{center}

This diagram provides the desired identification between the $p$-curvature of $M_{dR}$ after passing to the conjugate filtration on the left, and the Frobenius twist of non-abelian Kodaira-Spencer on the right.

%Consider the map $\phi: (S/T)_{FDol} \to (S/T)_{Dol}$ which is absolute Frobenius on the object objects $S$ and $\Gamma_S \Omega_S \to F^* \Gamma_S \Omega_S, m \mapsto m \otimes 1$ on the morphism object. We can pullback

%\begin{center}
%\begin{tikzcd}

 %\phi^* M_{Dol}(X/S)   \arrow[d] \arrow[r] &   M_{Dol}(X/S) \arrow[d]  \\

 %(S/T)_{FDol} \arrow[r,"\phi"]           &   (S/T)_{Dol}

%\end{tikzcd}
%\end{center}

%Another way to describe the central fibre $M(X/S)_{conj,0}$ is to say it consists of vector bundles $E'$ over $X'$ with a morphism $E' \to E' \otimes_{x'} F_* F^* \Omega_{X'}$ where $F$ is relative Frobenius here. Alternatively, a module over $F_* F^* S(T_{X'})$. 

%So we have a diagram

%\begin{center}
%\begin{tikzcd}

%\Theta_0  ^* M_{conj}(X/S)_0 \arrow[dr] & \phi^* M_{Dol}(X/S)  \arrow[l,"F"] \arrow[d] \arrow[r] &   M_{Dol}(X/S) \arrow[d]  \\

% & (S/T)_{FDol} \arrow[r,"\phi"]           &   (S/T)_{Dol}

%\end{tikzcd}
%\end{center}

Now the result follows by combining the above steps. By proposition 8.1, $M_{dR}$ has a canonical integrable horizontal subbundle $H$ of its tangent bundle. Assume the $p$-curvature of the Gauss-Manin connection of $M_{dR}$ is trivial in the sense of Bost's condition, namely $H$ is preserved by $p$-th powers. By proposition 8.2,
$\theta^* M_{dR}$ has a trivial stratification over $[(S/T)]_{FDol}$. Extend $\theta$ to $\Theta$ using proposition 8.3 and use the $\mathbb{G}_m$-equivariance of $M_{conj}$ to deduce the stratification of $\Theta^* M_{conj}(X/S)$ is trivial above all of $\mathbb{G}_m$.

Now map $M_{Dol}$ to $\Theta_0^* M_{conj,0}$ as above, and let the image be $N_0$. We invoke the deformation in lemma 8.4. The assumption that $F_X$ has a lift of Frobenius provides a splitting of Cartier necessary for the lemma.

By the deformation in lemma 8.4, $N_0$ may deformed off the central fibre, and so lies in the closure of the non-zero fibres. Since the stratification is trivial away from $0$, deduce the stratification of $N_0$ over $[(S/T)]_{FDol}$ is trivial. Since everything takes place over the base scheme $T$, we may replace $F$ with relative Frobenius relative to $T$. Since $X/S/T$ are smooth, $F_T^*$ is injective. Deduce the stratification of $M_{Dol}$ above $[(S/T)]_{Dol}$ is trivial.

Finally, by theorem 6.33 it follows that the Gauss-Manin connection on $M_{dR}$ \\ extends over $M_{Hod}$.

\end{proof}
\thispagestyle{plain}

\pagebreak
% ----------------------------------------------------------------

\singlespacing

% ----------------------------------------------------------------

% the page number on the first page of the bibliography has to be
% centered at the bottom (!)

\thispagestyle{plain}

% the bibliography has to be single-spaced within items but
% double-spaced between items

\let\oldthebibliography=\thebibliography
\let\endoldthebibliography=\endthebibliography
\renewenvironment{thebibliography}[1]{
  \begin{oldthebibliography}{#1}
    \setlength{\itemsep}{3ex}
  }{
    \end{oldthebibliography}
  }

\bibliographystyle{abbrv}
\bibliography{biblio}

% ----------------------------------------------------------------

%%% Local Variables: 
%%% mode: latex
%%% TeX-master: "main"
%%% End: 

% ----------------------------------------------------------------

\end{document}